\documentclass[11pt]{amsart}
\usepackage[T1]{fontenc}
\usepackage{bbm}
\usepackage{amsthm}
\usepackage{amsmath, amssymb, amsfonts, verbatim, dsfont, graphicx}
\usepackage{hyperref}
\usepackage{color}
\usepackage{mathtools}
\usepackage{array}
\usepackage{tabu}
\usepackage{enumitem}

\makeatletter
\renewcommand\tableofcontents{%
    \section*{\huge{Table of Contents}
        \@mkboth{%
           \MakeUppercase\contentsname}{\MakeUppercase\contentsname}}
    \@starttoc{toc}%
    } 
\makeatother

\newsavebox\MBox

\newtheorem{theorem}{Theorem}[section]
\newtheorem{proposition}[theorem]{Proposition}

\newtheorem{lemma}[theorem]{Lemma}

\theoremstyle{definition} \newtheorem{remark}[theorem]{Remark}

\numberwithin{figure}{section}
\numberwithin{equation}{section}

\newcommand{\im}{\mbox{Im}}
\newcommand{\re}{\mbox{Re}}

\newcommand{\bR}{{\mathbb R}}
\newcommand{\bN}{{\mathbb N}}
\newcommand{\bC}{{\mathbb C}}

\newcommand{\bZ}{{\mathbb Z}}

\newcommand{\bP}{{\mathbb P}}

\newcommand{\cF}{{\mathcal{F}}}

\newcommand{\cI}{{\mathcal{I}}}

\newcommand{\be}{{\mathbf e}}
\newcommand{\ve}{{\varepsilon}}

\title[Almost sure local well-posedness and scattering for the 4D cubic NLS]{Almost sure local well-posedness and scattering for the 4D cubic nonlinear Schr\"odinger equation}

\author[B. Dodson]{Benjamin Dodson}
\address{Department of Mathematics \\ Johns Hopkins University \\ Krieger Hall 214 \\ 3400 N. Charles Street \\ Baltimore, MD 21218, USA}
\email{bdodson4@jhu.edu}

\author[J. L\"uhrmann]{Jonas L\"uhrmann}
\address{Department of Mathematics \\ Johns Hopkins University \\ Krieger Hall 219 \\ 3400 N. Charles Street \\ Baltimore, MD 21218, USA}
\email{luehrmann@math.jhu.edu}

\author[D. Mendelson]{Dana Mendelson}
\address{Department of Mathematics \\ University of Chicago \\ Eckhart 220 \\ 5734 S. University Ave \\ Chicago, IL 60637, USA}
\email{dana@math.uchicago.edu}

\thanks{The first author was supported in part by National Science Foundation grant DMS-1500424.}

\begin{document}

\begin{abstract}
 We consider the Cauchy problem for the defocusing cubic nonlinear Schr\"odinger equation in four space dimensions and establish almost sure local well-posedness and conditional almost sure scattering for random initial data in $H^s_x(\bR^4)$ with $\frac{1}{3} < s < 1$. The main ingredient in the proofs is the introduction of a functional framework for the study of the associated forced cubic nonlinear Schr\"odinger equation, which is inspired by certain function spaces used in the study of the Schr\"odinger maps problem, and is based on Strichartz spaces as well as variants of local smoothing, inhomogeneous local smoothing, and maximal function spaces. Additionally, we prove an almost sure scattering result for randomized radially symmetric initial data in $H^s_x(\bR^4)$ with $\frac{1}{2} < s < 1$.
\end{abstract}

\maketitle

\section{Introduction} \label{sec:introduction}
\setcounter{equation}{0}

We consider the Cauchy problem for the defocusing cubic nonlinear Schr\"odinger equation (NLS) in four space dimensions
\begin{equation} \label{equ:ivp_nls}
 \left\{ \begin{aligned} 
  (i \partial_t + \Delta) u &= |u|^2 u \text{ on } \bR \times \bR^4, \\
  u(0) &= f \in H^s_x(\bR^4).
 \end{aligned} \right.
\end{equation}
The equation \eqref{equ:ivp_nls} is invariant under the scaling
\begin{equation} \label{equ:scaling_cubic}
 u(t,x) \mapsto \lambda u(\lambda^2 t, \lambda x) \quad \text{for } \lambda > 0,
\end{equation}
and the scaling critical regularity $s_c = 1$ is, by definition, such that the corresponding homogeneous Sobolev norms of the initial data are left invariant by the scaling transformation~\eqref{equ:scaling_cubic}. Sufficiently smooth solutions to~\eqref{equ:ivp_nls} conserve the energy 
\begin{equation*} 
 E(u) =  \int_{\bR^4} \frac{1}{2} |\nabla u|^2 + \frac{1}{4} |u|^{4} \, dx.
\end{equation*}
Since this energy functional is also invariant under the scaling~\eqref{equ:scaling_cubic}, the Cauchy problem for~\eqref{equ:ivp_nls} is referred to as energy-critical.

\medskip 

The goal of this work is to investigate the local-in-time as well as the asymptotic behavior of solutions to~\eqref{equ:ivp_nls} for random initial data below the scaling critical regularity. Our main results establish almost sure local well-posedness and conditional almost sure scattering of solutions to~\eqref{equ:ivp_nls} with respect to a randomization of initial data in $H^s_x(\bR^4)$ with $\frac{1}{3} < s < 1$ that is based on a unit-scale decomposition of frequency space. Moreover, we prove that the unit-scale randomization of radially symmetric initial data in $H^s_x(\bR^4)$ with $\frac{1}{2} < s < 1$ almost surely leads to global-in-time scattering solutions to~\eqref{equ:ivp_nls}.

\medskip

The energy-critical defocusing nonlinear Schr\"odinger equation has been studied extensively over the past decades. Correspondingly, in what follows we only mention the most relevant results for this paper. For initial data above or at the scaling critical regularity, local solutions may be constructed using fixed point arguments based on Strichartz estimates, see for instance \cite{Ginibre_Velo_92_CMP, Cazenave_Weissler, Cazenave_Weissler2, Cazenave_book}. In particular, these results imply that any finite energy initial datum leads to a unique local solution to~\eqref{equ:ivp_nls}, and they also yield small data global well-posedness and scattering. Finite energy global well-posedness and scattering for the energy-critical defocusing NLS on $\mathbb{R}^3$ was established by Colliander-Keel-Staffilani-Takaoka-Tao~\cite{CKSTT08}, building upon the work of Bourgain~\cite{B99} in the radial case, while the analogous result for the defocusing cubic NLS~\eqref{equ:ivp_nls} on $\bR^4$ was obtained by Ryckman-Visan~\cite{Ryckman_Visan} and Visan~\cite{Visan}.

\medskip

Even though the nonlinear Schr\"odinger equation \eqref{equ:ivp_nls} is ill-posed below the scaling critical regularity $s_c = 1$, see for instance Christ-Colliander-Tao~\cite{CCT}, it is sometimes possible to construct unique local and even global solutions for suitably randomized initial data, and thereby conclude that large sets of initial data of scaling super-critical regularity do indeed lead to global solutions. This approach was initiated by Bourgain~\cite{B94, B96} for the periodic nonlinear Schr\"odinger equation in one and two space dimensions, building upon the constructions of invariant measures by Glimm-Jaffe~\cite{Glimm_Jaffe} and Lebowitz-Rose-Speer~\cite{LRS}, and by Burq-Tzvetkov~\cite{BT1, BT2} in the context of the cubic nonlinear wave equation on a three-dimensional compact Riemannian manifold. There has since been a vast and fascinating body of research, using probabilistic tools to study many nonlinear dispersive or hyperbolic equations in scaling super-critical regimes, see for example~\cite{Tz08, CO, NORS, D1, BT4, Suzzoni1, NPS, Nahmod_Staffilani, LM, Bourgain_Bulut14, BOP1, BOP2, Pocovnicu, DTV} and references therein. 

\medskip

In the following we restrict our overview to prior related probabilistic well-posedness results for the nonlinear Schr\"odinger equation for initial data randomized according to a unit-scale decomposition of frequency space as in~\eqref{equ:randomization} below. We emphasize though that the nonlinear Schr\"odinger equation on Euclidean space has also been considered in many other works relying on different randomizations, see for instance~\cite{BTT, D1, Thomann, Poiret1, Poiret2, Poiret_Robert_Thomann, Murphy}. In~\cite{BOP2, BOP1}, B\'enyi-Oh-Pocovnicu studied the probabilistic local well-posedness and conditional global well-posedness for the cubic NLS on $\bR^d$, $d \geq 3$, below the scaling critical regularity. They established almost sure local well-posedness and conditional almost sure global well-posedness results, where the latter in particular rely on an a priori hypothesis that a scaling critical Sobolev norm of the nonlinear component of the solutions does not blow up in finite time. We also refer to Brereton~\cite{Brereton} for analogous results for the defocusing quintic NLS. In the context of the cubic NLS on $\bR^3$, B\'enyi-Oh-Pocovnicu~\cite{BOP3} recently introduced an iterative procedure based on a partial power series expansion in terms of the free evolution of the random data, which allows to lower the regularity threshold for almost sure local well-posedness obtained in their previous work~\cite{BOP2}.

\medskip

Although in some of the aforementioned results on the nonlinear Schr\"odinger equation on Euclidean space, one obtains scattering with positive probability as a consequence of the probabilistic local theory, establishing almost sure scattering requires a more delicate argument, especially for energy-critical equations. The first almost sure scattering result for an energy-critical dispersive or hyperbolic equation with scaling super-critical random initial data was obtained recently by the authors in \cite{DLuM} for the defocusing cubic nonlinear wave equation on $\bR^4$ for randomized radially symmetric data. The proof in~\cite{DLuM} is based on the introduction of an approximate Morawetz estimate to the random data setting and new almost sure bounds for the free wave evolution of randomized radially symmetric data. The methods from \cite{DLuM} as well as from \cite{BOP2, Pocovnicu, OP} were subsequently further developed by Killip-Murphy-Visan~\cite{KMV} to obtain an analogous almost sure scattering result for the defocusing cubic NLS on $\bR^4$ for randomized radially symmetric initial data. Finally, we mention the recent work of Oh-Okamoto-Pocovnicu~\cite{OOP} establishing almost sure global well-posedness (without scattering) for the energy-critical defocusing NLS on $\bR^d$, $d = 5, 6$.

\subsection{Randomization procedure}

Before providing the precise statements of our main results, we introduce our randomization procedure for the initial data, which is based on a unit-scale decomposition of frequency space~\cite{ZF, LM, BOP2, BOP1}.

Let $\psi \in C_c^{\infty}(\bR^4)$ be an even, non-negative bump function with $\text{supp} \, (\psi) \subseteq B(0,1)$ and such that 
\[
 \sum_{k \in \bZ^4} \psi(\xi - k) = 1 \quad \text{for all } \xi \in \bR^4.
\]
Let $s \in \bR$ and let $f \in H^s_x(\bR^4)$. For every $k \in \bZ^4$, we define the function $P_k f \colon \bR^4 \rightarrow \bC$ by
\begin{equation} \label{equ:unit_scale_projection}
 (P_k f)(x) = \cF^{-1} \bigl( \psi(\xi - k) \hat{f}(\xi) \bigr)(x) \quad \text{for } x \in \bR^4.
\end{equation} 
We exploit that these Fourier projections satisfy a unit-scale Bernstein inequality, namely for all $1 \leq r_1 \leq r_2 \leq \infty$ and for all $k \in \bZ^4$ we have that 
\begin{align} \label{equ:unit_scale_bernstein}
  \|P_k f\|_{L^{r_2}_x(\bR^4)} \leq C(r_1, r_2) \|P_k f\|_{L^{r_1}_x(\bR^4)}
\end{align}
with a constant that is independent of $k \in \bZ^4$. 

\medskip

We let $\{ g_k \}_{k \in \bZ^4}$ be a sequence of zero-mean, complex-valued Gaussian random variables on a probability space $(\Omega, {\mathcal A}, \bP)$. Given a complex-valued function $f \in H^s_x(\bR^4)$ for some $s \in \bR$, we define its randomization by
\begin{equation} \label{equ:randomization}
 f^\omega :=  \sum_{k \in \bZ^4} g_k(\omega) P_k f.
\end{equation}
This quantity is understood as a Cauchy limit in $L^2_\omega\bigl(\Omega; H^s_x(\bR^4)\bigr)$, and in the sequel, we will restrict ourselves to a subset $\Sigma \subset \Omega$ with $\bP(\Sigma) = 1$ such that $f^\omega \in H^s_x(\bR^4)$ for every $\omega \in \Sigma$. 

Importantly, the randomization~\eqref{equ:randomization} almost surely does not regularize at the level of Sobolev spaces, see for instance \cite[Lemma B.1]{BT1}. However, the free Schr\"odinger evolution $e^{it\Delta} f^\omega$ of the random data does enjoy various types of significantly improved space-time integrability properties, see Section~\ref{sec:as_bounds_free_evolution}, which crucially enter the proofs of our main results. This phenomenon is akin to the classical results of Paley and Zygmund~\cite{Paley_Zygmund1} on the improved integrability of random Fourier series.

\begin{remark}
 One could also randomize with respect to a more general sequence of random variables $\{ g_k \}_{k\in\bZ^4}$ satisfying the following condition: there exists $c > 0$ so that the joint distributions $\{ \mu_{k} \}_{k\in\bZ^4}$ of the real and imaginary parts of the random variables $\{ g_k \}_{k\in\bZ^4}$ fulfill
 \begin{equation} \label{equ:rvassumption}
  \left| \int_{-\infty}^{+\infty} e^{\gamma x} \, d\mu_{k}(x) \right| \leq e^{c \gamma^2} \quad \text{for all } \gamma \in \bR \text{ and for all } k \in \bZ^4.
 \end{equation}
 The assumption~\eqref{equ:rvassumption} is satisfied, for example, by standard Gaussian random variables, standard Bernoulli random variables, or any random variables with compactly supported distributions. 
\end{remark}

\begin{remark} 
 In Theorem~\ref{thm:scattering_radial} we randomize radially symmetric functions. However, it should be noted that the unit-scale randomization~\eqref{equ:randomization} of a radially symmetric function is not radially symmetric.
\end{remark}

\subsection{Main results}

We are now ready to state our first main theorem on the almost sure local well-posedness of the cubic NLS in four space dimensions for scaling super-critical random data. 

\begin{theorem} \label{thm:as_local_wellposedness}
 Let $\frac{1}{3} < s < 1$. Let $f \in H^s_x(\bR^4)$ and denote by $f^\omega$ the randomization of $f$ as defined in~\eqref{equ:randomization}. Then for almost every $\omega \in \Omega$ there exists an open interval $I \ni 0$ and a unique solution 
 \[
  u(t) \in e^{it\Delta} f^\omega + C \bigl( I; \dot{H}^1_x(\bR^4) \bigr)
 \]
 to the cubic nonlinear Schr\"odinger equation 
 \begin{equation} \label{equ:cubic_nls_as_lwp_theorem}
  \left\{ \begin{aligned}
   (i\partial_t + \Delta) u &= \pm |u|^2 u \text{ on } I \times \bR^4, \\
   u(0) &= f^\omega.
  \end{aligned} \right.
 \end{equation}
\end{theorem}

\begin{remark} \label{rem:uniqueness}
 In the statement of Theorem~\ref{thm:as_local_wellposedness}, uniqueness holds in the sense that upon writing 
 \[
  u(t) = e^{it \Delta} f^\omega + v(t),
 \]
 there exists a unique local solution 
 \[
  v \in C \bigl( I; \dot{H}^1_x(\bR^4) \bigr) \cap X(I)
 \]
 to the forced cubic nonlinear Schr\"odinger equation
 \begin{equation*} 
  \left\{ \begin{aligned}
   (i \partial_t + \Delta) v &= \pm | e^{it\Delta} f^\omega + v |^2 (e^{it\Delta} f^\omega + v) \text{ on } I \times \bR^4, \\
   v(0) &= 0,
  \end{aligned} \right.
 \end{equation*}
 where the function space $X(I)$ is defined in Section~\ref{sec:functional_framework}.
\end{remark}

\begin{remark}
The length of the time interval $I$ in the statement of Theorem~\ref{thm:as_local_wellposedness} depends on the profile of the free evolution $e^{it\Delta} f^\omega$ of the random data in the sense that the time interval $I$ has to satisfy $\| e^{it\Delta} f^\omega\|_{Y(I)} \leq \delta$ for some small absolute constant $0 < \delta \ll 1$, where the function space $Y(I)$ is defined in Section \ref{sec:functional_framework}.
\end{remark}

We emphasize that prior to this work, almost sure local well-posedness for the cubic NLS~\eqref{equ:cubic_nls_as_lwp_theorem} in four space dimensions had been established by B\'enyi-Oh-Pocovnicu~\cite{BOP2, BOP1} for random initial data in $H^s_x(\bR^4)$ for the more restrictive range of regularities $\frac{3}{5} < s < 1$. The main tools in~\cite{BOP2, BOP1} are the improved almost sure space-time integrability of the free evolution of the random data and a bilinear refinement of the Strichartz estimate by Bourgain~\cite{B98} and Ozawa-Tsutsumi~\cite{Ozawa_Tsutsumi}.

\medskip 

The proof of Theorem~\ref{thm:as_local_wellposedness} proceeds by writing the solution to~\eqref{equ:cubic_nls_as_lwp_theorem} as a superposition of the free evolution of the random initial data and a nonlinear component
\[
 u(t) = e^{it\Delta} f^\omega + v(t).
\]
Then the nonlinear component $v(t)$ has to satisfy the following forced cubic NLS, where the forcing term inside the cubic nonlinearity is given by the free evolution of the (low-regularity) random initial data,
\begin{equation*} 
 (i \partial_t + \Delta) v = \pm | e^{it\Delta} f^\omega + v |^2 (e^{it\Delta} f^\omega + v)
\end{equation*}
with zero initial data $v(0) = 0$. Establishing local existence of solutions to this forced cubic NLS at energy regularity then has certain features in common with proving local well-posedness for a derivative nonlinear Schr\"odinger equation. In this spirit the main idea of our almost sure local well-posedness result in Theorem~\ref{thm:as_local_wellposedness} is to set up a suitable functional framework, whose precise definition is given in Section~\ref{sec:functional_framework}, based on Strichartz estimates and variants of the strong local smoothing $L^{\infty,2}_{\be}$, inhomogeneous local smoothing $L^{1,2}_{\be}$, and maximal function estimates $L^{2,\infty}_\be$ for the free Schr\"odinger evolution that have for instance played a key role in the study of the Schr\"odinger maps problem in~\cite{Ionescu-KenigI, Ionescu-KenigII, BIK, BIKT}. The main benefit of this new functional framework is that whenever the (low-regularity) free evolution $e^{it\Delta} f^\omega$ of the random data appears at highest frequency in the forced cubic nonlinearity, the local smoothing space component along with the inhomogeneous local smoothing and maximal function space components enable us to gain some derivatives. Instead, if the deterministic solution $v$ appears at highest frequency, the Strichartz components suffice for the nonlinear estimates. See the discussion before Proposition~\ref{prop:main_trilinear_estimates} for more details. Beyond the improved almost sure space-time integrability of the free evolution of the random data, the key ingredient for this scheme to work is an improved maximal function estimate for the free evolution of unit-scale frequency localized data, see Lemma~\ref{lem:maximal_function_unit_scale}, which implies an improved almost sure maximal function type estimate for the free evolution of the random initial data. We also refer to \cite{KPV93, Linares_Ponce93, KPV98, KPV00, KPV04, Ionescu_Kenig05, Ionescu_Kenig07, Bejenaru08, Bejenaru_Tataru08, CIKS08} and references therein for the many other uses of local smoothing estimates in the study of local and global well-posedness of derivative nonlinear Schr\"odinger and related equations.

\begin{remark}
 From the proof of Theorem~\ref{thm:as_local_wellposedness}, it is clear that our methods easily generalize to other space dimensions $d \geq 3$ and to other power-type nonlinearities. We also expect that our functional framework is compatible with the iterative procedure put forth in~\cite{BOP3} and that these ideas can be combined to further lower the regularity threshold. 
\end{remark}

\begin{remark}
We note that the main idea in our proof of Theorem~\ref{thm:as_local_wellposedness} does not apply to the periodic setting since there is no local smoothing effect for the Schr\"odinger equation on a compact domain such as the torus. The methods used to prove analogous almost sure local well-posedness results on the torus usually rely on random initial data with a specific form inspired by a typical element in the support of a certain Gibbs measure, and multilinear estimates which exploit properties of products of Gaussian random variables. For probabilistic well-posedness results for power-type NLS on the torus, we refer to \cite{B94, B96, CO, Nahmod_Staffilani, Yue} and references therein.
\end{remark}

\medskip 

Next we turn to the study of the long-time dynamics of solutions to the defocusing cubic NLS~\eqref{equ:ivp_nls} for scaling super-critical random initial data and establish a conditional scattering result for the associated forced defocusing cubic NLS
\begin{equation} \label{equ:forced_cubic_nls_scattering}
 \left\{ \begin{aligned}
          (i \partial_t + \Delta) v &= |F+v|^2 (F+v) \text{ on } \bR \times \bR^4, \\
          v(0) &= v_0 \in \dot{H}^1_x(\bR^4)
         \end{aligned} 
 \right.
\end{equation}
for forcing terms $F \colon \bR \times \bR^4 \to \bC$ satisfying $\|F\|_{Y(\bR)} < \infty$, where the precise definition of the function space $Y(\bR)$ is postponed to Section~\ref{sec:functional_framework}. Note that we will establish in Proposition~\ref{prop:as_bound_Y_norm} that $\| e^{it\Delta} f^\omega \|_{Y(\bR)} < \infty$ almost surely for any $f \in H^s_x(\bR^4)$ with $\frac{1}{3} < s < 1$. The next theorem asserts that if the maximal lifespan solution to \eqref{equ:forced_cubic_nls_scattering} satisfies a uniform-in-time a priori energy bound, then it must exist globally in time and scatter. 

\begin{theorem} \label{thm:scattering_conditional}
 Let $v_0 \in \dot{H}^1_x(\bR^4)$ and let $F \in Y(\bR)$. Let $v(t)$ be a solution to \eqref{equ:forced_cubic_nls_scattering} defined on its maximal time interval of existence $I_\ast$. Suppose in addition that
 \begin{equation} \label{equ:energy_hypothesis}
  M := \sup_{t \in I_\ast} \, E(v(t)) < \infty,
 \end{equation}
 where
 \[
  E(v(t)) = \int_{\bR^4} \frac{1}{2} |\nabla v(t)|^2 + \frac{1}{4} |v(t)|^4 \, dx.
 \]
 Then $I_\ast = \bR$, that is $v(t)$ is globally defined, and it holds that
 \begin{equation} \label{equ:conditional_spacetime_bound}
  \|v\|_{X(\bR)} \leq C\bigl( M, \|F\|_{Y(\bR)} \bigr),
 \end{equation}
 where the function spaces $X(\bR)$ and $Y(\bR)$ are defined in Section~\ref{sec:functional_framework}. In particular, the solution $v(t)$ scatters in the sense that there exist states $v^{\pm} \in \dot{H}^1_x(\bR^4)$ such that
 \[
  \lim_{t \to \pm \infty} \, \bigl\|  v(t) - e^{it\Delta} v^{\pm} \bigr\|_{\dot{H}^1_x(\bR^4)} = 0.
 \]
\end{theorem}

The proof of Theorem~\ref{thm:scattering_conditional} follows the idea of the proof of an analogous conditional scattering result by the authors~\cite[Theorem 1.3]{DLuM} for the forced defocusing cubic nonlinear wave equation. The main ingredients are the a priori bounds for the ``usual'' defocusing cubic NLS on $\bR^4$ from the work of Ryckman-Visan~\cite{Ryckman_Visan} and Visan~\cite{Visan} as well as the development of a suitable perturbation theory within our functional framework for the forced defocusing cubic NLS~\eqref{equ:forced_cubic_nls_scattering}, see Lemma~\ref{lem:long_time_perturbations}.

\medskip 

Furthermore, we establish the following almost sure scattering result for the defocusing cubic NLS~\eqref{equ:ivp_nls} for randomized radially symmetric initial data. 

\begin{theorem} \label{thm:scattering_radial}
Let $\frac{1}{2} < s < 1$ and let $f \in H^s_x(\bR^4)$ be radially symmetric. Let $f^\omega$ be the randomized initial data defined in~\eqref{equ:randomization}. Then for almost every $\omega \in \Omega$, there exists a unique global solution 
 \begin{equation} 
  u(t) \in e^{i t \Delta} f^\omega + C\bigl(\bR; \dot{H}^1_x(\bR^4) \bigr)
 \end{equation}
 to the defocusing cubic nonlinear Schr\"odinger equation
 \begin{equation} \label{equ:cubic_nls_radial_scattering_thm}
  \left\{ \begin{aligned}
   (i \partial_t + \Delta) u &= |u|^2 u \text{ on } \bR \times \bR^4, \\
   u(0) &= f^\omega,
  \end{aligned} \right.
 \end{equation}
 which scatters as $t \to \pm \infty$ in the sense that there exist states $v^{\pm} \in \dot{H}^1_x(\bR^4)$ such that 
 \[
  \lim_{t \to \pm \infty} \, \bigl\| u(t) - e^{i t \Delta}(f^\omega + v^{\pm}) \bigr\|_{\dot{H}^1_x(\bR^4)} = 0.
 \]
\end{theorem}

\begin{remark}
Analogously to Remark \ref{rem:uniqueness}, uniqueness in Theorem~\ref{thm:scattering_radial} holds in the sense that upon writing 
 \[
  u(t) = e^{it \Delta} f^\omega + v(t),
 \]
 there exists a unique global solution 
 \[
  v \in C \bigl( \bR; \dot{H}^1_x(\bR^4) \bigr) \cap X(\bR)
 \]
 to the forced defocusing cubic nonlinear Schr\"odinger equation
 \begin{equation*} 
  \left\{ \begin{aligned}
   (i \partial_t + \Delta) v &= | e^{it\Delta} f^\omega + v |^2 (e^{it\Delta} f^\omega + v) \text{ on } \bR \times \bR^4, \\
   v(0) &= 0,
  \end{aligned} \right.
 \end{equation*}
 where the function space $X(I)$ is defined in Section~\ref{sec:functional_framework}.
\end{remark}

We emphasize that prior to this work almost sure scattering for the defocusing cubic NLS~\eqref{equ:cubic_nls_radial_scattering_thm} in four space dimensions had been established by Killip-Murphy-Visan~\cite{KMV} for randomized radially symmetric initial data in $H^s_x(\bR^4)$ for the more restrictive range of regularities $\frac{5}{6} < s < 1$. 

\medskip 

In view of Theorem~\ref{thm:as_local_wellposedness} and Theorem~\ref{thm:scattering_conditional}, the proof of Theorem~\ref{thm:scattering_radial} reduces to proving the uniform-in-time energy bound~\eqref{equ:energy_hypothesis} for the nonlinear component of the solution. To this end we follow quite closely the scheme introduced by the authors~\cite{DLuM} of combining energy growth estimates with suitable approximate Morawetz estimates for the forced cubic equation, as well as incorporating further developments by Killip-Murphy-Visan~\cite{KMV}. The main novelty of our proof in comparison with~\cite{KMV} is the introduction of new almost sure bounds for weighted $L^2_t L^\infty_x(\bR\times\bR^4)$ norms of the derivative of the free evolution of the randomized radially symmetric initial data in Proposition~\ref{prop:weighted_nabla_L2Linfty_schroedinger}. The proof of these almost sure bounds hinges on a delicate combination of local smoothing estimates for the Schr\"odinger evolution and a ``radialish'' Sobolev type estimate for the square-function associated with the unit-scale frequency projections of a radially symmetric function from Lemma~\ref{lem:radialish_sobolev}. These improved almost sure bounds ultimately allow us to reach lower regularities for the random initial data and enable us to use the standard Lin-Strauss Morawetz weight $a(x) := |x|$ for our approximate Morawetz estimate in contrast to the weight $a(x) := \langle x \rangle$ used in~\cite{KMV}.

\medskip 

Finally, as a byproduct of our proof of Theorem~\ref{thm:scattering_radial}, we obtain an improvement of our almost sure scattering result for the defocusing cubic nonlinear wave equation in four space dimensions~\cite[Theorem 1.9]{DLuM}. In a very similar manner to the proof of Proposition~\ref{prop:weighted_nabla_L2Linfty_schroedinger}, we can combine local energy decay estimates for the wave equation and the aforementioned ``radialish'' Sobolev type estimate to establish almost sure bounds for weighted $L^2_t L^\infty_x(\bR\times\bR^4)$ norms of the free wave evolution of randomized radially symmetric data. These almost sure bounds for the free wave evolution are an important improvement over the authors' related almost sure bounds~\cite[Proposition 5.4]{DLuM}, and lead to a significant strengthening, in the form of a lower regularity threshold, of the almost sure scattering result for the defocusing energy-critical nonlinear wave equation in four space dimensions for randomized radially symmetric data from~\cite[Theorem 1.9]{DLuM}.

To state the precise result, we first have to recall the randomization of a pair of real-valued functions $(f_0, f_1) \in  H^s_x(\bR^4) \times H^{s-1}_x(\bR^4)$ as in \cite{DLuM}. Specifically,  we let $\{ (g_k, h_k) \}_{k \in \bZ^4}$  be a sequence of zero-mean, complex-valued Gaussian random variables on a probability space $(\Omega, {\mathcal A}, \bP)$ with the symmetry condition $g_{-k} = \overline {g_k}$ and $h_{-k} = \overline {h_k}$  for all $k \in \bZ^4$. We assume that $\{g_0, \textup{Re}(g_k), \textup{Im}(g_k)\}_{k \in \cI}$ are independent, zero-mean, real-valued random variables, where $\cI \subset \bZ^4$ is such that we have a \textit{disjoint} union $\bZ^4 = \cI \cup (-\cI) \cup \{0\}$, and similarly for the $h_k$. Then we set
\begin{equation}\label{equ:wave_randomization}
 (f_0^{\omega}, f_1^{\omega}) := \biggl( \sum_{k \in \bZ^4} g_k(\omega) P_k f_0, \sum_{k \in \bZ^4} h_k(\omega) P_k f_1 \biggr),
\end{equation}
which we note is real-valued due to the imposed symmetry conditions on the Gaussian random variables. We denote the free wave evolution of a random initial data pair $(f_0^\omega, f_1^\omega)$ by
\[
 S(t)(f^\omega_0, f^\omega_1) = \cos (t |\nabla|) f^\omega_0 + \frac{\sin(t |\nabla|)}{|\nabla|} f^\omega_1.
\]
Then we obtain the following almost sure scattering result for the defocusing energy-critical nonlinear wave equation in four space dimensions for randomized radially symmetric data whose proof is sketched in the appendix. The (much lower) regularity threshold should be compared with the regularity restriction $\frac{1}{2} < s < 1$ from the authors' previous work~\cite[Theorem 1.9]{DLuM}.
\begin{theorem} \label{thm:scattering_nlw_radial}
Let $0 < s < 1$. For real-valued radially symmetric $(f_0, f_1) \in H^s_x(\bR^4) \times H^{s-1}_x(\bR^4)$, let $(f_0^\omega, f_1^\omega)$ be the randomized initial data defined in~\eqref{equ:wave_randomization}. Then for almost every $\omega \in \Omega$, there exists a unique global solution 
 \begin{equation} 
  (u, \partial_t u) \in \bigl( S(t)(f_0^\omega, f_1^\omega), \partial_t S(t)(f_0^\omega, f_1^\omega) \bigr) + C\bigl(\bR; \dot{H}^1_x(\bR^4) \times L^2_x(\bR^4)\bigr)
 \end{equation}
 to the energy-critical defocusing nonlinear wave equation
 \begin{equation} 
  \left\{ \begin{aligned}
   -\partial_t^2 u + \Delta u &= u^3 \text{ on } \bR \times \bR^4, \\
   (u, \partial_t u)|_{t=0} &= (f_0^\omega, f_1^\omega),
  \end{aligned} \right.
 \end{equation}
 which scatters to free waves as $t \to \pm \infty$ in the sense that there exist states $(v_0^{\pm}, v_1^{\pm}) \in \dot{H}^1_x(\bR^4) \times L^2_x(\bR^4)$ such that 
 \[
  \lim_{t \to \pm \infty} \, \bigl\| \nabla_{t,x} \bigl( u(t) - S(t)(f_0^\omega + v_0^\pm, f_1^\omega + v_1^\pm) \bigr) \bigr\|_{L^2_x(\bR^4)} = 0.
 \]
\end{theorem}

\begin{remark}
 In the statement of Theorem~\ref{thm:scattering_nlw_radial} uniqueness holds in the following sense: Writing
 \[
  (u, \partial_t u) = \bigl( S(t)(P_{>4} f_0^\omega, P_{>4} f_1^\omega), \partial_t S(t)(P_{>4} f_0^\omega, P_{>4} f_1^\omega) \bigr) + (v, \partial_t v),
 \]
 there exists a unique global solution 
 \[
  (v, \partial_t v) \in C \bigl(\bR; \dot{H}^1_x(\bR^3)\bigr) \cap L^{3}_{t, loc} L^{6}_x(\bR\times\bR^4) \times C\bigl(\bR; L^2_x(\bR^4)\bigr)
 \]
 to the forced cubic nonlinear wave equation
 \begin{equation} 
  \left\{ \begin{aligned}
   -\partial_t^2 v + \Delta v &= \bigl( S(t)(P_{>4}  f_0^\omega, P_{>4} f_1^\omega) + v \bigr)^{3} \text{ on } \bR \times \bR^4, \\
   (v, \partial_t v)|_{t=0} &=  (P_{\leq 4}  f_0^\omega, P_{\leq 4} f_1^\omega),
  \end{aligned} \right.
 \end{equation}
 where $P_{\leq 4}$ and $P_{> 4}$ are the usual dyadic Littlewood-Paley projections defined in Section \ref{sec:preliminaries}.
\end{remark}

\medskip 

\noindent {\it Organization of the paper.} In Section~\ref{sec:preliminaries} we set up some notation used throughout this paper. In Section~\ref{sec:functional_framework} we introduce the functional framework for the proofs of the almost sure local well-posedness result of Theorem~\ref{thm:as_local_wellposedness} and of the conditional scattering result of Theorem~\ref{thm:scattering_conditional}. In Section~\ref{sec:trilinear_estimates} we develop the key trilinear estimates to handle all possible interactions in the forced cubic nonlinearity within this functional framework. In Section~\ref{sec:as_bounds_free_evolution} we establish various almost sure bounds on the free evolution of the random data. Finally, we provide the proofs of Theorem~\ref{thm:as_local_wellposedness} in Section~\ref{sec:as_lwp}, of Theorem~\ref{thm:scattering_conditional} in Section~\ref{sec:conditional_scattering}, and of Theorem~\ref{thm:scattering_radial} in Section~\ref{sec:radial_scattering}.

\section{Notation and preliminaries} \label{sec:preliminaries}

We denote by $C > 0$ an absolute constant which only depends on fixed parameters and whose value may change from line to line. We write $X \lesssim Y$ to indicate that $X \leq C Y$ and we use the notation $X \sim Y$ if $X \lesssim Y \lesssim X$. Moreover, we write $X \lesssim_\nu Y$ to indicate that the implicit constant depends on a parameter $\nu$ and we write $X \ll Y$ if the implicit constant should be regarded as small. We also use the notation $\langle \nabla \rangle := (1-\Delta)^{\frac{1}{2}}$, $\langle x \rangle := (1 + |x|^2)^{\frac{1}{2}}$ as well as $\langle N \rangle := (1+N^2)^{\frac{1}{2}}$.

Apart from the unit-scale frequency projections $P_k$, $k \in \bZ^4$, defined in~\eqref{equ:unit_scale_projection}, we will also make frequent use of the usual dyadic Littlewood-Paley projections $P_N$, $N \in 2^{\bZ}$, which we introduce next. Let $\varphi \in C_c^\infty(\bR^4)$ be a smooth bump function such that $\varphi(\xi) = 1$ for $|\xi| \leq 1$ and $\varphi(\xi) = 0$ for $|\xi| > 2$. Then we define for every dyadic integer $N \in 2^{\bZ}$,
\[
 \widehat{P_N f}(\xi) := \bigl( \varphi(\xi/N) - \varphi(2 \xi/N) \bigr) \hat{f}(\xi).
\]
In addition, for each dyadic integer $N \in 2^{\bZ}$ we set 
\[
 \widehat{P_{\leq N} f}(\xi) := \varphi(\xi/N) \hat{f}(\xi), \quad \widehat{P_{>N} f}(\xi) := \bigl(1-\varphi(\xi/N)\bigr) \hat{f}(\xi).
\]
We denote by $\widetilde{P}_N := P_{\leq 8N} - P_{\leq N/8}$ fattened Littlewood-Paley projections with the property that $P_N = P_N \widetilde{P}_N$. Moreover, we recall the following Bernstein estimates for the dyadic Littlewood-Paley projections.
\begin{lemma}
 Let $N \in 2^{\bZ}$. For any $1 \leq r_1 \leq r_2 \leq \infty$ and any $s \geq 0$, it holds that
 \begin{align*}
  \bigl\| P_N f \bigr\|_{L^{r_2}_x(\bR^4)} &\lesssim N^{\frac{4}{r_1}-\frac{4}{r_2}} \bigl\| P_N f \bigr\|_{L^{r_1}_x(\bR^4)}, \\
  \bigl\| P_{\leq N} f \bigr\|_{L^{r_2}_x(\bR^4)} &\lesssim N^{\frac{4}{r_1}-\frac{4}{r_2}} \bigl\| P_{\leq N} f \bigr\|_{L^{r_1}_x(\bR^4)}, \\
  \bigl\| |\nabla|^{\pm s} P_N f \bigr\|_{L^{r_1}_x(\bR^4)} &\sim N^{\pm s} \| P_N f \|_{L^{r_1}_x(\bR^4)}.
 \end{align*}
\end{lemma}

We let $\{ \be_1, \be_2, \be_3, \be_4 \}$ be an orthonormal basis of $\bR^4$ and henceforth fix our coordinate system accordingly. To formulate certain local smoothing estimates for the Schr\"odinger evolution, we will use smooth frequency projections that localize the frequency variable in the direction of an element of the orthonormal basis $\{ \be_1, \be_2, \be_3, \be_4 \}$. To this end let $\phi \in C_c^\infty(\bR)$ be a smooth bump function supported around $\sim 1$. For every dyadic integer $N \in 2^{\bZ}$ and for every $\ell = 1, \ldots, 4$, we define
\[
 \widehat{ P_{N, \be_\ell} f}(\xi) := \phi\bigl( |\xi \cdot \be_\ell| / N \bigr) \hat{f}(\xi).
\]
We may assume that the bump function $\phi$ is chosen so that for all dyadic integers $N \in 2^{\bZ}$, the frequency projections satisfy
\begin{equation} \label{equ:choice_directional_projections}
 (1 - P_{N, \be_1}) (1 - P_{N, \be_2}) (1 - P_{N, \be_3}) (1 - P_{N, \be_4}) \, P_N = 0.
\end{equation}

In the proof of a weighted almost sure bound in Proposition~\ref{prop:weighted_nabla_L2Linfty_schroedinger} we will have to decompose physical space dyadically. To this end we introduce the spatial cut-off functions
\[
 \chi_0(x) := \varphi(x)
\]
and for every integer $j \geq 1$,
\[
 \chi_j(x) := \varphi(2^{-j} x) - \varphi(2^{-(j-1)} x),
\]
where $\varphi \in C_c^\infty(\bR^4)$ is the smooth bump function introduced further above. Moreover, for any integer $j \geq 0$ we define
\[
 \chi_{\leq j}(x) := \varphi(2^{-j} x), \quad \chi_{>j}(x) := 1 - \varphi(2^{-j}x).
\]
We denote by $\widetilde{\chi}_j(x)$ slightly fattened cut-offs satisfying $\chi_j(x) = \chi_j(x) \widetilde{\chi}_j(x)$ for any integer $j \geq 0$.

\section{Functional framework} \label{sec:functional_framework}

In this section we introduce the precise functional framework that we will use in the proofs of the almost sure local well-posedness result of Theorem~\ref{thm:as_local_wellposedness} and the conditional scattering result of Theorem~\ref{thm:scattering_conditional}. 

We begin by recalling the usual Strichartz estimates for the Schr\"odinger propagator in four space dimensions. An exponent pair $(q,r)$ is called \emph{admissible} if $2 \leq q, r \leq \infty$ and the following scaling condition is satisfied
\[
 \frac{2}{q} + \frac{4}{r} = 2.
\]
\begin{proposition}(Strichartz estimates; \cite{Strichartz, Ginibre_Velo_92_CMP, KeelTao}) \label{prop:strichartz}
 Let $I \subset \bR$ be a time interval and let $(q,r), (\tilde{q}, \tilde{r})$ be admissible pairs. Then we have
 \begin{align} 
  \bigl\| e^{it\Delta} f \bigr\|_{L^q_t L^r_x(I\times\bR^4)} &\lesssim \|f\|_{L^2_x(\bR^4)}, \label{equ:strichartz_estimate}\\
  \biggl\| \int_I e^{-is\Delta} h(s,\cdot) \, ds \biggr\|_{L^2_x(\bR^4)} &\lesssim \|h\|_{L^{q'}_t L^{r'}_x(I\times\bR^4)}.
 \end{align}
 Assuming that $0 \in I$ we also have 
 \begin{align}
  \biggl\| \int_0^t e^{i(t-s)\Delta} h(s, \cdot) \, ds \biggr\|_{L^q_t L^r_x(I\times\bR^4)} &\lesssim \|h\|_{L^{\tilde{q}'}_t L^{\tilde{r}'}_x(I\times\bR^4)}.
 \end{align}
\end{proposition}
Next we introduce the lateral spaces $L^{p,q}_{\be_\ell}$ where we recall that $\{ \be_1, \be_2, \be_3, \be_4 \}$ is a fixed orthonormal basis of $\bR^4$. Given a time interval $I \subset \bR$, we define the lateral spaces $L^{p,q}_{\be_\ell}(I\times\bR^4)$ for $\ell = 1$ with norms
\begin{equation*}
 \|h\|_{L^{p,q}_{\be_1}(I\times\bR^4)} := \biggl( \int_{\bR_{x_1}} \biggl( \int_I \int_{\bR^3_{x'}} |h(t, x_1, x')|^q \, dx' \, dt \biggr)^{\frac{p}{q}} \, dx_1 \biggr)^{\frac{1}{p}}
\end{equation*}
with analogous definitions for $\ell = 2, 3, 4$ and the usual modifications when $p=\infty$ or $q=\infty$. The most important members of this family of spaces are the local smoothing space $L^{\infty,2}_{\be_\ell}$ and the inhomogeneous local smoothing space $L^{1,2}_{\be_\ell}$, which allow us to gain derivatives. In nonlinear estimates these are used along with the maximal function space~$L^{2,\infty}_{\be_\ell}$. The next proposition summarizes the estimates satisfied by the Schr\"odinger propagator in four space dimensions in the lateral spaces. These estimates follow from the local smoothing and maximal function estimates that were established by Ionescu-Kenig~\cite{Ionescu-KenigI, Ionescu-KenigII}.
\begin{proposition} \label{prop:lateral_spaces}
 Let $I \subset \bR$ be a time interval. Let $2 \leq p, q \leq \infty$ with $\frac{1}{p} + \frac{1}{q} = \frac{1}{2}$, $N \in 2^{\bZ}$ any dyadic integer and $\ell \in \{1, 2, 3, 4\}$. Then it holds that
 \begin{align}
  \bigl\| e^{i t \Delta} P_N f \bigr\|_{L^{p,q}_{\be_\ell}(I\times\bR^4)} &\lesssim N^{\frac{4}{p}-\frac{1}{2}} \|f\|_{L^2_x(\bR^4)}, \quad p \leq q, \label{equ:lateral_spaces_pleqq} \\
  \bigl\| e^{i t \Delta} P_{N, \be_\ell} P_N f \bigr\|_{L^{p,q}_{\be_\ell}(I\times\bR^4)} &\lesssim N^{\frac{4}{p}-\frac{1}{2}} \|f\|_{L^2_x(\bR^4)}, \quad p \geq q. \label{equ:lateral_spaces_pgeqq}
 \end{align}
 By duality we also have that
 \begin{align}
  \biggl\| \int_I e^{-is\Delta} P_N h(s, \cdot) \, ds \biggr\|_{L^2_x(\bR^4)} &\lesssim N^{\frac{4}{p}-\frac{1}{2}} \|h\|_{L^{p',q'}_{\be_\ell}(I\times\bR^4)}, \quad p \leq q, \\
  \biggl\| \int_I e^{-is\Delta} P_{N, \be_\ell} P_N h(s, \cdot) \, ds \biggr\|_{L^2_x(\bR^4)} &\lesssim N^{\frac{4}{p}-\frac{1}{2}} \|h\|_{L^{p',q'}_{\be_\ell}(I\times\bR^4)}, \quad p \geq q. \label{equ:lateral_spaces_pgeqq_dual}
 \end{align}
\end{proposition}
\begin{proof}
 We begin with the proof of \eqref{equ:lateral_spaces_pleqq}. The maximal function estimate from Ionescu-Kenig~\cite{Ionescu-KenigI, Ionescu-KenigII} asserts that 
 \[
  \bigl\| e^{i t \Delta} P_N f \bigr\|_{L^{2, \infty}_{\be_\ell}(I\times\bR^4)} \lesssim N^{\frac{3}{2}} \|f\|_{L^2_x(\bR^4)},
 \]
 while an application of Fubini's theorem, Bernstein estimates and the Strichartz estimate~\eqref{equ:strichartz_estimate} yields
 \begin{align*}
  \bigl\| e^{i t \Delta} P_N f \bigr\|_{L^{4, 4}_{\be_\ell}(I\times\bR^4)} &= \bigl\| e^{i t \Delta} P_N f \bigr\|_{L^4_t L^4_x(I\times\bR^4)} \lesssim N^{\frac{1}{2}} \bigl\| e^{i t \Delta} P_N f \bigr\|_{L^4_t L^{\frac{8}{3}}_x(I\times\bR^4)} \lesssim N^{\frac{1}{2}} \|f\|_{L^2_x(\bR^4)}.
 \end{align*}
 The estimate \eqref{equ:lateral_spaces_pleqq} then follows by interpolation. Analogously, \eqref{equ:lateral_spaces_pgeqq} is a consequence of the following local smoothing estimate from Ionescu-Kenig~\cite{Ionescu-KenigI, Ionescu-KenigII}
 \[
  \bigl\| e^{i t \Delta} P_N P_{N, \be_\ell} f \bigr\|_{L^{\infty, 2}_{\be_\ell}(I\times\bR^4)} \lesssim N^{-\frac{1}{2}} \|f\|_{L^2_x(\bR^4)}
 \] 
 and interpolation.
\end{proof}

To ensure that our function spaces have certain time-divisibility properties, and in order to carry out certain large deviation estimates, we rely in our work on slight variants of the local smoothing, inhomogeneous local smoothing, and maximal function spaces, namely $L^{\infty-,2+}_{\be_\ell}$, $L^{1+,2-}_{\be_\ell}$, and $L^{2+, \infty-}_{\be_\ell}$ respectively. We emphasize that the lateral space norms $\|h\|_{L^{p,q}_{\be_\ell}(I\times\bR^4)}$ are continuous as functions of the endpoints of the time interval $I$ and for $p, q < \infty$ have the following time-divisibility property
\begin{equation} \label{equ:divisibility_lateral_spaces}
 \Bigl\| \bigl\{ \|h\|_{L^{p,q}_{\be_\ell}(I_j\times\bR^4)} \bigr\}_{j=1}^{J} \Bigr\|_{\ell^{\max\{p,q\}}_j} \leq \|h\|_{L^{p,q}_{\be_\ell}(I\times\bR^4)}
\end{equation}
for any partition of a time interval $I$ into consecutive intervals $I_j$, $j = 1, \ldots, J$, with disjoint interiors. The estimate~\eqref{equ:divisibility_lateral_spaces} is a consequence of Minkowski's inequality and the embedding properties of the sequence spaces $\ell^r$. For $p, q < \infty$, \eqref{equ:divisibility_lateral_spaces} allows to partition the time interval $I$ into a controlled number of subintervals on each of which the restricted norm is arbitrarily small.

\medskip 

Finally, we are ready to give the precise definition of the space $X(I)$ to hold the solutions to the forced cubic NLS on a given time interval $I \subset \bR$. It is built from dyadic pieces in the sense that 
\[
 \|v\|_{X(I)} := \biggl( \sum_{N \in 2^{\bZ}} \|P_N v\|_{X_N(I)}^2 \biggr)^{\frac{1}{2}}.
\]
The dyadic subspace $X_N(I)$ scales at $\dot{H}^1_x(\bR^4)$-regularity and consists of several Strichartz components and a maximal function type component. To provide its precise definition we introduce a fixed, sufficiently small, absolute constant $0 < \varepsilon \ll 1$. Throughout this work it will always be implicitly understood that $\ve$ is chosen sufficiently small so that $\frac{1}{3} + 3 \varepsilon \leq s$, where $\frac{1}{3} < s < 1$ refers to the Sobolev regularity assumption for the random data in the statement of Theorem~\ref{thm:as_local_wellposedness}. For every dyadic integer $N \in 2^{\bZ}$ we then set 
\begin{align*}
 \| P_N v \|_{X_N(I)} &:= N \| P_N v \|_{L^2_t L^{4}_x(I\times\bR^4)} + N \| P_N v \|_{L^3_t L^3_x(I\times\bR^4)} + N \| P_N v \|_{L^6_t L^{\frac{12}{5}}_x(I\times\bR^4)} \\
 &\quad \quad + \sum_{\ell=1}^4 N^{-\frac{1}{2}+\varepsilon} \| P_N v \|_{L^{\frac{4}{2-\varepsilon}, \frac{4}{\varepsilon}}_{\be_\ell}(I\times\bR^4)}.
\end{align*}
We will estimate the forced cubic nonlinearity in the space $G(I)$ which is also built from dyadic pieces
\begin{align*}
 \| h \|_{G(I)} := \biggl( \sum_{N \in 2^{\bZ}} \, \| P_N h \|_{G_N(I)}^2 \biggr)^{\frac{1}{2}}
\end{align*}
and whose dyadic subspaces are defined as
\begin{align*}
 \| P_N h \|_{G_N(I)} := \inf_{P_N h = h_N^{(1)} + h_N^{(2)}} \biggl\{ N \| h_N^{(1)} \|_{L^1_t L^2_x(I\times\bR^4)} + \sum_{\ell=1}^4 N^{\frac{1}{2} + \varepsilon} \| h_N^{(2)} \|_{L^{\frac{4}{4-\varepsilon}, \frac{4}{2+\varepsilon}}_{\be_\ell}(I\times\bR^4)} \biggr\}.
\end{align*}
It will be convenient to introduce a space $Y(I)$, in which we will place the forcing term $F$. As usual, it is built from dyadic pieces in the sense that
\begin{align*}
 \| F \|_{Y(I)} := \biggl( \sum_{N \in 2^{\bZ}} \, \| P_N F \|_{Y_N(I)}^2 \biggr)^{\frac{1}{2}},
\end{align*}
where we set
\begin{align*}
 \| P_N F \|_{Y_N(I)} &:=  \langle N \rangle^{\frac{1}{3}+3\ve} \| P_N F \|_{L^3_t L^6_x(I\times\bR^4)} + \langle N \rangle^{\frac{1}{3}+3\ve} \| P_N F \|_{L^6_t L^6_x(I\times\bR^4)} \\
 &\quad \quad + \sum_{\ell=1}^4 \langle N \rangle^{\frac{1}{3}+3\varepsilon} N^{\frac{1}{2}-\varepsilon} \| P_{N, \be_\ell} P_N F \|_{L^{\frac{4}{\varepsilon}, \frac{4}{2-\varepsilon}}_{\be_\ell}(I\times\bR^4)} + \sum_{\ell=1}^4 N^{-\frac{1}{6}} \| P_N F \|_{L^{\frac{4}{2-\varepsilon}, \frac{4}{\varepsilon}}_{\be_\ell}(I\times\bR^4)}.  
\end{align*}
We will later establish in Proposition~\ref{prop:as_bound_Y_norm} that for $\frac{1}{3} < s < 1$ and $0 < \varepsilon < \frac{1}{3} (s-\frac{1}{3})$, we have $\| e^{it\Delta} f^\omega \|_{Y(\bR)} < \infty$ almost surely for the free evolution of the random data $f^\omega$ as defined in~\eqref{equ:randomization} for any $f \in H^s_x(\bR^4)$.

\medskip 

In the next lemma we collect continuity and time-divisibility properties of the $X(I)$ and $Y(I)$ norms that we will repeatedly make use of.
\begin{lemma} \label{lem:properties_XY}
 \begin{itemize}
  \item[(i)] Let $I \subset \bR$ be a closed interval. Assume that $\|v\|_{X(I)} < \infty$ and $\|F\|_{Y(I)} < \infty$. Then the mappings
   \begin{equation*}
    I \ni t \mapsto \|v\|_{X([\inf I, t])}, \qquad I \ni t \mapsto \|F\|_{Y([\inf I, t])}
   \end{equation*}
   and 
   \begin{equation*}
    I \ni t \mapsto \|v\|_{X([t, \sup I])}, \qquad I \ni t \mapsto \|F\|_{Y([t, \sup I])}
   \end{equation*}   
   are continuous with analogous statements for half-open and open intervals.
  \item[(ii)] Let $I \subset \bR$ be an interval and let $v \in X(I), F \in Y(I)$. For any partition of the interval $I$ into consecutive intervals $I_j$, $j = 1, \ldots, J$, with disjoint interiors it holds that
  \begin{equation} \label{equ:time_divisibility}
   \Bigl\| \bigl\{ \|v\|_{X(I_j)} \bigr\}_{j=1}^J \Bigr\|_{\ell^{\frac{4}{\varepsilon}}_j} \leq \|v\|_{X(I)}, \qquad \Bigl\| \bigl\{ \|F\|_{Y(I_j)} \bigr\}_{j=1}^J \Bigr\|_{\ell^{\frac{4}{\varepsilon}}_j} \leq \|F\|_{Y(I)}. 
  \end{equation}
 \end{itemize}
\end{lemma}
\begin{proof}
 The first part (i) follows from the dominated convergence theorem and the definition of the spaces $X(I)$ and $Y(I)$. For the second part (ii) we note that $\frac{4}{\ve}$ is the largest exponent in the definition of the Strichartz $L^q_t L^r_x$ and $L^{p,q}_{\be_\ell}$ components of the spaces $X(I)$ and $Y(I)$. Then the time divisibility properties~\eqref{equ:time_divisibility} follow from Minkowski's inequality, the embedding $\ell^{r_1} \hookrightarrow \ell^{r_2}$ for any $1 \leq r_1 \leq r_2\leq \infty$ for the $\ell^r$ sequence spaces and the definitions of $X(I)$ and $Y(I)$.
\end{proof}

The spaces $X(I)$ and $G(I)$ are connected by the following key linear estimate.
\begin{proposition}(Main linear estimate) \label{prop:main_linear_estimate}
 Let $I \subset \bR$ be a time interval with $t_0 \in I$ and let $v_0 \in \dot{H}^1_x(\bR^4)$. Assume that $v \colon I \times \bR^4 \to \bC$ is a solution to 
 \begin{equation} \label{equ:inhomogeneous_schroedinger}
  \left\{ \begin{aligned}
   (i \partial_t + \Delta) v &= h \text{ on } I \times \bR^4, \\
   v(t_0) &= v_0.
  \end{aligned} \right.
 \end{equation}
 Then we have for any dyadic integer $N \in 2^{\bZ}$ that 
 \begin{equation} \label{equ:main_linear_estimate_freq_localized}
  N \| P_N v \|_{L^\infty_t L^2_x(I\times\bR^4)} + \| P_N v \|_{X_N(I)} \lesssim N \| P_N v_0 \|_{L^2_x(\bR^4)} + \| P_N h \|_{G_N(I)}.
 \end{equation}
 Consequently, it holds that
 \begin{equation} \label{equ:main_linear_estimate}
  \| v \|_{L^\infty_t \dot{H}^1_x(I\times\bR^4)} + \| v \|_{X(I)} \lesssim \| v_0 \|_{\dot{H}^1_x(\bR^4)} + \| h \|_{G(I)}.
 \end{equation}
\end{proposition}

Before we can turn to the proof of Proposition~\ref{prop:main_linear_estimate}, we first need two auxiliary lemmas. 

\begin{lemma} \label{lem:auxiliary_christ_kiselev}
 Let $J, I \subset \bR$ be time intervals with $\sup J \leq \inf I$ and let $N \in 2^{\bZ}$ be any dyadic integer. Then we have for any admissible Strichartz pair $(q,r)$ that
 \begin{equation} \label{equ:auxiliary_christ_kiselev_est1}
  N \biggl\| \int_J e^{i(t-s)\Delta} P_N h(s) \, ds \biggr\|_{L^q_t L^r_x(I\times\bR^4)} \lesssim \sum_{\ell=1}^4 N^{\frac{1}{2}+\ve} \|P_N h\|_{L^{\frac{4}{4-\ve}, \frac{4}{2+\ve}}_{\be_\ell}(J\times\bR^4)}.
 \end{equation}
 Furthermore, it holds that 
 \begin{equation} \label{equ:auxiliary_christ_kiselev_est2}
  \sum_{\ell=1}^4 N^{-\frac{1}{2}+\ve} \biggl\| \int_J e^{i(t-s)\Delta} P_N h(s) \, ds \biggr\|_{L^{\frac{4}{2-\ve},\frac{4}{\ve}}_{\be_\ell}(I\times\bR^4)} \lesssim \sum_{\ell=1}^4 N^{\frac{1}{2}+\ve} \|P_N h\|_{L^{\frac{4}{4-\ve}, \frac{4}{2+\ve}}_{\be_\ell}(J\times\bR^4)}.
 \end{equation}
\end{lemma}
\begin{proof}
 The left-hand sides of \eqref{equ:auxiliary_christ_kiselev_est1} and \eqref{equ:auxiliary_christ_kiselev_est2} are both bounded by 
 \begin{equation} \label{equ:auxiliary_christ_kiselev_est_intermediate}
  N \biggl\| \int_J e^{-is\Delta} P_N h(s) \, ds \biggr\|_{L^2_x(\bR^4)}
 \end{equation}
 by the Strichartz estimate~\eqref{equ:strichartz_estimate} and by the estimate~\eqref{equ:lateral_spaces_pleqq} for the lateral spaces, respectively. Relying on the identity~\eqref{equ:choice_directional_projections}, we now further frequency decompose $P_N h$ into
 \begin{align*}
  P_N h &= P_{N, \be_1} P_{N} h  + P_{N, \be_2} (1-P_{N,\be_1}) P_{N} h + P_{N,\be_3} (1-P_{N, \be_2}) (1-P_{N, \be_1}) P_{N} h \\
  &\quad\quad  + P_{N, \be_4} (1-P_{N, \be_3}) (1-P_{N, \be_2}) (1-P_{N, \be_1}) P_{N} h.
 \end{align*}
 Using the boundedness of the projections $(1-P_{N, \be_\ell})$ on $L^2_x(\bR^4)$, we can then estimate~\eqref{equ:auxiliary_christ_kiselev_est_intermediate} by
 \begin{equation*}
  \sum_{\ell=1}^4 N \biggl\| \int_J e^{-is\Delta} P_{N, \be_\ell} P_N h(s) \, ds \biggr\|_{L^2_x(\bR^4)}.
 \end{equation*}
 Finally, by the dual estimate~\eqref{equ:lateral_spaces_pgeqq_dual} for the lateral spaces we conclude the desired bound 
 \begin{align*}
  \sum_{\ell=1}^4 N \biggl\| \int_J e^{-is\Delta} P_{N, \be_\ell} P_N h(s) \, ds \biggr\|_{L^2_x(\bR^4)} \lesssim \sum_{\ell=1}^4 N^{\frac{1}{2}+\ve} \|P_N h\|_{L^{\frac{4}{4-\ve}, \frac{4}{2+\ve}}_{\be_\ell}(J\times\bR^4)}.
 \end{align*}
\end{proof}

\begin{lemma} \label{lem:christ_kiselev}
 Let $I \subset \bR$ be a time interval with $0 = \inf I$ and let $N \in 2^{\bZ}$ be any dyadic integer. Then we have for any admissible Strichartz pair $(q,r)$ that 
 \begin{equation} \label{equ:christ_kiselev_strichartz_part}
  N \biggl\| \int_0^t e^{i(t-s)\Delta} P_N h(s) \, ds \biggr\|_{L^q_t L^r_x(I\times\bR^4)} \lesssim \sum_{\ell=1}^4 N^{\frac{1}{2}+\ve} \|P_N h\|_{L^{\frac{4}{4-\ve}, \frac{4}{2+\ve}}_{\be_\ell}(I\times\bR^4)}.
 \end{equation}
 Furthermore, it holds that 
 \begin{equation} \label{equ:christ_kiselev_maximal_function_part}
  \sum_{\ell=1}^4 N^{-\frac{1}{2}+\ve} \biggl\| \int_0^t e^{i(t-s)\Delta} P_N h(s) \, ds \biggr\|_{L^{\frac{4}{2-\ve},\frac{4}{\ve}}_{\be_\ell}(I\times\bR^4)} \lesssim \sum_{\ell=1}^4 N^{\frac{1}{2}+\ve} \|P_N h\|_{L^{\frac{4}{4-\ve}, \frac{4}{2+\ve}}_{\be_\ell}(I\times\bR^4)}.
 \end{equation}
\end{lemma}
\begin{proof}
 Our argument is a modified version of the Christ-Kiselev lemma~\cite{Christ_Kiselev}. We only prove \eqref{equ:christ_kiselev_maximal_function_part} in detail because the proof of \eqref{equ:christ_kiselev_strichartz_part} is similar. Normalizing we may assume that
 \[
  \sum_{\ell=1}^4 N^{\frac{1}{2}+\ve} \|P_N h\|_{L^{\frac{4}{4-\ve}, \frac{4}{2+\ve}}_{\be_\ell}(I\times\bR^4)} = 1.
 \]
 In order to prove \eqref{equ:christ_kiselev_maximal_function_part} it now suffices to verify that
 \begin{equation} \label{equ:christ_kiselev_maximal_function_part_simplified}
  N^{-\frac{1}{2}+\ve} \biggl\| \int_0^t e^{i(t-s)\Delta} P_N h(s) \, ds \biggr\|_{L^{\frac{4}{2-\ve},\frac{4}{\ve}}_{\be_1}(I\times\bR^4)} \lesssim 1.
 \end{equation}
 Using the time divisibility property \eqref{equ:divisibility_lateral_spaces} of the lateral spaces, we can proceed inductively to construct for every $n \in \bN$ a partition $\{ I_j^n \}_{j=1, \ldots, 2^n}$ of the interval $I$ into consecutive intervals with disjoint interiors such that for $j = 1, \ldots, 2^n$,
 \begin{equation} \label{equ:christ_kiselev_partition_smallness}
  \sum_{\ell=1}^4 N^{\frac{1}{2}+\ve} \|P_N h\|_{L^{\frac{4}{4-\ve}, \frac{4}{2+\ve}}_{\be_\ell}(I^n_j\times\bR^4)} \leq 2^{-(\frac{1}{2}+\frac{\ve}{4}) n}.
 \end{equation}
 We then perform a Whitney type decomposition of the interval $I$ and obtain that for almost every $t_1, t_2 \in I$ with $t_1 < t_2$, there exist unique $n \in \bN$ and $j \in \{1, \ldots, 2^n\}$ such that $t_1 \in I_j^n$ and $t_2 \in I_{j+1}^n$. Correspondingly, we may write 
 \begin{equation*}
  \int_0^t e^{i(t-s)\Delta} P_N h(s) \, ds = \sum_{n\in\bN} \sum_{j=1}^{2^n} \chi_{I_{j+1}^n}(t) \int_{I_j^n} e^{i(t-s)\Delta} P_N h(s) \, ds
 \end{equation*}
 with the understanding that $I_{2^n+1}^n = \emptyset$ and where $\chi_{I_{j+1}^n}(t)$ denotes a sharp cut-off function to the interval $I_{j+1}^n$. To somewhat ease the notation in the following, we shall write $(p,q) = (\frac{4}{2-\ve}, \frac{4}{\ve})$. Note that by Lemma~\ref{lem:auxiliary_christ_kiselev} and by~\eqref{equ:christ_kiselev_partition_smallness}, we have for any $n \in \bN$ and $j \in \{1, \ldots, 2^n\}$ the bound 
 \begin{equation*}
  N^{-\frac{1}{2}+\ve} \biggl\| \int_{I_j^n} e^{i(t-s)\Delta} P_N h(s) \, ds \biggr\|_{L^{p,q}_{\be_1}(I_{j+1}^n\times\bR^4)} \lesssim \sum_{\ell=1}^4 N^{\frac{1}{2}+\varepsilon} \|P_N h\|_{L^{\frac{4}{4-\ve}, \frac{4}{2+\ve}}_{\be_\ell}(I^n_j\times\bR^4)} \lesssim 2^{-(\frac{1}{2}+\frac{\ve}{4}) n}.
 \end{equation*}
 Hence, using also that $\frac{p}{q} \leq 1$, we compute that
 \begin{align*}
  &N^{-\frac{1}{2}+\ve} \biggl\| \int_0^t e^{i(t-s)\Delta} P_N h(s) \, ds \biggr\|_{L^{p,q}_{\be_1}(I\times\bR^4)} \\
  &\leq N^{-\frac{1}{2}+\ve} \sum_{n\in\bN} \, \biggl\| \sum_{j=1}^{2^n} \chi_{I_{j+1}^n} \int_{I^n_j} e^{i(t-s)\Delta} P_N h(s) \, ds \biggr\|_{L^{p,q}_{\be_1}(I\times\bR^4)} \\
  &= N^{-\frac{1}{2}+\ve} \sum_{n\in\bN} \, \Biggl( \int_{\bR_{x_1}} \biggl( \sum_{j=1}^{2^n} \int_{I_{j+1}^n} \int_{\bR^3_{x'}} \biggl| \int_{I^n_j} e^{i(t-s)\Delta} P_N h(s) \, ds \biggr|^q \, dx' \, dt \biggr)^{\frac{p}{q}} \, dx_1 \Biggr)^{\frac{1}{p}} \\
  &\leq \sum_{n\in\bN} \, \Biggl( \sum_{j=1}^{2^n} \, \biggl( N^{-\frac{1}{2}+\ve} \biggl\| \int_{I^n_j} e^{i(t-s)\Delta} P_N h(s) \, ds \biggr\|_{L^{p,q}_{\be_1}(I_{j+1}^n\times\bR^4)} \biggr)^p \Biggr)^{\frac{1}{p}} \\
  &\lesssim \sum_{n\in\bN} 2^{\frac{n}{p}} 2^{-(\frac{1}{2}+\frac{\ve}{4}) n} \\
  &\simeq \sum_{n\in\bN} 2^{-\frac{\ve}{2} n} \\
  &\lesssim 1.
 \end{align*}
 This yields \eqref{equ:christ_kiselev_maximal_function_part_simplified} and therefore finishes the proof of Lemma~\ref{lem:christ_kiselev}.
\end{proof}

\begin{proof}[Proof of Proposition~\ref{prop:main_linear_estimate}]
 Without loss of generality we may assume that $0 = t_0 = \inf I$. By Duhamel's formula for the solution to the inhomogeneous Schr\"odinger equation~\eqref{equ:inhomogeneous_schroedinger}, we have for any dyadic integer $N \in 2^{\bZ}$ that 
 \begin{equation} \label{equ:duhamel_formula_linear_estimate}
  P_N v(t) = e^{it\Delta} P_N v_0 - i \int_0^t e^{i(t-s)\Delta} P_N h(s) \, ds.
 \end{equation}
 By Proposition~\ref{prop:strichartz} and Proposition~\ref{prop:lateral_spaces} it then holds that
 \[
  N \bigl\| e^{it\Delta} P_N v_0 \bigr\|_{L^\infty_t L^2_x(I\times\bR^4)} + \bigl\| e^{it\Delta} P_N v_0 \bigr\|_{X_N(I)} \lesssim N \|P_N v_0\|_{L^2_x(\bR^4)}.
 \]
 In order to complete the proof of~\eqref{equ:main_linear_estimate_freq_localized}, it remains to verify that the Duhamel term in \eqref{equ:duhamel_formula_linear_estimate} satisfies for any admissible Strichartz pair $(q,r)$ that 
 \begin{equation} \label{equ:duhamel_estimate_one_main_linear}
  \begin{aligned}
   &N \biggl\| \int_0^t e^{i(t-s)\Delta} P_N h(s) \, ds \biggr\|_{L^q_t L^r_x(I\times\bR^4)} \\
   &\qquad \qquad \qquad + \sum_{\ell=1}^4 N^{-\frac{1}{2}+\ve} \biggl\| \int_0^t e^{i(t-s)\Delta} P_N h(s) \, ds \biggr\|_{L^{\frac{4}{2-\ve},\frac{4}{\ve}}_{\be_\ell}(I\times\bR^4)} \lesssim N \|P_N h\|_{L^1_t L^2_x(I\times\bR^4)}
  \end{aligned}
 \end{equation}
 as well as 
 \begin{equation} \label{equ:duhamel_estimate_two_main_linear}
  \begin{aligned}
   &N \biggl\| \int_0^t e^{i(t-s)\Delta} P_N h(s) \, ds \biggr\|_{L^q_t L^r_x(I\times\bR^4)} \\
   &\qquad + \sum_{\ell=1}^4 N^{-\frac{1}{2}+\ve} \biggl\| \int_0^t e^{i(t-s)\Delta} P_N h(s) \, ds \biggr\|_{L^{\frac{4}{2-\ve},\frac{4}{\ve}}_{\be_\ell}(I\times\bR^4)} \lesssim \sum_{\ell=1}^4 N^{\frac{1}{2}+\ve} \|P_N h\|_{L^{\frac{4}{4-\ve}, \frac{4}{2+\ve}}_{\be_\ell}(I\times\bR^4)}.
  \end{aligned}
 \end{equation} 
 The proof of \eqref{equ:duhamel_estimate_one_main_linear} is standard and therefore omitted, while the estimate~\eqref{equ:duhamel_estimate_two_main_linear} is provided by Lemma~\ref{lem:christ_kiselev}.
\end{proof}

\section{Trilinear estimates} \label{sec:trilinear_estimates}

In this section we systematically develop the trilinear estimates to handle all possible interaction terms that arise in the forced cubic nonlinearity within the functional framework laid out in the previous section. To this end, we frequency localize all inputs and order the inputs by the size of their frequency supports. There are then two main types of estimates. The first is when a deterministic solution $v$ appears at highest frequency. In this case, we place the associated trilinear term into the $L^1_t L^2_x$ component of the $G(I)$ space and estimate it using just a combination of Bernstein estimates and Strichartz components of the $X(I)$ and $Y(I)$ spaces. Another, more delicate, type of estimate is required if instead the (random, low-regularity) forcing term $F$ appears at highest frequency. Here we place the associated trilinear term into the $L^{1+,2-}_{\be_\ell}$ component of the $G(I)$ space, which results in a gain of $\frac{1}{2}-$ derivatives. The forcing term $F$ appearing at highest frequency is then put into the local smoothing type $L^{\infty-, 2+}_{\be_\ell}$ component of the $Y(I)$ space, which gains another $\frac{1}{2}-$ derivative. For the remaining lower frequency terms we use a mixture of Strichartz components and maximal function type $L^{2+,\infty-}_{\be_\ell}$ components of the $X(I)$ and $Y(I)$ spaces, in particular we crucially rely on our improved maximal function estimate for unit-scaled frequency localized data, see Lemma~\ref{lem:maximal_function_unit_scale} and Proposition~\ref{prop:as_bound_Y_norm}. It is here where we have to pay back some of the gained derivatives. It is important to observe that the most severe trilinear term is the $|F|^2 F$ term, see the proof of~\eqref{equ:FFF_estimate} in Proposition~\ref{prop:main_trilinear_estimates}, which ultimately leads to the regularity restriction $s > \frac{1}{3}$ in the statement of Theorem~\ref{thm:as_local_wellposedness}.

The next proposition establishes the key, frequency localized, trilinear estimates of this work.
\begin{proposition}(Main trilinear estimates) \label{prop:main_trilinear_estimates}
 Let $N_1 \gtrsim N$ and $N_1 \geq N_2 \geq N_3$ be dyadic integers. Let $\be \in \{ \be_1, \ldots, \be_4 \}$ and let $I \subset \bR$ be a time interval. Then the following trilinear estimates hold where all space-time norms are taken over $I \times \bR^4$. 
 \begin{align}
  N \bigl\| P_N \bigl( P_{N_1} v_1 \, P_{N_2} v_2 \, P_{N_3} v_3 \bigr) \bigr\|_{L^1_t L^2_x} &\lesssim \biggl( \frac{N}{N_1} \biggr) \biggl( \frac{N_3}{N_2} \biggr)^{\frac{2}{3}} \| P_{N_1} v_1 \|_{X_{N_1}} \| P_{N_2} v_2 \|_{X_{N_2}} \| P_{N_3} v_3 \|_{X_{N_3}} \label{equ:vvv_estimate}  \\
  N \bigl\| P_N \bigl( P_{N_1} v_1 \, P_{N_2} F_2 \, P_{N_3} v_3 \bigr) \bigr\|_{L^1_t L^2_x} &\lesssim \biggl( \frac{N}{N_1} \biggr) \biggl( \frac{N_3}{N_2} \biggr)^{\frac{1}{3}} \| P_{N_1} v_1 \|_{X_{N_1}} \| P_{N_2} F_2 \|_{Y_{N_2}} \| P_{N_3} v_3 \|_{X_{N_3}} \label{equ:vFv_estimate} \\
  N \bigl\| P_N \bigl( P_{N_1} v_1 \, P_{N_2} v_2 \, P_{N_3} F_3 \bigr) \bigr\|_{L^1_t L^2_x} &\lesssim \biggl( \frac{N}{N_1} \biggr) \biggl( \frac{N_3}{N_2} \biggr)^{\frac{2}{3}} \| P_{N_1} v_1 \|_{X_{N_1}} \| P_{N_2} v_2 \|_{X_{N_2}} \| P_{N_3} F_3 \|_{Y_{N_3}} \label{equ:vvF_estimate} \\
  N \bigl\| P_N \bigl( P_{N_1} v_1 \, P_{N_2} F_2 \, P_{N_3} F_3 \bigr) \bigr\|_{L^1_t L^2_x} &\lesssim \biggl( \frac{N}{N_1} \biggr) \biggl( \frac{N_3}{N_2} \biggr)^{\frac{1}{3}} \| P_{N_1} v_1 \|_{X_{N_1}} \| P_{N_2} F_2 \|_{Y_{N_2}} \| P_{N_3} F_3 \|_{Y_{N_3}} \label{equ:vFF_estimate}
 \end{align} 
 \begin{align}
  N^{\frac{1}{2}+\ve} \bigl\| P_N \bigl( P_{N_1} F_1 \, P_{N_2} F_2 \, P_{N_3} F_3 \bigr) \bigr\|_{L^{\frac{4}{4-\ve}, \frac{4}{2+\ve}}_{\be}} &\lesssim \biggl( \frac{N}{N_1} \biggr)^{\frac{1}{2}+\ve} \biggl( \frac{N_3}{N_1} \biggr)^{\frac{1}{6}} \| P_{N_1} F_1 \|_{Y_{N_1}} \| P_{N_2} F_2 \|_{Y_{N_2}} \| P_{N_3} F_3 \|_{Y_{N_3}} \label{equ:FFF_estimate} \\
  N^{\frac{1}{2}+\ve} \bigl\| P_N \bigl( P_{N_1} F_1 \, P_{N_2} v_2 \, P_{N_3} v_3 \bigr) \bigr\|_{L^{\frac{4}{4-\ve}, \frac{4}{2+\ve}}_{\be}} &\lesssim \biggl( \frac{N}{N_1} \biggr)^{\frac{1}{2}+\ve} \biggl( \frac{N_3}{N_2} \biggr)^{\frac{1}{2}-\ve} \| P_{N_1} F_1 \|_{Y_{N_1}} \| P_{N_2} v_2 \|_{X_{N_2}} \| P_{N_3} v_3 \|_{X_{N_3}} \label{equ:Fvv_estimate} \\
  N^{\frac{1}{2}+\ve} \bigl\| P_N \bigl( P_{N_1} F_1 \, P_{N_2} F_2 \, P_{N_3} v_3 \bigr) \bigr\|_{L^{\frac{4}{4-\ve}, \frac{4}{2+\ve}}_{\be}} &\lesssim \biggl( \frac{N}{N_1} \biggr)^{\frac{1}{2}+\ve} \biggl( \frac{N_3}{N_1} \biggr)^{(\frac{5}{6}-\ve) \ve} \| P_{N_1} F_1 \|_{Y_{N_1}} \| P_{N_2} F_2 \|_{Y_{N_2}} \| P_{N_3} v_3 \|_{X_{N_3}} \label{equ:FFv_estimate} \\
  N^{\frac{1}{2}+\ve} \bigl\| P_N \bigl( P_{N_1} F_1 \, P_{N_2} v_2 \, P_{N_3} F_3 \bigr) \bigr\|_{L^{\frac{4}{4-\ve}, \frac{4}{2+\ve}}_{\be}} &\lesssim \biggl( \frac{N}{N_1} \biggr)^{\frac{1}{2}+\ve} \biggl( \frac{N_3}{N_1} \biggr)^{\frac{1}{6}-\frac{2}{3}\ve} \| P_{N_1} F_1 \|_{Y_{N_1}} \| P_{N_2} F_2 \|_{Y_{N_2}} \| P_{N_3} v_3 \|_{X_{N_3}} \label{equ:FvF_estimate}
 \end{align}
\end{proposition}
\begin{proof}
 In the following all space-time norms are taken over $I\times\bR^4$. We begin with the derivation of the estimates \eqref{equ:vvv_estimate}--\eqref{equ:vFF_estimate} where a deterministic solution $v$ appears at highest frequency. The proofs are simple applications of H\"older's inequality and Bernstein estimates. 
 
 \medskip 
 
 \noindent {\it \underline{Proof of \eqref{equ:vvv_estimate}: $v_1 v_2 v_3$ case.}} 
 \begin{align*}
  &N \bigl\| P_N \bigl( P_{N_1} v_1 \, P_{N_2} v_2 \, P_{N_3} v_3 \bigr) \bigr\|_{L^1_t L^2_x} \\
  &\lesssim N \| P_{N_1} v_1 \|_{L^2_t L^4_x} \| P_{N_2} v_2 \|_{L^3_t L^4_x} \| P_{N_3} v_3 \|_{L^6_t L^\infty_x} \\
  &\lesssim N \| P_{N_1} v_1 \|_{L^2_t L^4_x} N_2^{\frac{1}{3}} \| P_{N_2} v_2 \|_{L^3_t L^3_x} N_3^{\frac{5}{3}} \| P_{N_3} v_3 \|_{L^6_t L^{\frac{12}{5}}_x} \\
  &\simeq \biggl( \frac{N}{N_1} \biggr) \biggl( \frac{N_3}{N_2} \biggr)^{\frac{2}{3}} N_1 \| P_{N_1} v_1 \|_{L^2_t L^4_x} N_2 \| P_{N_2} v_2 \|_{L^3_t L^3_x} N_3 \| P_{N_3} v_3 \|_{L^6_t L^{\frac{12}{5}}_x} \\
  &\lesssim \biggl( \frac{N}{N_1} \biggr) \biggl( \frac{N_3}{N_2} \biggr)^{\frac{2}{3}} \| P_{N_1} v_1 \|_{X_{N_1}} \| P_{N_2} v_2 \|_{X_{N_2}} \| P_{N_3} v_3 \|_{X_{N_3}}
 \end{align*}

 \medskip 
 
 \noindent {\it \underline{Proof of \eqref{equ:vFv_estimate}: $v_1 F_2 v_3$ case.}}
 \begin{align*}
  &N \bigl\| P_N \bigl( P_{N_1} v_1 \, P_{N_2} F_2 \, P_{N_3} v_3 \bigr) \bigr\|_{L^1_t L^2_x} \\
  &\lesssim N \| P_{N_1} v_1 \|_{L^2_t L^4_x} \| P_{N_2} F_2 \|_{L^3_t L^6_x} \| P_{N_3} v_3 \|_{L^6_t L^{12}_x} \\
  &\lesssim \biggl( \frac{N}{N_1} \biggr) \biggl( \frac{N_3}{N_2} \biggr)^{\frac{1}{3}} N_1 \| P_{N_1} v_1 \|_{L^2_t L^4_x} N_2^{\frac{1}{3}} \| P_{N_2} F_2 \|_{L^3_t L^6_x} N_3 \| P_{N_3} v_3 \|_{L^6_t L^{\frac{12}{5}}_x} \\
  &\lesssim \biggl( \frac{N}{N_1} \biggr) \biggl( \frac{N_3}{N_2} \biggr)^{\frac{1}{3}} \| P_{N_1} v_1 \|_{X_{N_1}} \| P_{N_2} F_2 \|_{Y_{N_2}} \| P_{N_3} v_3 \|_{X_{N_3}} 
 \end{align*}

 \medskip 
 
 \noindent {\it \underline{Proof of \eqref{equ:vvF_estimate}: $v_1 v_2 F_3$ case.}}
 \begin{align*}
  &N \bigl\| P_N \bigl( P_{N_1} v_1 \, P_{N_2} v_2 \, P_{N_3} F_3 \bigr) \bigr\|_{L^1_t L^2_x} \\
  &\lesssim N \| P_{N_1} v_1 \|_{L^2_t L^4_x} \| P_{N_2} v_2 \|_{L^3_t L^4_x} \| P_{N_3} F_3 \|_{L^6_t L^\infty_x} \\
  &\lesssim \biggl( \frac{N}{N_1} \biggr) \biggl( \frac{N_3}{N_2} \biggr)^{\frac{2}{3}} N_1 \| P_{N_1} v_1 \|_{L^2_t L^4_x} N_2 \| P_{N_2} v_2 \|_{L^3_t L^3_x} \| P_{N_3} F_3 \|_{L^6_t L^6_x} \\
  &\lesssim \biggl( \frac{N}{N_1} \biggr) \biggl( \frac{N_3}{N_2} \biggr)^{\frac{2}{3}} \| P_{N_1} v_1 \|_{X_{N_1}} \| P_{N_2} v_2 \|_{X_{N_2}} \| P_{N_3} F_3 \|_{Y_{N_3}} 
 \end{align*}
 
 \medskip 
 
 \noindent {\it \underline{Proof of \eqref{equ:vFF_estimate}: $v_1 F_2 F_3$ case.}} 
 \begin{align*}
  &N \bigl\| P_N \bigl( P_{N_1} v_1 \, P_{N_2} F_2 \, P_{N_3} F_3 \bigr) \bigr\|_{L^1_t L^2_x} \\
  &\lesssim N \| P_{N_1} v_1 \|_{L^2_t L^4_x} \| P_{N_2} F_2 \|_{L^3_t L^6_x} \| P_{N_3} F_3 \|_{L^6_t L^{12}_x} \\
  &\lesssim \biggl( \frac{N}{N_1} \biggr) \biggl( \frac{N_3}{N_2} \biggr)^{\frac{1}{3}} N_1 \| P_{N_1} v_1 \|_{L^2_t L^4_x} N_2^{\frac{1}{3}} \| P_{N_2} F_2 \|_{L^3_t L^6_x} \| P_{N_3} F_3 \|_{L^6_t L^6_x} \\
  &\lesssim \biggl( \frac{N}{N_1} \biggr) \biggl( \frac{N_3}{N_2} \biggr)^{\frac{1}{3}} \| P_{N_1} v_1 \|_{X_{N_1}} \| P_{N_2} F_2 \|_{Y_{N_2}} \| P_{N_3} F_3 \|_{Y_{N_3}} 
 \end{align*}
 
 \medskip 
 
 We now turn to the proofs of the more delicate trilinear estimates \eqref{equ:FFF_estimate}--\eqref{equ:FvF_estimate} where a (random, low-regularity) forcing term $F$ appears at highest frequency. 
 
 \medskip 
 
 \noindent {\it \underline{Proof of \eqref{equ:FFF_estimate}: $F_1 F_2 F_3$ case.}} We first use H\"older's inequality to place the $P_{N_3} F_3$ piece into $L^{\frac{4}{2-\ve}, \frac{4}{\ve}}_{\be}$,
 \begin{equation} \label{equ:FFF_case_first_step}
  N^{\frac{1}{2}+\ve} \bigl\| P_N \bigl( P_{N_1} F_1 \, P_{N_2} F_2 \, P_{N_3} F_3 \bigr) \bigr\|_{L^{\frac{4}{4-\ve}, \frac{4}{2+\ve}}_{\be}} \lesssim N^{\frac{1}{2}+\ve} \bigl\| P_{N_1} F_1 \, P_{N_2} F_2 \bigr\|_{L^{2,2}_{\be}} \| P_{N_3} F_3 \|_{L^{\frac{4}{2-\ve},\frac{4}{\ve}}_{\be}}. 
 \end{equation}
 Relying on the identity \eqref{equ:choice_directional_projections} we further decompose the highest frequency piece $P_{N_1} F_1$ into
 \begin{align*}
  P_{N_1} F_1 &= P_{N_1, \be_1} P_{N_1} F_1 + P_{N_1, \be_2} (1-P_{N_1,\be_1}) P_{N_1} F_1 + P_{N_1,\be_3} (1-P_{N_1, \be_2}) (1-P_{N_1, \be_1}) P_{N_1} F_1 \\
  &\quad\quad  + P_{N_1, \be_4} (1-P_{N_1, \be_3}) (1-P_{N_1, \be_2}) (1-P_{N_1, \be_1}) P_{N_1} F_1.
 \end{align*}
 We note that the operators $(1-P_{N_1,\be_\ell}) \widetilde{P}_{N_1}$ are disposable since their kernels are uniformly bounded in $L^1_x$. Using this disposability and the fact that $L^{2,2}_{\be} = L^{2,2}_{\be_\ell}$ for $\ell = 1, \ldots, 4$ by Fubini's theorem, we may now use H\"older's inequality to bound \eqref{equ:FFF_case_first_step} by
 \begin{align*}
  &N^{\frac{1}{2}+\ve} \bigl\| P_{N_1} F_1 \, P_{N_2} F_2 \bigr\|_{L^{2,2}_{\be}} \| P_{N_3} F_3 \|_{L^{\frac{4}{2-\ve},\frac{4}{\ve}}_{\be}} \\
  &\lesssim N^{\frac{1}{2}+\varepsilon} \sum_{\ell=1}^4 \| P_{N_1, \be_\ell} P_{N_1} F_1 \|_{L^{\frac{4}{\ve}, \frac{4}{2-\ve}}_{\be_\ell}} \| P_{N_2} F_2 \|_{L^{\frac{4}{2-\ve}, \frac{4}{\ve}}_{\be_\ell}} \| P_{N_3} F_3 \|_{L^{\frac{4}{2-\ve},\frac{4}{\ve}}_{\be}} \\
  &\simeq \biggl( \frac{N}{N_1} \biggr)^{\frac{1}{2}+\ve} \biggl( \frac{N_2}{N_1} \biggr)^{\frac{1}{6}} \biggl( \frac{N_3}{N_1} \biggr)^{\frac{1}{6}} \sum_{\ell=1}^4 N_1^{\frac{5}{6}+\ve} \| P_{N_1, \be_\ell} P_{N_1} F_1 \|_{L^{\frac{4}{\ve}, \frac{4}{2-\ve}}_{\be_\ell}} N_2^{-\frac{1}{6}} \| P_{N_2} F_2 \|_{L^{\frac{4}{2-\ve}, \frac{4}{\ve}}_{\be_\ell}} N_3^{-\frac{1}{6}} \| P_{N_3} F_3 \|_{L^{\frac{4}{2-\ve},\frac{4}{\ve}}_{\be}} \\
  &\lesssim \biggl( \frac{N}{N_1} \biggr)^{\frac{1}{2}+\ve} \biggl( \frac{N_3}{N_1} \biggr)^{\frac{1}{6}} \| P_{N_1} F_1 \|_{Y_{N_1}} \| P_{N_2} F_2 \|_{Y_{N_2}} \| P_{N_3} F_3 \|_{Y_{N_3}}.
 \end{align*}
 
 \medskip 
 
 \noindent {\it \underline{Proof of \eqref{equ:Fvv_estimate}: $F_1 v_2 v_3$ case.}} We begin by placing the $P_{N_3} v_3$ piece into $L^{\frac{4}{2-\ve}, \frac{4}{\ve}}_{\be}$,
 \begin{equation} \label{equ:Fvv_case_first_step}
  N^{\frac{1}{2}+\ve} \bigl\| P_N \bigl( P_{N_1} F_1 \, P_{N_2} v_2 \, P_{N_3} v_3 \bigr) \bigr\|_{L^{\frac{4}{4-\ve}, \frac{4}{2+\ve}}_{\be}} \lesssim N^{\frac{1}{2}+\ve} \bigl\| P_{N_1} F_1 \, P_{N_2} v_2 \bigr\|_{L^{2,2}_{\be}} \| P_{N_3} v_3 \|_{L^{\frac{4}{2-\ve}, \frac{4}{\ve}}_{\be}}.
 \end{equation}
 On the one hand we may proceed as in the proof of~\eqref{equ:FFF_estimate} to estimate the term $\bigl\| P_{N_1} F_1 \, P_{N_2} v_2 \bigr\|_{L^{2,2}_{\be}}$ by
 \begin{equation*}
  \bigl\| P_{N_1} F_1 \, P_{N_2} v_2 \bigr\|_{L^{2,2}_{\be}} \lesssim \biggl( \sum_{\ell=1}^4 \| P_{N_1, \be_\ell} P_{N_1} F_1 \|_{L^{\frac{4}{\ve}, \frac{4}{2-\ve}}_{\be_\ell}} \biggr) \biggl( \sum_{\ell=1}^4 \| P_{N_2} v_2 \|_{L^{\frac{4}{2-\ve}, \frac{4}{\ve}}_{\be_\ell}} \biggr),
 \end{equation*}
 while on the other hand we may use that $L^{2,2}_\be = L^2_t L^2_x$ to bound the term by
 \begin{equation*}
  \bigl\| P_{N_1} F_1 \, P_{N_2} v_2 \bigr\|_{L^{2,2}_{\be}} \lesssim \| P_{N_1} F_1 \|_{L^6_t L^6_x} \| P_{N_2} v_2 \|_{L^3_t L^3_x}.
 \end{equation*}
 Interpolating between the two cases we now estimate~\eqref{equ:Fvv_case_first_step} by 
 \begin{align*}
  &N^{\frac{1}{2}+\ve} \bigl\| P_{N_1} F_1 \, P_{N_2} v_2 \bigr\|_{L^{2,2}_{\be}} \| P_{N_3} v_3 \|_{L^{\frac{4}{2-\ve}, \frac{4}{\ve}}_{\be}} \\
  &\lesssim N^{\frac{1}{2}+\ve} \biggl( \sum_{\ell=1}^4 \| P_{N_1, \be_\ell} P_{N_1} F_1 \|_{L^{\frac{4}{\ve}, \frac{4}{2-\ve}}_{\be_\ell}} \biggr)^{\frac{1}{3} + \frac{8 \ve}{9-6\ve}} \biggl( \| P_{N_1} F_1 \|_{L^6_t L^6_x} \biggr)^{\frac{2}{3} - \frac{8 \ve}{9-6\ve}}  \times \\
  &\qquad \qquad \times \biggl( \sum_{\ell=1}^4 \| P_{N_2} v_2 \|_{L^{\frac{4}{2-\ve}, \frac{4}{\ve}}_{\be_\ell}} \biggr)^{\frac{1}{3} + \frac{8 \ve}{9-6\ve}} \biggl( \| P_{N_2} v_2 \|_{L^3_t L^3_x} \biggr)^{\frac{2}{3} - \frac{8 \ve}{9-6\ve}} \| P_{N_3} v_3 \|_{L^{\frac{4}{2-\ve}, \frac{4}{\ve}}_{\be}} \\
  &\simeq \biggl( \frac{N}{N_1} \biggr)^{\frac{1}{2}+\ve} \biggl( \frac{N_3}{N_2} \biggr)^{\frac{1}{2}-\ve} \times \\
  &\qquad \qquad \times \biggl( \sum_{\ell=1}^4 N_1^{\frac{5}{6} + (1-\frac{4}{9-6\ve}) \ve} \| P_{N_1, \be_\ell} P_{N_1} F_1 \|_{L^{\frac{4}{\ve}, \frac{4}{2-\ve}}_{\be_\ell}} \biggr)^{\frac{1}{3} + \frac{8 \ve}{9-6\ve}} \biggl( N_1^{\frac{1}{3} + (1-\frac{4}{9-6\ve}) \ve} \| P_{N_1} F_1 \|_{L^6_t L^6_x} \biggr)^{\frac{2}{3} - \frac{8 \ve}{9-6\ve}}  \times \\
  &\qquad \qquad \times \biggl( \sum_{\ell=1}^4 N_2^{-\frac{1}{2}+\ve} \| P_{N_2} v_2 \|_{L^{\frac{4}{2-\ve}, \frac{4}{\ve}}_{\be_\ell}} \biggr)^{\frac{1}{3} + \frac{8 \ve}{9-6\ve}} \biggl( N_2 \| P_{N_2} v_2 \|_{L^3_t L^3_x} \biggr)^{\frac{2}{3} - \frac{8 \ve}{9-6\ve}}  N_3^{-\frac{1}{2}+\ve} \| P_{N_3} v_3 \|_{L^{\frac{4}{2-\ve}, \frac{4}{\ve}}_{\be}} \\
  &\lesssim \biggl( \frac{N}{N_1} \biggr)^{\frac{1}{2}+\ve} \biggl( \frac{N_3}{N_2} \biggr)^{\frac{1}{2}-\ve} \| P_{N_1} F_1 \|_{Y_{N_1}} \| P_{N_2} v_2 \|_{X_{N_2}} \| P_{N_3} v_3 \|_{X_{N_3}}.
 \end{align*}

 \medskip 
 
 \noindent {\it \underline{Proof of \eqref{equ:FFv_estimate}: $F_1 F_2 v_3$ case.}} Here we first place the $P_{N_2} F_2$ piece into $L^{\frac{4}{2-\ve}, \frac{4}{\ve}}_{\be}$,
 \begin{equation*}
  N^{\frac{1}{2}+\ve} \bigl\| P_N \bigl( P_{N_1} F_1 \, P_{N_2} F_2 \, P_{N_3} v_3 \bigr) \bigr\|_{L^{\frac{4}{4-\ve}, \frac{4}{2+\ve}}_{\be}} \lesssim N^{\frac{1}{2}+\ve} \bigl\| P_{N_1} F_1 \, P_{N_3} v_3 \bigr\|_{L^{2,2}_{\be}} \| P_{N_2} F_2 \|_{L^{\frac{4}{2-\ve}, \frac{4}{\ve}}_{\be}}.
 \end{equation*}
 Then interpolating as in the proof of \eqref{equ:Fvv_estimate}, we bound this by
 \begin{align*}
  &N^{\frac{1}{2}+\ve} \bigl\| P_{N_1} F_1 \, P_{N_3} v_3 \bigr\|_{L^{2,2}_{\be}} \| P_{N_2} F_2 \|_{L^{\frac{4}{2-\ve}, \frac{4}{\ve}}_{\be}} \\
  &\lesssim N^{\frac{1}{2}+\ve} \biggl( \sum_{\ell=1}^4 \| P_{N_1, \be_\ell} P_{N_1} F_1 \|_{L^{\frac{4}{\ve}, \frac{4}{2-\ve}}_{\be_\ell}} \biggr)^{\frac{2}{3}+\ve} \biggl( \| P_{N_1} F_1 \|_{L^6_t L^6_x} \biggr)^{\frac{1}{3}-\ve} \times \\
  &\qquad \qquad \times  \| P_{N_2} F_2 \|_{L^{\frac{4}{2-\ve}, \frac{4}{\ve}}_{\be}} \biggl( \sum_{\ell=1}^4 \|P_{N_3} v_3\|_{L^{\frac{4}{2-\ve},\frac{4}{\ve}}_{\be_\ell}} \biggr)^{\frac{2}{3}+\ve} \biggl( \| P_{N_3} v_3 \|_{L^3_t L^3_x} \biggr)^{\frac{1}{3}-\ve} \\
  &\simeq \biggl( \frac{N}{N_1} \biggr)^{\frac{1}{2}+\ve} \biggl( \frac{N_2}{N_1} \biggr)^{\frac{1}{6}} \biggl( \frac{N_3}{N_1} \biggr)^{(\frac{5}{6}-\ve) \ve} \times \\
  &\qquad \qquad \times \biggl( \sum_{\ell=1}^4 N_1^{\frac{5}{6} + (\frac{4}{3}-\ve)\ve} \| P_{N_1, \be_\ell} P_{N_1} F_1 \|_{L^{\frac{4}{\ve}, \frac{4}{2-\ve}}_{\be_\ell}} \biggr)^{\frac{2}{3}+\ve} \biggl( N_1^{\frac{1}{3} + (\frac{4}{3}-\ve)\ve} \| P_{N_1} F_1 \|_{L^6_t L^6_x} \biggr)^{\frac{1}{3}-\ve} \times \\
  &\qquad \qquad \times N_2^{-\frac{1}{6}} \| P_{N_2} F_2 \|_{L^{\frac{4}{2-\ve}, \frac{4}{\ve}}_{\be}} \biggl( \sum_{\ell=1}^4 N_3^{-\frac{1}{2}+\ve} \|P_{N_3} v_3\|_{L^{\frac{4}{2-\ve},\frac{4}{\ve}}_{\be_\ell}} \biggr)^{\frac{2}{3}+\ve} \biggl( N_3 \| P_{N_3} v_3 \|_{L^3_t L^3_x} \biggr)^{\frac{1}{3}-\ve} \\
  &\lesssim \biggl( \frac{N}{N_1} \biggr)^{\frac{1}{2}+\ve} \biggl( \frac{N_3}{N_1} \biggr)^{(\frac{5}{6}-\ve) \ve} \| P_{N_1} F_1 \|_{Y_{N_1}} \| P_{N_2} F_2 \|_{Y_{N_2}} \| P_{N_3} v_3 \|_{X_{N_3}}.
 \end{align*}

 \medskip 
 
 \noindent {\it \underline{Proof of \eqref{equ:FvF_estimate}: $F_1 v_2 F_3$ case.}} As usual we first place the random piece $P_{N_3} F_3$ into $L^{\frac{4}{2-\ve},\frac{4}{\ve}}_{\be}$,
 \begin{equation*}
  N^{\frac{1}{2}+\ve} \bigl\| P_N \bigl( P_{N_1} F_1 \, P_{N_2} v_2 \, P_{N_3} F_3 \bigr) \bigr\|_{L^{\frac{4}{4-\ve}, \frac{4}{2+\ve}}_{\be}} \lesssim N^{\frac{1}{2}+\ve} \bigl\| P_N \bigl( P_{N_1} F_1 \, P_{N_2} v_2 \bigr\|_{L^{2,2}_\be} \| P_{N_3} F_3 \|_{L^{\frac{4}{2-\ve},\frac{4}{\ve}}_{\be}}.
 \end{equation*}
 Then, analogously to the previous two cases, we interpolate to obtain 
 \begin{align*}
  &N^{\frac{1}{2}+\ve} \bigl\| P_N \bigl( P_{N_1} F_1 \, P_{N_2} v_2 \bigr\|_{L^{2,2}_\be} \| P_{N_3} F_3 \|_{L^{\frac{4}{2-\ve},\frac{4}{\ve}}_{\be}} \\
  &\lesssim N^{\frac{1}{2}+\ve} \biggl( \sum_{\ell=1}^4 \| P_{N_1, \be_\ell} P_{N_1} F_1 \|_{L^{\frac{4}{\ve}, \frac{4}{2-\ve}}_{\be_\ell}} \biggr)^{\frac{2}{3}} \biggl( \| P_{N_1} F_1 \|_{L^6_t L^6_x} \biggr)^{\frac{1}{3}}  \times \\
  &\qquad \qquad \times \biggl( \sum_{\ell=1}^4 \| P_{N_2} v_2 \|_{L^{\frac{4}{2-\ve}, \frac{4}{\ve}}_{\be_\ell}} \biggr)^{\frac{2}{3}} \biggl( \| P_{N_2} v_2 \|_{L^3_t L^3_x} \biggr)^{\frac{1}{3}} \| P_{N_3} F_3 \|_{L^{\frac{4}{2-\ve}, \frac{4}{\ve}}_{\be}} \\
  &\simeq \biggl( \frac{N}{N_1} \biggr)^{\frac{1}{2}+\ve} \biggl( \frac{N_3}{N_2} \biggr)^{\frac{2}{3} \ve} \biggl( \frac{N_3}{N_1} \biggr)^{\frac{1}{6}-\frac{2}{3}\ve} \times \\
  &\qquad \qquad \times \biggl( \sum_{\ell=1}^4 N_1^{\frac{5}{6}+\frac{1}{3}\ve} \| P_{N_1, \be_\ell} P_{N_1} F_1 \|_{L^{\frac{4}{\ve}, \frac{4}{2-\ve}}_{\be_\ell}} \biggr)^{\frac{2}{3}} \biggl( N_1^{\frac{1}{3}+\frac{1}{3} \ve} \| P_{N_1} F_1 \|_{L^6_t L^6_x} \biggr)^{\frac{1}{3}}  \times \\
  &\qquad \qquad \times \biggl( \sum_{\ell=1}^4 N_2^{-\frac{1}{2}+\ve} \| P_{N_2} v_2 \|_{L^{\frac{4}{2-\ve}, \frac{4}{\ve}}_{\be_\ell}} \biggr)^{\frac{2}{3}} \biggl( N_2 \| P_{N_2} v_2 \|_{L^3_t L^3_x} \biggr)^{\frac{1}{3}} N_3^{-\frac{1}{6}} \| P_{N_3} F_3 \|_{L^{\frac{4}{2-\ve}, \frac{4}{\ve}}_{\be}} \\
  &\lesssim \biggl( \frac{N}{N_1} \biggr)^{\frac{1}{2}+\ve} \biggl( \frac{N_3}{N_1} \biggr)^{\frac{1}{6}-\frac{2}{3}\ve} \|P_{N_1} F_1\|_{Y_{N_1}} \|P_{N_2} v_2\|_{X_{N_2}} \|P_{N_3} F_3\|_{Y_{N_3}}. \qedhere
 \end{align*}
\end{proof}

The frequency localized, trilinear estimates~\eqref{equ:vvv_estimate}--\eqref{equ:FvF_estimate} imply an important set of nonlinear estimates that we will need for the proofs of the almost sure local well-posedness result of Theorem~\ref{thm:as_local_wellposedness} and the conditional scattering result of Theorem~\ref{thm:scattering_conditional}. More precisely, given any time interval $I$, any forcing term $F \in Y(I)$, and any $v, v_1, v_2, u, w \in X(I)$, it is an easy consequence of the exponential gains in the frequency differences in the trilinear estimates~\eqref{equ:vvv_estimate}--\eqref{equ:FvF_estimate} to conclude that
\begin{equation} \label{equ:forced_cubic_nonlinearity_in_G}
 \begin{aligned}
  \bigl\| |F+v|^2 (F+v) \bigr\|_{G(I)} &\lesssim \| v \|_{X(I)}^3 + \| v \|_{X(I)}^2 \| F \|_{Y(I)} + \| v \|_{X(I)} \| F \|_{Y(I)}^2 + \| F \|_{Y(I)}^3 \\
  &\lesssim \| F \|_{Y(I)}^3 + \| v \|_{X(I)}^3
 \end{aligned}
\end{equation}
as well as
\begin{equation} \label{equ:difference_forced_cubic_nonlinearity_in_G}
 \begin{aligned}
  &\bigl\| |F+v_1|^2 (F+v_1) - |F+v_2|^2 (F+v_2) \bigr\|_{G(I)} \\
  &\quad \quad \lesssim \| v_1 - v_2 \|_{X(I)} \bigl( \| F \|_{Y(I)}^2 + \| v_1 \|_{X(I)}^2 + \| v_2 \|_{X(I)}^3 \bigr).
 \end{aligned}
\end{equation}
Moreover, using that 
\begin{align*}
  |F+u+w|^2(F+u+w) - |u|^2 u &= |F|^2 F + |w|^2 w + \mathcal{O}\bigl( F^2 u \bigr) + \mathcal{O}\bigl( F^2 w \bigr) + \mathcal{O}\bigl( F u^2 \bigr) \\ 
  &\quad \quad + \mathcal{O}\bigl( F u w \bigr) + \mathcal{O}\bigl( F w^2 \bigr) + \mathcal{O}\bigl( u^2 w \bigr) + \mathcal{O}\bigl( u w^2 \bigr),
\end{align*}
we may also infer that 
\begin{equation} \label{equ:difference_equation_nonlinearity_in_G}
 \begin{aligned}
  &\bigl\| |F+u+w|^2(F+u+w) - |u|^2 u \bigr\|_{G(I)} \\
  &\quad \quad \lesssim \|F\|_{Y(I)}^3 + \|w\|_{X(I)}^3 + \|F\|_{Y(I)} \|u\|_{X(I)}^2 + \|u\|_{X(I)}^2 \|w\|_{X(I)}.
 \end{aligned}
\end{equation}

\section{Almost sure bounds for the free evolution} \label{sec:as_bounds_free_evolution}

In this section we establish various almost sure bounds for the free evolution of the random data. In Subsection~\ref{subsec:prob_preliminaries} we recall some probabilistic preliminaries. Then Subsection~\ref{subsec:as_bounds_Y} is dedicated to the proof that the $Y(\bR)$ norm of $e^{it\Delta} f^\omega$ is almost surely finite for any $f \in H^s_x(\bR^4)$ with $\frac{1}{3} < s < 1$, while in Subsection~\ref{subsec:as_bounds_radial} we establish further almost sure bounds that enter the proof of Theorem~\ref{thm:scattering_radial} and mostly crucially rely on a radial symmetry assumption.

\subsection{Probabilistic preliminaries} \label{subsec:prob_preliminaries}

We first recall the following large deviation estimate.
\begin{lemma}[\protect{\cite[Lemma 3.1]{BT1}}] \label{lem:large_deviation_estimate}
 Let $\{g_n\}_{n=1}^{\infty}$ be a sequence of real-valued, independent, zero-mean random variables with associated distributions $\{\mu_n\}_{n=1}^{\infty}$ on a probability space $(\Omega, {\mathcal A}, \bP)$. Assume that the distributions satisfy the property that there exists $c > 0$ such that
 \begin{equation*}
  \biggl| \int_{-\infty}^{+\infty} e^{\gamma x} d\mu_n(x) \biggr| \leq e^{c \gamma^2} \text{ for  all } \gamma \in \bR \text{ and for all } n \in \mathbb{N}.
 \end{equation*}
 Then there exists $\alpha > 0$ such that for every $\lambda > 0$ and every sequence $\{c_n\}_{n=1}^{\infty} \in \ell^2(\bN;\bC)$ of complex numbers, 
 \begin{equation*}
  \bP \Bigl( \bigl\{ \omega : \bigl| \sum_{n=1}^{\infty} c_n g_n(\omega) \bigr| > \lambda \bigr\} \Bigr) \leq 2 \exp \biggl(- \alpha \frac{\lambda^2}{\sum_n |c_n|^2} \biggr). 
 \end{equation*}
 As a consequence there exists $C > 0$ such that for every $2 \leq p < \infty$ and every $\{c_n\}_{n=1}^{\infty} \in \ell^2(\bN; \bC)$,
 \begin{equation*}
  \Bigl\| \sum_{n=1}^{\infty} c_n g_n(\omega) \Bigr\|_{L^p_\omega(\Omega)} \leq C \sqrt{p} \Bigl( \sum_{n=1}^{\infty} |c_n|^2 \Bigr)^{1/2}.
 \end{equation*}
\end{lemma}

We also present a lemma that will be used to estimate the probability of certain events. Its proof is an adaptation of the proof of Lemma~4.5 in \cite{Tz10}.

\begin{lemma} \label{lem:probability_estimate}
Let $F$ be a real-valued measurable function on a probability space $(\Omega, {\mathcal A}, \bP)$. Suppose that there exist $C_0 > 0$, $K > 0$ and $p_0 \geq 1$ such that for every $p \geq p_0$ we have
\begin{align*}
 \| F \|_{L^p_{\omega}(\Omega)} \leq \sqrt{p} \, C_0 K. 
\end{align*}
Then there exist $c > 0$ and $C_1 > 0$, depending on $C_0$ and $p_0$ but independent of $K$, such that for every $\lambda > 0$,
\begin{align*}
  \bP \bigl( \bigl\{ \omega \in \Omega : |F(\omega)| > \lambda \bigr\} \bigr) \leq C_1 e^{- c \lambda^2/K^2}.
\end{align*}
In particular, it follows that 
\[
 \bP \bigl( \bigl\{ \omega \in \Omega : |F(\omega)| < \infty \bigr\} \bigr) = 1.
\]
\end{lemma}

\subsection{Almost sure bounds for the $Y(\bR)$ norm} \label{subsec:as_bounds_Y}

The purpose of this subsection is to establish the following almost sure bound for the $Y(\bR)$ norm of the free evolution of the random data.
\begin{proposition} \label{prop:as_bound_Y_norm}
 Let $\frac{1}{3} < s < 1$ and let $0 < \varepsilon < \frac{1}{3} (s-\frac{1}{3})$. Let $f \in H^s_x(\bR^4)$ and denote by $f^\omega$ the randomization of $f$ as defined in~\eqref{equ:randomization}. Then there exist absolute constants $C > 0$ and $c > 0$ such that for any $\lambda > 0$ it holds that
 \begin{equation} 
  \bP \Bigl( \Bigl\{ \omega \in \Omega : \bigl\| e^{i t \Delta} f^\omega \bigr\|_{Y(\bR)} > \lambda  \Bigr\} \Bigr) \leq C \exp \Bigl( - c \lambda^2 \| f \|_{H^s_x(\bR^4)}^{-2} \Bigr).
 \end{equation}
 In particular, we have for almost every $\omega \in \Omega$ that
 \begin{equation}
  \bigl\| e^{i t \Delta} f^\omega \bigr\|_{Y(\bR)} < \infty.
 \end{equation} 
\end{proposition}

Before we turn to the proof of Proposition~\ref{prop:as_bound_Y_norm}, we first present a key improvement of the maximal function estimate for unit-scale frequency localized data. Its proof is an adaptation of the proof of Lemma~4.1 in Ionescu-Kenig~\cite{Ionescu-KenigII}.

\begin{lemma}[Unit-scale maximal function estimate] \label{lem:maximal_function_unit_scale}
There exists an absolute constant $C \geq 1$ such that for all $k \in \bZ^4$ with $|k| \geq 10$ and for each $\ell = 1, \ldots, 4$, it holds that
 \begin{equation} \label{equ:maximal_function_unit_scale_support}
  \bigl\| e^{it\Delta} P_k f \bigr\|_{L^{2,\infty}_{\be_\ell}(\bR\times\bR^4)} \leq C |k|^{\frac{1}{2}} \|P_k f\|_{L^2_x(\bR^4)}.
 \end{equation}
\end{lemma}
\begin{proof}
It suffices to consider the case $\ell=1$. Let $k \in \bZ^4$ with $|k| \geq 10$ be fixed. We use the notation $x' = (x_2, x_3, x_4)$ and we denote by $\xi'$ the associated Fourier variable coordinates. Now let $\varphi_k \in C_c^\infty(\bR)$ and $\eta_k \in C_c^\infty(\bR^3)$ be bump functions with unit-sized support and uniformly bounded derivatives such that 
\[
 \psi(\xi - k) = \psi(\xi - k) \varphi_k(\xi_1) \eta_k(\xi') \quad \text{ for all } \xi = (\xi_1, \xi') \in \bR^4.
\]
In particular, we have that $| \text{supp}(\varphi_k) | \lesssim 1$ and $| \text{supp}(\eta_k) | \lesssim 1$. 

By a $TT^\ast$--argument the proof of~\eqref{equ:maximal_function_unit_scale_support} reduces to establishing the estimate
\begin{equation} \label{equ:max_est_goal}
 \biggl\| \int_{\bR_{\xi_1}} \int_{\bR_{\xi'}^{3}} e^{i x_1 \xi_1} e^{i x' \cdot \xi'} e^{- i t (\xi_1^2 + |\xi'|^2)} \varphi_k(\xi_1) \eta_k(\xi') \, d\xi' \, d\xi_1 \biggr\|_{L^1_{x_1} L^\infty_{t, x'}} \lesssim |k|.
\end{equation}
Since $\eta_k(\cdot)$ has unit-sized support, a stationary phase argument yields
\begin{equation} \label{equ:max_est_bound1}
 \sup_{x' \in \bR^{3}, \, t \in \bR} \, \biggl| \int_{\bR_{\xi'}^3} e^{i x' \cdot \xi'} e^{- i t |\xi'|^2} \eta_k(\xi') \, d\xi' \biggr| \lesssim \min \bigl\{ 1, |t|^{-\frac{3}{2}} \bigr\}.
\end{equation}
Analogously, by the unit-sized support of $\varphi_k(\cdot)$ and by stationary phase we have
\begin{equation} \label{equ:max_est_bound2}
 \sup_{x_1 \in \bR, \,t \in \bR} \, \biggl| \int_{\bR_{\xi_1}} e^{i x_1 \xi_1} e^{- i t \xi_1^2} \varphi_k(\xi_1) \, d\xi_1 \biggr| \lesssim \min \bigl\{ 1, |t|^{-\frac{1}{2}} \bigr\}.
\end{equation}
Moreover, if $|x_1| \geq 1000 |k| |t|$, then integrating by parts twice, we find that
\begin{equation} \label{equ:max_est_bound3}
 \biggl| \int_{\bR_{\xi_1}} e^{i x_1 \xi_1} e^{- i t \xi_1^2} \varphi_k(\xi_1) \, d\xi_1 \biggr| \lesssim \frac{1}{1 +  |x_1|^2},
\end{equation}
using that in the regime $|x_1| \geq 1000 |k| |t|$, we have 
\[
 \biggl| \frac{1}{x_1 - 2 t \xi_1} \biggr| \lesssim \frac{1}{|x_1|} \quad \text{ and } \quad \biggl| \frac{\partial^j}{\partial \xi_1^j} \biggl( \frac{1}{x_1 - 2 t \xi_1} \biggr) \biggr| \lesssim \frac{1}{|x_1|} \quad \text{ for } j = 1,2.
\]
The latter bounds follow from the observation that if $|x_1| \geq 1000 |k| |t|$, then we have $|x_1 - 2 t \xi_1| \geq c |x_1|$ for some small absolute constant $0 < c \ll 1$ since also $|\xi_1| \leq 10 |k|$. Additionally, note that in this regime we may also bound 
\[
 |t| \lesssim \frac{|x_1|}{|k|} \lesssim |x_1|.
\]
 
Thus, using \eqref{equ:max_est_bound1}, \eqref{equ:max_est_bound2} and \eqref{equ:max_est_bound3}, we obtain uniformly for all $x_1 \in \bR$ that
\begin{equation} \label{equ:max_est_bound4}
 \begin{aligned}
  &\biggl\| \int_{\bR_{\xi_1}} \int_{\bR_{\xi'}^{3}} e^{i x_1 \xi_1} e^{i x' \cdot \xi'} e^{- i t (\xi_1^2 + |\xi'|^2)} \varphi_k(\xi_1) \eta_k(\xi') \, d\xi' \, d\xi_1 \biggr\|_{L^\infty_{t, x'}} \\
  &\qquad \lesssim \chi_{[0, |k|]}(|x_1|) + \chi_{[|k|, \infty)}(|x_1|) \biggl( \frac{|k|}{|x_1|} \biggr)^2 +  \frac{1}{1 +  |x_1|^2},
 \end{aligned}
\end{equation}
where $\chi_{[0, |k|]}(\cdot)$ and $\chi_{[|k|, \infty)}(\cdot)$ are sharp cut-off functions to the intervals $[0,|k|]$ and $[|k|,\infty)$, respectively. Integrating in $x_1$ over \eqref{equ:max_est_bound4}, we obtain the desired bound \eqref{equ:max_est_goal}. 
 \end{proof}

\begin{remark}
 We emphasize that the proof of Lemma~\ref{lem:maximal_function_unit_scale} generalizes to all space dimensions $d \geq 3$ and the same half derivative cost $|k|^{\frac{1}{2}}$ occurs in all space dimensions $d\geq 3$. The cost of only half a derivative in~\eqref{equ:maximal_function_unit_scale_support} should be compared with the three halves derivative cost of the usual maximal function estimate~\eqref{equ:lateral_spaces_pleqq} for dyadically frequency localized data.
\end{remark}

We are now prepared to prove Proposition~\ref{prop:as_bound_Y_norm}.
\begin{proof}[Proof of Proposition \ref{prop:as_bound_Y_norm}]
For any $p \geq \frac{4}{\varepsilon} $ we have by Minkowski's inequality that
\begin{align*}
 \Bigl\| \bigl\| e^{i t \Delta} f^\omega \bigr\|_{Y(\bR)} \Bigr\|_{L^p_\omega} \leq \biggl( \sum_{N \in 2^{\bZ}} \,  \Bigl\| \bigl\| P_N e^{i t \Delta} f^\omega \bigr\|_{Y_N(\mathbb{R})} \Bigr\|_{L^p_\omega}^2 \biggr)^{\frac{1}{2}}.
\end{align*}
We consider the components of the $Y_N(\bR)$ norm separately. In the sequel, we will repeatedly use that the free evolution and the frequency projections commute. We first treat the component coming from the Strichartz spaces. By Lemma~\ref{lem:large_deviation_estimate}, the unit-scale Bernstein estimate~\eqref{equ:unit_scale_bernstein}, and the Strichartz estimates~\eqref{equ:strichartz_estimate}, we have
\begin{align*}
 &\langle N\rangle^{\frac{1}{3} + 3 \ve} \bigl\| \| P_N e^{i t \Delta} f^\omega \|_{L^3_tL^6_x \cap L^6_t L^6_x(\bR\times\bR^4)} \bigr\|_{L^p_\omega} \\
 &\leq \langle N\rangle^{\frac{1}{3} + 3 \ve} \biggl\| \Bigl\| \sum_{|k| \sim N}  g_{k}(\omega) e^{i t \Delta} P_k f \Bigr\|_{L^p_\omega} \biggr\|_{L^3_t L^6_x \cap L^6_t L^6_x(\bR\times\bR^4)} \\
 &\lesssim \sqrt{p}\, \langle N\rangle^{\frac{1}{3} + 3 \ve}  \biggl( \sum_{|k| \sim N} \bigl\| e^{i t \Delta} P_k f \bigr\|_{L^3_tL^6_x \cap L^6_t L^6_x(\bR\times\bR^4)}^2 \biggr)^{\frac{1}{2}}\\
 &\lesssim\sqrt{p}\,  \langle N\rangle^{\frac{1}{3} + 3 \ve}  \biggl( \sum_{|k| \sim N} \bigl\| e^{i t \Delta} P_k f \bigr\|_{L^3_tL^3_x \cap L^6_t L^{\frac{12}{5}}_x(\bR\times\bR^4)}^2 \biggr)^{\frac{1}{2}}\\
 &\lesssim\sqrt{p}  \,\langle N\rangle^{\frac{1}{3} + 3 \ve}  \biggl( \sum_{|k|\sim N} \| P_k f \|_{L^2_x(\bR^4)}^2 \biggr)^{\frac{1}{2}} \\
 &\simeq \sqrt{p} \, \langle N\rangle^{\frac{1}{3} + 3 \ve}  \| P_N f \|_{L^2_x(\bR^4)}.
\end{align*}
Next, we estimate the local smoothing type component of the $Y_N(\bR)$ component. Here we first apply the local smoothing type estimate~\eqref{equ:lateral_spaces_pgeqq} for the lateral spaces and then use the large deviation estimate from Lemma~\ref{lem:large_deviation_estimate} to obtain that
\begin{align*}
 \biggl\| \, \sum_{\ell=1}^4 \langle N \rangle^{\frac{1}{3}+3\ve} N^{\frac{1}{2}-\ve} \bigl\| e^{it\Delta} P_{N, \be_\ell} P_N f^\omega \bigr\|_{L^{\frac{4}{\ve}, \frac{4}{2-\ve}}_{\be_\ell}(\bR\times\bR^4)} \biggr\|_{L^p_\omega} &\lesssim \biggl\| \langle N \rangle^{\frac{1}{3}+3\ve} \| P_N f^\omega \|_{L^2_x(\bR^4)} \biggr\|_{L^p_\omega} \\
 &\lesssim \sqrt{p} \, \langle N \rangle^{\frac{1}{3}+3\ve} \biggl( \sum_{|k| \sim N} \| P_k f \|_{L^2_x(\bR^4)}^2 \biggr)^{\frac{1}{2}} \\
 &\lesssim \sqrt{p} \, \langle N \rangle^{\frac{1}{3}+3\ve} \| P_N f\|_{L^2_x(\bR^4)}.
\end{align*}
Finally, we turn to the maximal function type component of the $Y_N(\bR)$ norm, where we distinguish the large frequency regime $N \gtrsim 1$ and the small frequency regime $N \lesssim 1$. For large frequencies $N \gtrsim 1$ we first use the large deviation estimate from Lemma~\ref{lem:large_deviation_estimate} and then interpolate between the improved maximal function estimate~\eqref{equ:maximal_function_unit_scale_support} for unit-scale frequency localized data and an estimate of the $L^4_t L^4_x(\bR\times\bR^4)$ norm of the free evolution of unit-scale frequency localized data (based on the unit-scale Bernstein estimate~\eqref{equ:unit_scale_bernstein} and Strichartz estimates~\eqref{equ:strichartz_estimate}) to conclude that
\begin{align*}
 \biggl\| \, \sum_{\ell=1}^4 N^{-\frac{1}{6}} \bigl\| e^{it\Delta} P_N f^\omega \bigr\|_{L^{\frac{4}{2-\ve}, \frac{4}{\ve}}_{\be_\ell}(\bR\times\bR^4)} \biggr\|_{L^p_\omega} &\lesssim \sqrt{p} \sum_{\ell=1}^4 N^{-\frac{1}{6}} \biggl( \sum_{|k| \sim N} \bigl\| e^{it\Delta} P_k f \bigr\|_{L^{\frac{4}{2-\ve}, \frac{4}{\ve}}_{\be_\ell}(\bR\times\bR^4)}^2 \biggr)^{\frac{1}{2}} \\
 &\lesssim \sqrt{p} \, N^{-\frac{1}{6}} N^{\frac{1}{2}} \biggl( \sum_{|k| \sim N} \| P_k f \|_{L^2_x(\bR^4)}^2 \biggr)^{\frac{1}{2}} \\
 &\lesssim \sqrt{p} \, \langle N \rangle^{\frac{1}{3}} \| P_N f \|_{L^2_x(\bR^4)}.
\end{align*}
For small frequencies $N \lesssim 1$, we directly apply the usual maximal function type estimate~\eqref{equ:lateral_spaces_pleqq}, trivially bounding the resulting frequency factors, and then use the large deviation estimate to infer that in this case
\begin{align*}
 \biggl\| \, \sum_{\ell=1}^4 N^{-\frac{1}{6}} \bigl\| e^{it\Delta} P_N f^\omega \bigr\|_{L^{\frac{4}{2-\ve}, \frac{4}{\ve}}_{\be_\ell}(\bR\times\bR^4)} \biggr\|_{L^p_\omega} &\lesssim N^{-\frac{1}{6}} N^{\frac{3}{2}-\ve} \Bigl\| \bigl\| P_N f^\omega \bigr\|_{L^2_x(\bR^4)} \Bigr\|_{L^p_\omega} \lesssim \sqrt{p} \, \| P_N f \|_{L^2_x(\bR^4)}.
\end{align*}

Putting all of the above estimates together, we find that 
\begin{align*}
 \Bigl\| \bigl\| e^{i t \Delta} f^\omega \bigr\|_{Y(\bR)} \Bigr\|_{L^p_\omega} \lesssim \sqrt{p} \biggl( \sum_{N \in 2^{\bZ}} \,  \bigl( \langle N \rangle^{\frac{1}{3} + 3 \ve} \| P_N f \|_{L^2_x(\bR^4)} \bigr)^2 \biggr)^{\frac{1}{2}} \lesssim \sqrt{p} \, \|f\|_{H^s_x(\bR^4)},
\end{align*}
from which the assertion follows by Lemma~\ref{lem:probability_estimate}.
\end{proof}

\subsection{Almost sure bounds for the proof of Theorem~\ref{thm:scattering_radial}} \label{subsec:as_bounds_radial}

We first collect several harmonic analysis estimates that will play an important role in establishing the additional almost sure bounds for the free evolution of randomized radial data for the proof of Theorem~\ref{thm:scattering_radial}. We will crucially rely on the following local smoothing estimate for the Schr\"odinger evolution.
\begin{proposition}(Local smoothing estimate; \cite{Constantin_Saut, Sjoelin, Vega}) \label{prop:local_smoothing}
 For any $\delta > 0$ it holds that
 \begin{equation} \label{equ:local_smoothing_weight}
  \bigl\| \langle x \rangle^{-\frac{1}{2}-\delta} e^{i t \Delta} f \bigr\|_{L^2_t L^2_x(\bR\times\bR^d)} \lesssim_{\delta} \bigl\| |\nabla|^{-\frac{1}{2}} f \bigr\|_{L^2_x(\bR^d)}.
 \end{equation}
 Furthermore, by scaling, we have 
 \begin{equation} \label{equ:local_smoothing_ball}
  \sup_{R > 0} \, R^{-\frac{1}{2}} \bigl\| e^{ i t \Delta} f \bigr\|_{L^2_t L^2_x(\bR \times \{ |x| \leq R \})} \lesssim \bigl\| |\nabla|^{-\frac{1}{2}} f \bigr\|_{L^2_x(\bR^d)}.
 \end{equation}
\end{proposition}

Moreover, we recall the following ``radialish'' Sobolev type estimate for the square-function associated with the unit-scale projections of a radial function that was established by the authors~\cite{DLuM}. It will be a key ingredient in the proofs of several almost sure bounds in this subsection.

\begin{lemma}[\protect{``Radialish'' Sobolev estimate; \cite[Lemma 2.2]{DLuM}}] \label{lem:radialish_sobolev}
 For any $\delta > 0$ there exists $C_\delta > 0$ such that for all radially symmetric $f \colon \bR^4 \to \bC$ it holds that
 \begin{equation} \label{equ:radialish_sobolev}
  \biggl\| |x|^{\frac{3}{2}} \Bigl( \sum_{k \in \bZ^4} \bigl| P_k f \bigr|^2 \Bigr)^{\frac{1}{2}} \biggr\|_{L^\infty_x(\bR^4)} \leq C_\delta \| f \|_{H^\delta_x(\bR^4)}.
 \end{equation}
\end{lemma}

We also state a simple corollary of the above ``radialish'' Sobolev estimate.

\begin{lemma}
 Let $2 \leq r \leq \infty$ and $\delta > 0$. Then there exists $C_\delta > 0$ such that we have for all radially symmetric functions $f \colon \bR^4 \to \bC$ and for all dyadic integers $N \geq 1$ that
 \begin{equation} \label{equ:radialish_sobolev_interpolated}
   \biggl\| |x|^{\frac{3}{2} (1-\frac{2}{r}) } \Bigl( \sum_{k \in \bZ^4} \bigl| P_k f_N \bigr|^2 \Bigr)^{\frac{1}{2}} \biggr\|_{L^r_x(\bR^4)} \leq C_\delta N^\delta \bigl\| f_N \bigr\|_{L^2_x(\bR^4)}.
 \end{equation}
\end{lemma}
\begin{proof}
 The estimate follows by interpolation between the radialish Sobolev estimate from Lemma~\ref{lem:radialish_sobolev} and the trivial estimate 
 \[
  \biggl\| \Bigl( \sum_{k \in \bZ^4} \bigl| P_k f_N \bigr|^2 \Bigr)^{\frac{1}{2}} \biggr\|_{L^2_x(\bR^4)} \lesssim \bigl\| f_N \bigr\|_{L^2_x(\bR^4)}. \qedhere
 \]
\end{proof}

We borrow the following useful technical lemma from~\cite[Lemma 2.3]{KMV}.

\begin{lemma} \label{lem:weighted_bound_into_freq_pieces}
 For $0 \leq \alpha \leq 1$ and $4 < r < \infty$, we have 
 \begin{equation*}
  \bigl\| \langle x \rangle^\alpha u \bigr\|_{L^\infty_x(\bR^4)} \lesssim \bigl\| \langle x \rangle^{\alpha} P_{\leq 1} u \bigr\|_{L^r_x(\bR^4)} + \sum_{N \geq 2} N^{\frac{4}{r}} \bigl\| \langle x \rangle^{\alpha} P_N u \bigr\|_{L^r_x(\bR^4)}.
 \end{equation*}
\end{lemma}

Moreover, we will require the next two technical lemmas on certain operator norm bounds.

\begin{lemma}
 Let $2 \leq r \leq \infty$. For any $k \in \bZ^4$, any integers $j, \ell > 0$ with $\ell > j + 5$ and any integer $M > 0$, it holds that
 \begin{equation} \label{equ:operator_bound_chi_P_chi}
  \bigl\| \chi_j P_k \chi_\ell \bigr\|_{L^2_x(\bR^4) \to L^r_x(\bR^4)} \leq C_M 2^{-M \ell}.
 \end{equation}
\end{lemma}
\begin{proof}
 We have that
 \begin{align*}
  \chi_j(x) \bigl( P_k \chi_\ell f \bigr)(x) &\sim \chi_j(x) \int_{\bR^4} e^{i k (x-y)} \check{\psi}(x-y) \chi_\ell(y) f(y) \, dy \\
  &\sim \chi_j(x) \int_{\bR^4} e^{i k (x-y)} \widetilde{\varphi}(2^{-\ell}(x-y)) \check{\psi}(x-y) \chi_\ell(y) f(y) \, dy,
 \end{align*}
 where we could freely introduce a suitable bump function $\widetilde{\varphi}(\cdot)$ with $\widetilde{\varphi}(z) = 1$ for $|z| \sim 1$, because $|x-y| \sim 2^\ell$ thanks to $|x| \sim 2^j$, $|y| \sim 2^\ell$ and $\ell > j+5$. Thus, we obtain from Young's inequality and the rapid decay of $\check{\psi}(\cdot)$ that for any integer $M > 0$,
 \begin{align*}
  \bigl\| \chi_j P_k \chi_\ell f \bigr\|_{L^r_x(\bR^4)} \lesssim \bigl\| \widetilde{\varphi}(2^{-\ell} \cdot) \check{\psi}(\cdot) \bigr\|_{L^{\frac{2r}{r+2}}_x(\bR^4)} \| f \|_{L^2_x(\bR^4)} \lesssim C_M 2^{-M \ell} \| f \|_{L^2_x(\bR^4)}. &\qedhere
 \end{align*}
\end{proof}

\begin{lemma}
 For any integer $\ell \geq 0$, any $k, m \in \bZ^4$ with $|k-m| \geq 100$, any $2 \leq r \leq \infty$, and any integer $M > 0$ we have 
 \begin{equation} \label{equ:operator_bound_PchiP}
  \bigl\| P_k \chi_\ell P_m \bigr\|_{L^2_x(\bR^4) \to L^r_x(\bR^4)} \leq C_M 2^{-M \ell} |k-m|^{-M}.
 \end{equation}
\end{lemma}
\begin{proof}
 By repeated integration by parts we find that the Fourier transform of $\chi_\ell$ satisfies for any $\xi \neq 0$ that
 \begin{equation} \label{equ:bounds_fourier_transform_chi_ell}
  \bigl| \widehat{\chi}_\ell(\xi) \bigr| \leq C_M 2^{-M \ell} |\xi|^{-M} 2^{d \ell} \bigl| ( \widehat{\nabla_x^M \chi} )(2^\ell \xi) \bigr|.
 \end{equation}
 Thus, from the unit-scale Bernstein estimate~\eqref{equ:unit_scale_bernstein} we obtain for any function $g \in L^2_x(\bR^4)$ that
 \begin{align*}
  \bigl\| P_k \chi_\ell P_m g \bigr\|_{L^r_x(\bR^4)} &\lesssim \bigl\| P_k \chi_\ell P_m g \bigr\|_{L^2_x(\bR^4)}.
 \end{align*}
 Using Plancherel's theorem, it then follows that
 \begin{align*}
  \bigl\| P_k \chi_\ell P_m g \bigr\|_{L^2_x(\bR^4)} &= \bigl\| {\mathcal F}\bigl( P_k \chi_\ell P_m g \bigr) \bigr\|_{L^2_\xi(\bR^4)} \\
  &\sim \biggl\| \psi(\xi - k) \int_{\bR^4} \widehat{\chi}_\ell(\xi - \eta) \psi(\eta - m) \hat{g}(\eta) \, d\eta \biggr\|_{L^2_\xi(\bR^4)} \\
  &\sim \biggl\| \psi(\xi - k) \int_{\bR^4} \widetilde{\varphi}\bigl( |k-m|^{-1} (\xi-\eta) \bigr) \widehat{\chi}_\ell(\xi - \eta) \psi(\eta - m) \hat{g}(\eta) \, d\eta \biggr\|_{L^2_\xi(\bR^4)}, 
 \end{align*} 
 where in the last step we have exploited that we may freely introduce a suitable bump function $\widetilde{\varphi}(\cdot)$ with $\widetilde{\varphi}(z) = 1$ near $|z| \sim 1$ due to the frequency support properties of $\psi(\cdot - k)$ and $\psi(\cdot - m)$. We then use Young's inequality as well as the bound~\eqref{equ:bounds_fourier_transform_chi_ell} on the Fourier transform of $\chi_\ell$ to obtain
 \begin{align*} 
  \bigl\| P_k \chi_\ell P_m g \bigr\|_{L^2_x(\bR^4)} \lesssim \bigl\| \widetilde{\varphi}\bigl( |k-m|^{-1} \cdot \bigr) \widehat{\chi}_\ell(\cdot) \bigr\|_{L^1_\xi(\bR^4)} \| \hat{g} \|_{L^2_\xi(\bR^4)} \lesssim 2^{-M \ell} |k-m|^{-M} \| g \|_{L^2_x(\bR^4)},
 \end{align*}
 which concludes the proof.
\end{proof}

We are now prepared to establish a key weighted $L^2_t L^\infty_x(\bR\times\bR^4)$ almost sure bound for the derivative of the free evolution of randomized radial data. The proof combines the large deviation estimate from Lemma~\ref{lem:large_deviation_estimate}, the local smoothing estimate from Proposition~\ref{prop:local_smoothing} and the ``radialish'' Sobolev estimate from Lemma~\ref{lem:radialish_sobolev}.

\begin{proposition} \label{prop:weighted_nabla_L2Linfty_schroedinger}
 Let $\frac{1}{2} < s < 1$ and $0 \leq \alpha < 1$. Let $f \in H^s_x(\bR^4)$ be radially symmetric and denote by $f^\omega$ the randomization of $f$ as defined in~\eqref{equ:randomization}. Then there exist absolute constants $C > 0$ and $c > 0$ such that for any $\lambda > 0$ it holds that
 \begin{equation} 
  \bP \Bigl( \Bigl\{ \omega \in \Omega : \bigl\| \langle x \rangle^\alpha \nabla e^{i t \Delta} f^\omega \bigr\|_{L^2_t L^\infty_x(\bR \times \bR^4)} > \lambda  \Bigr\} \Bigr) \leq C \exp \Bigl( - c \lambda^2 \| f \|_{H^s_x(\bR^4)}^{-2} \Bigr).
 \end{equation}
 In particular, we have for almost every $\omega \in \Omega$ that
\begin{equation}
  \bigl\| \langle x \rangle^\alpha \nabla e^{i t \Delta} f^\omega \bigr\|_{L^2_t L^\infty_x(\bR \times \bR^4)} < \infty.
 \end{equation} 
\end{proposition}
\begin{proof}
In the following all space-time norms are taken over $\bR \times \bR^4$. We denote by $\delta > 0$ a constant which will be chosen sufficiently small at the end of the proof. For any $4 < r < \infty$ we have by Lemma~\ref{lem:weighted_bound_into_freq_pieces} that
 \begin{equation} \label{equ:weighted_nabla_L2Linfty_first_decomposition}
  \bigl\| \langle x \rangle^\alpha \nabla e^{i t \Delta} f^\omega \bigr\|_{L^p_\omega L^2_t L^\infty_x} \lesssim \bigl\| \langle x \rangle^\alpha \nabla e^{i t \Delta} P_{\leq 1} f^\omega \bigr\|_{L^p_\omega L^2_t L^r_x} + \sum_{N \geq 2} N^{\frac{4}{r}} \bigl\| \langle x \rangle^\alpha \nabla e^{i t \Delta} P_N f^\omega \bigr\|_{L^p_\omega L^2_t L^r_x}.
 \end{equation}
 We now estimate the high-frequency terms in the sum on the right-hand side of~\eqref{equ:weighted_nabla_L2Linfty_first_decomposition}. For each dyadic integer $N \geq 2$ we decompose physical space dyadically to write 
 \begin{align*}
  \bigl\| \langle x \rangle^\alpha \nabla e^{i t \Delta} P_N f^\omega \bigr\|_{L^p_\omega L^2_t L^r_x} &\leq \sum_{j \geq 0} \, \bigl\| \chi_j \langle x \rangle^\alpha \nabla e^{i t \Delta} P_N f^\omega \bigr\|_{L^p_\omega L^2_t L^r_x} \lesssim \sum_{j \geq 0} 2^{\alpha j} \bigl\| \chi_j \nabla e^{i t \Delta} P_N f^\omega \bigr\|_{L^p_\omega L^2_t L^r_x}.
 \end{align*}
 Applying the large deviation estimate from Lemma~\ref{lem:large_deviation_estimate} and using the shorthand notation $f_N \equiv P_N f$, we find for any $p \geq r$ that 
 \begin{align*}
  &\sum_{j \geq 0} 2^{\alpha j} \bigl\| \chi_j \nabla e^{i t \Delta} P_N f^\omega \bigr\|_{L^p_\omega L^2_t L^r_x} \\
  &\lesssim \sqrt{p} \sum_{j \geq 0} 2^{\alpha j} \biggl\| \biggl( \sum_{|k| \sim N} \bigl| \chi_j P_k \nabla e^{i t \Delta} f_N \bigr|^2 \biggr)^{\frac{1}{2}} \biggr\|_{L^2_t L^r_x} \\
  &\lesssim \sqrt{p} \sum_{j \geq 0} 2^{\alpha j} \biggl\| \biggl( \sum_{|k| \sim N} \bigl| \chi_j P_k \chi_{\leq j+5} \nabla e^{i t \Delta} f_N \bigr|^2 \biggr)^{\frac{1}{2}} \biggr\|_{L^2_t L^r_x} \\
  &\quad + \sqrt{p} \sum_{j \geq 0} 2^{\alpha j} \biggl\| \biggl( \sum_{|k| \sim N} \bigl| \chi_j P_k \chi_{> j+5} \nabla e^{i t \Delta} f_N \bigr|^2 \biggr)^{\frac{1}{2}} \biggr\|_{L^2_t L^r_x} \\
  &\equiv \sqrt{p} ( I + II ).
 \end{align*}
 We first estimate the main term $I$. For $j=0$ we just use the unit-scale Bernstein estimate~\eqref{equ:unit_scale_bernstein} and the local smoothing estimate~\eqref{equ:local_smoothing_ball} for the free Schr\"odinger evolution
 \begin{align*}
  \biggl\| \biggl( \sum_{|k| \sim N} \bigl| \chi_0 P_k \chi_{\leq 5} \nabla e^{i t \Delta} f_N \bigr|^2 \biggr)^{\frac{1}{2}} \biggr\|_{L^2_t L^r_x} &\lesssim \biggl( \sum_{|k| \sim N} \bigl\| P_k \chi_{\leq 5} \nabla e^{i t \Delta} f_N \bigr\|^2_{L^2_t L^r_x} \biggr)^{\frac{1}{2}} \\
  &\lesssim \biggl( \sum_{|k| \sim N} \bigl\| P_k \chi_{\leq 5} \nabla e^{i t \Delta} f_N \bigr\|^2_{L^2_t L^2_x} \biggr)^{\frac{1}{2}} \\
  &\lesssim \bigl\| \chi_{\leq 5} \nabla e^{i t \Delta} f_N \bigr\|_{L^2_t L^2_x} \\
  &\lesssim \bigl\| |\nabla|^{+\frac{1}{2}} f_N \bigr\|_{L^2_x}.
 \end{align*}
 Then the most delicate case is to estimate the sum over all $j \geq 1$ in term $I$. Here we use a combination of the ``radialish'' Sobolev estimate~\eqref{equ:radialish_sobolev_interpolated} and the local smoothing estimate~\eqref{equ:local_smoothing_ball} to obtain for all sufficiently large $r < \infty$ with $\alpha + \frac{3}{r} < 1$ that
 \begin{align*}
  &\sum_{j \geq 1} 2^{\alpha j} \biggl\| \biggl( \sum_{|k| \sim N} \bigl| \chi_j P_k \chi_{\leq j+5} \nabla e^{i t \Delta} f_N \bigr|^2 \biggr)^{\frac{1}{2}} \biggr\|_{L^2_t L^r_x} \\
  &\lesssim \sum_{j \geq 1} 2^{\alpha j} 2^{-\frac{3}{2}(1 - \frac{2}{r}) j} \biggl\| |x|^{\frac{3}{2}(1-\frac{2}{r})} \biggl( \sum_{|k| \sim N} \bigl| P_k \chi_{\leq j+5} \nabla e^{i t \Delta} f_N \bigr|^2 \biggr)^{\frac{1}{2}} \biggr\|_{L^2_t L^r_x} \\
  &\lesssim \sum_{j \geq 1} 2^{\alpha j} 2^{-\frac{3}{2}(1 - \frac{2}{r}) j} N^{\delta} \bigl\| \chi_{\leq j+5} \nabla e^{i t \Delta} f_N \bigr\|_{L^2_t L^2_x} \\
  &\lesssim \sum_{j \geq 1} 2^{\alpha j} 2^{-\frac{3}{2}(1 - \frac{2}{r}) j} N^{\delta} 2^{\frac{1}{2} j} \bigl\| |\nabla|^{+\frac{1}{2}} f_N \bigr\|_{L^2_x} \\
  &\lesssim N^\delta \bigl\| |\nabla|^{+\frac{1}{2}} f_N \bigr\|_{L^2_x}. 
 \end{align*}
 Next we turn to estimating the remainder term $II$. To this end we introduce the shorthand notation $g_N \equiv \nabla e^{i t \Delta} f_N$. Then we have 
 \begin{align*}
  II \lesssim \sum_{j \geq 0} 2^{\alpha j} \biggl( \sum_{|k| \sim N} \bigl\| \chi_j P_k \chi_{> j+5} g_N \bigr\|^2_{L^2_t L^r_x} \biggr)^{\frac{1}{2}} \lesssim \sum_{j \geq 0} 2^{\alpha j} \Biggl( \sum_{|k| \sim N} \biggl( \sum_{\ell > j+5}  \bigl\| \chi_j P_k \chi_\ell \widetilde{\chi}_\ell g_N \bigr\|_{L^2_t L^r_x}  \biggr)^2 \Biggr)^{\frac{1}{2}}.
 \end{align*}
 Now we further decompose $\widetilde{\chi}_\ell g_N$ in frequency space at unit-scale and obtain that the previous line is bounded by
 \begin{align*}
  &\sum_{j \geq 0} 2^{\alpha j} \Biggl( \sum_{|k| \sim N} \biggl( \sum_{\ell > j+5} \sum_{m \in \bZ^4} \bigl\| \chi_j P_k \chi_\ell P_m \widetilde{\chi}_\ell g_N \bigr\|_{L^2_t L^r_x}  \biggr)^2 \Biggr)^{\frac{1}{2}} \\
  &\lesssim \sum_{j \geq 0} 2^{\alpha j} \Biggl( \sum_{|k| \sim N} \biggl( \sum_{\ell > j+5} \sum_{|m-k| \leq 100} \bigl\| \chi_j P_k \chi_\ell P_m \widetilde{\chi}_\ell g_N \bigr\|_{L^2_t L^r_x}  \biggr)^2 \Biggr)^{\frac{1}{2}} \\
  &\quad \quad + \sum_{j \geq 0} 2^{\alpha j} \Biggl( \sum_{|k| \sim N} \biggl( \sum_{\ell > j+5} \sum_{|m-k| > 100} \bigl\| \chi_j P_k \chi_\ell P_m \widetilde{\chi}_\ell g_N \bigr\|_{L^2_t L^r_x}  \biggr)^2 \Biggr)^{\frac{1}{2}} \\
  &\equiv IIA + IIB.
 \end{align*}
 Then we can easily estimate the term $IIA$ using the operator norm bound~\eqref{equ:operator_bound_chi_P_chi} and the local smoothing estimate~\eqref{equ:local_smoothing_ball} to find that 
 \begin{align*}
  IIA &\lesssim \sum_{j \geq 0} 2^{\alpha j} \Biggl( \sum_{|k| \sim N} \biggl( \sum_{\ell > j+5} \sum_{|m-k| \leq 100} 2^{-10 \ell} \bigl\| P_m \widetilde{\chi}_\ell g_N \bigr\|_{L^2_t L^2_x}  \biggr)^2 \Biggr)^{\frac{1}{2}} \\
  &\lesssim \sum_{j \geq 0} 2^{\alpha j} \Biggl( \sum_{|k| \sim N} \sum_{\ell > j+5} \sum_{|m-k| \leq 100} 2^{-10 \ell} \bigl\| P_m \widetilde{\chi}_\ell g_N \bigr\|_{L^2_t L^2_x}^2 \Biggr)^{\frac{1}{2}} \\
  &\lesssim \sum_{j \geq 0} 2^{\alpha j} \Biggl( \sum_{\ell > j+5} 2^{-10 \ell} \bigl\| \widetilde{\chi}_\ell g_N \bigr\|_{L^2_t L^2_x}^2 \Biggr)^{\frac{1}{2}} \\
  &\lesssim \sum_{j \geq 0} 2^{\alpha j} \Biggl( \sum_{\ell > j+5} 2^{-10 \ell} 2^{\ell} \bigl\| |\nabla|^{+\frac{1}{2}} f_N \bigr\|_{L^2_x}^2 \Biggr)^{\frac{1}{2}} \\
  &\lesssim \bigl\| |\nabla|^{+\frac{1}{2}} f_N \bigr\|_{L^2_x}.
 \end{align*}
 Next we use the operator norm bound~\eqref{equ:operator_bound_PchiP} to estimate the term $IIB$ by
 \begin{align*}
  IIB &\lesssim \sum_{j \geq 0} 2^{\alpha j} \Biggl( \sum_{|k| \sim N} \biggl( \sum_{\ell > j+5} \sum_{|m-k| > 100} \bigl\| P_k \chi_\ell P_m \widetilde{P}_m \widetilde{\chi}_\ell g_N \bigr\|_{L^2_t L^r_x} \biggr)^2 \Biggr)^{\frac{1}{2}} \\
  &\lesssim \sum_{j \geq 0} 2^{\alpha j} \Biggl( \sum_{k \in \bZ^4} \biggl( \sum_{|m-k| > 100} \sum_{\ell > j+5} |k-m|^{-10} 2^{-10\ell} \bigl\| \widetilde{P}_m \widetilde{\chi}_\ell g_N \bigr\|_{L^2_t L^2_x} \biggr)^2 \Biggr)^{\frac{1}{2}},
 \end{align*}
 where we denote by $\widetilde{P}_m$ a slight fattening of the unit-scale projection $P_m$, $m \in \bZ^4$, with the property that $P_m = P_m \widetilde{P}_m$. Using Young's inequality (for the convolution of $k, m \in \bZ^4$) and then Cauchy-Schwarz (for the sum over $\ell$) the previous line can be estimated by
 \begin{align*}
  &\sum_{j \geq 0} 2^{\alpha j} \biggl( \sum_{k \in \bZ^4, |k| > 100} |k|^{-10} \biggr) \Biggl( \sum_{m \in \bZ^4} \biggl( \sum_{\ell > j+5} 2^{-10 \ell} \bigl\| \widetilde{P}_m \widetilde{\chi}_\ell g_N \bigr\|_{L^2_t L^2_x} \biggr)^2 \Biggr)^{\frac{1}{2}} \\
  &\lesssim \sum_{j \geq 0} 2^{\alpha j} \Biggl( \sum_{m \in \bZ^4} \sum_{\ell > j+5} 2^{-10 \ell} \bigl\| \widetilde{P}_m \widetilde{\chi}_\ell g_N \bigr\|_{L^2_t L^2_x}^2 \Biggr)^{\frac{1}{2}} \\
  &\lesssim \sum_{j \geq 0} 2^{\alpha j} \Biggl( \sum_{\ell > j+5} 2^{-10 \ell} \sum_{m \in \bZ^4} \bigl\| \widetilde{P}_m \widetilde{\chi}_\ell g_N \bigr\|_{L^2_t L^2_x}^2 \Biggr)^{\frac{1}{2}}.
 \end{align*}
 Finally, we use that the projections $\widetilde{P}_m$, $m \in \bZ^4$, are just slight fattenings of the unit-scale projections~$P_m$, which constitute a finitely overlapping partition of unity of frequency space, and invoke the local smoothing estimate~\eqref{equ:local_smoothing_ball} to obtain that the last line is bounded by
 \begin{align*}
  &\sum_{j \geq 0} 2^{\alpha j} \Biggl( \sum_{\ell > j+5} 2^{-M \ell} \bigl\| \widetilde{\chi}_\ell g_N \bigr\|_{L^2_t L^2_x}^2 \Biggr)^{\frac{1}{2}} \lesssim \sum_{j \geq 0} 2^{\alpha j} \Biggl( \sum_{\ell > j+5} 2^{-M \ell} 2^{\ell} \bigl\| |\nabla|^{+\frac{1}{2}} f_N \bigr\|_{L^2_x}^2 \Biggr)^{\frac{1}{2}} \lesssim \bigl\| |\nabla|^{+\frac{1}{2}} f_N \bigr\|_{L^2_x}.
 \end{align*}
 
 This finishes the treatment of the high-frequency terms on the right-hand side of~\eqref{equ:weighted_nabla_L2Linfty_first_decomposition}. The first low-frequency term on the right-hand side of~\eqref{equ:weighted_nabla_L2Linfty_first_decomposition} can be estimated analogously and the details are left to the reader, we obtain for all sufficiently large $p < \infty$ that
 \[
  \bigl\| \langle x \rangle^\alpha \nabla e^{i t \Delta} P_{\leq 1} f^\omega \bigr\|_{L^p_\omega L^2_t L^r_x} \lesssim \sqrt{p} \, \bigl\| |\nabla|^{+\frac{1}{2}} P_{\leq 1} f \bigr\|_{L^2_x}.
 \]
 
 Thus, putting all of the above estimates together, we conclude that
 \begin{equation} \label{equ:final_estimate_weighted_nabla_L2Linfty_schroedinger}
  \begin{aligned}
   \bigl\| \langle x \rangle^\alpha \nabla e^{i t \Delta} f^\omega \bigr\|_{L^p_\omega L^2_t L^\infty_x} &\lesssim \sqrt{p} \, \bigl\| \langle x \rangle^\alpha \nabla e^{i t \Delta} P_{\leq 1} f^\omega \bigr\|_{L^p_\omega L^2_t L^r_x} + \sqrt{p} \sum_{N \geq 2} N^{\frac{4}{r}} N^\delta \bigl\| |\nabla|^{+\frac{1}{2}} P_N f \bigr\|_{L^2_x} \\
   &\lesssim \sqrt{p} \, \bigl\| f \bigr\|_{H^s_x}
  \end{aligned}
 \end{equation}
 for all $p \geq r$ for some sufficiently large $r < \infty$ and for some sufficiently small $\delta > 0$ (such that $\alpha + \frac{3}{r} < 1$ and $\frac{4}{r} + \delta < s$). The claim now follows from Lemma~\ref{lem:probability_estimate}.
\end{proof}

The next almost sure bound is an immediate consequence of the previous Proposition~\ref{prop:weighted_nabla_L2Linfty_schroedinger} and the local smoothing estimate from Proposition~\ref{prop:local_smoothing}.

\begin{proposition} 
 Let $\frac{1}{2} < s < 1$. Let $f \in H^s_x(\bR^4)$ be radially symmetric and denote by $f^\omega$ the randomization of $f$ as defined in~\eqref{equ:randomization}. Then there exist absolute constants $C > 0$ and $c > 0$ such that for any $\lambda > 0$ it holds that
 \begin{equation} 
  \bP \Bigl( \Bigl\{ \omega \in \Omega : \bigl\| \nabla e^{i t \Delta} f^\omega \bigr\|_{L^2_t L^4_x(\bR\times\bR^4)} > \lambda  \Bigr\} \Bigr) \leq C \exp \Bigl( - c \lambda^2 \| f \|_{H^s_x(\bR^4)}^{-2} \Bigr).
 \end{equation}
 In particular, we have for almost every $\omega \in \Omega$ that
 \begin{equation}
  \bigl\| \nabla e^{i t \Delta} f^\omega \bigr\|_{L^2_t L^4_x(\bR\times\bR^4)} < \infty.
 \end{equation} 
\end{proposition}
\begin{proof}
 In the following all space-time norms are taken over $\bR \times \bR^4$. By H\"older's inequality we have for any $p \geq 1$ that
 \begin{equation} \label{equ:nabla_L2L4_one}
  \bigl\| \nabla e^{i t \Delta} f^\omega \bigr\|_{L^p_\omega L^2_t L^4_x} \leq \bigl\| \langle x \rangle^{-\frac{3}{4}} \nabla e^{i t \Delta} f^\omega \bigr\|_{L^p_\omega L^2_t L^2_x}^{\frac{1}{2}} \bigl\| \langle x \rangle^{\frac{3}{4}} \nabla e^{i t \Delta} f^\omega \bigr\|_{L^p_\omega L^2_t L^\infty_x}^{\frac{1}{2}}.
 \end{equation}
 Then the local smoothing estimate~\eqref{equ:local_smoothing_weight} and the large deviation estimate from Lemma~\ref{lem:large_deviation_estimate} imply for all $p \geq 2$ that
 \begin{equation} \label{equ:nabla_L2L4_two}
  \bigl\| \langle x \rangle^{-\frac{3}{4}} \nabla e^{i t \Delta} f^\omega \bigr\|_{L^p_\omega L^2_t L^2_x} \lesssim \bigl\| |\nabla|^{\frac{1}{2}} f^\omega \bigr\|_{L^p_\omega L^2_x} \lesssim \sqrt{p} \, \bigl\| |\nabla|^{\frac{1}{2}} f \bigr\|_{L^2_x} \lesssim \sqrt{p} \, \| f \|_{H^s_x}.
 \end{equation}
 Moreover, the estimate~\eqref{equ:final_estimate_weighted_nabla_L2Linfty_schroedinger} from the proof of Proposition~\ref{prop:weighted_nabla_L2Linfty_schroedinger} yields for all sufficiently large $p < \infty$ that
 \begin{equation} \label{equ:nabla_L2L4_three}
  \bigl\| \langle x \rangle^{\frac{3}{4}} \nabla e^{i t \Delta} f^\omega \bigr\|_{L^p_\omega L^2_t L^\infty_x} \lesssim \sqrt{p} \, \| f \|_{H^s_x}.
 \end{equation}
 Combining \eqref{equ:nabla_L2L4_one}--\eqref{equ:nabla_L2L4_three} we obtain for all sufficiently large $p < \infty$ that
 \begin{equation*}
  \bigl\| \nabla e^{i t \Delta} f^\omega \bigr\|_{L^p_\omega L^2_t L^4_x} \lesssim \sqrt{p} \, \|f\|_{H^s_x}
 \end{equation*}
 and the assertion follows from Lemma~\ref{lem:probability_estimate}.
\end{proof}

We will also need almost sure bounds on weighted $L^2_t L^\infty_x(\bR\times\bR^4)$ norms of the free evolution of randomized radial data. The proof combines Strichartz estimates~\eqref{equ:strichartz_estimate} and the ``radialish'' Sobolev estimate from Lemma~\ref{lem:radialish_sobolev}.

\begin{proposition} 
 Let $0 < s < 1$ and let $0 \leq \alpha < \frac{3}{4}$. Let $f \in H^s_x(\bR^4)$ be radially symmetric and denote by $f^\omega$ the randomization of $f$ as defined in~\eqref{equ:randomization}. Then there exist absolute constants $C > 0$ and $c > 0$ such that for any $\lambda > 0$ it holds that
 \begin{equation} 
  \bP \Bigl( \Bigl\{ \omega \in \Omega : \bigl\| \langle x \rangle^\alpha e^{i t \Delta} f^\omega \bigr\|_{L^2_t L^\infty_x(\bR\times\bR^4)} > \lambda  \Bigr\} \Bigr) \leq C \exp \Bigl( - c \lambda^2 \| f \|_{H^s_x(\bR^4)}^{-2} \Bigr).
 \end{equation}
 In particular, we have for almost every $\omega \in \Omega$ that
 \begin{equation}
  \bigl\| \langle x \rangle^\alpha e^{i t \Delta} f^\omega \bigr\|_{L^2_t L^\infty_x(\bR\times\bR^4)} < \infty.
 \end{equation}
\end{proposition}
\begin{proof}
 As usual, in the following all space-time norms are taken over $\bR\times\bR^4$. For any $4 < r < \infty$ we have by Lemma~\ref{lem:weighted_bound_into_freq_pieces} and the elementary estimate $\langle x \rangle^\alpha \lesssim_\alpha 1 + |x|^\alpha$ that
 \begin{equation} \label{equ:weighted_L2Linfty_first_decomposition}
  \begin{aligned}
   \bigl\| \langle x \rangle^\alpha e^{it\Delta} f^\omega \bigr\|_{L^p_\omega L^2_t L^\infty_x} &\lesssim \bigl\| e^{it\Delta} P_{\leq 1} f^\omega \bigr\|_{L^p_\omega L^2_t L^r_x} + \sum_{N \geq 2} N^{\frac{4}{r}} \bigl\| e^{it\Delta} P_N f^\omega \bigr\|_{L^p_\omega L^2_t L^r_x} \\
   &\quad + \bigl\| |x|^\alpha e^{it\Delta} P_{\leq 1} f^\omega \bigr\|_{L^p_\omega L^2_t L^r_x} + \sum_{N \geq 2} N^{\frac{4}{r}} \bigl\| |x|^\alpha e^{it\Delta} P_N f^\omega \bigr\|_{L^p_\omega L^2_t L^r_x}.
  \end{aligned}
 \end{equation}
 In what follows we only estimate the weighted high-frequency terms on the second line of the right-hand side of~\eqref{equ:weighted_L2Linfty_first_decomposition}. The weighted low-frequency term on the second line of the right-hand side of~\eqref{equ:weighted_L2Linfty_first_decomposition} can be treated analogously and the large deviation estimates for the space-time norms in the first line of the right-hand side of~\eqref{equ:weighted_L2Linfty_first_decomposition} are standard and left to the reader. 
  
 For any dyadic integer $N \geq 2$ we have by the large deviation estimate from Lemma~\ref{lem:large_deviation_estimate} for any $r \leq p < \infty$ that 
 \begin{equation}
  \bigl\| |x|^\alpha e^{it\Delta} P_N f^\omega \bigr\|_{L^p_\omega L^2_t L^r_x} \lesssim \sqrt{p} \, \biggl\| |x|^\alpha \biggl( \sum_{k\in\bZ^4} \bigl| e^{it\Delta} P_k P_N f \bigr|^2 \biggr)^{\frac{1}{2}} \biggr\|_{L^2_t L^r_x}.
 \end{equation}
 Now let $0 < \delta \ll 1$ be an absolute constant whose size will be fixed sufficiently small further below. Interpolating the ``radialish'' Sobolev estimate~\eqref{equ:radialish_sobolev}
 \begin{equation*}
  \biggl\| |x|^{\frac{3}{2}} \biggl( \sum_{k\in\bZ^4} \bigl| e^{it\Delta} P_k P_N f \bigr|^2 \biggr)^{\frac{1}{2}} \biggr\|_{L^\infty_x} \lesssim_{\alpha, \delta} N^{\frac{3 \delta}{2\alpha}} \biggl( \sum_{k\in\bZ^4} \bigl\| e^{it\Delta} P_k P_N f \bigr\|_{L^2_x}^2 \biggr)^{\frac{1}{2}}
 \end{equation*}
 with the trivial bound 
 \begin{equation*}
  \biggl\| \biggl( \sum_{k\in\bZ^4} \bigl| e^{it\Delta} P_k P_N f \bigr|^2 \biggr)^{\frac{1}{2}} \biggr\|_{L^{\frac{(3-2\alpha)r}{3}}_x} \lesssim \biggl( \sum_{k\in\bZ^4} \bigl\| e^{it\Delta} P_k P_N f \bigr\|_{L^{\frac{(3-2\alpha)r}{3}}_x}^2 \biggr)^{\frac{1}{2}}
 \end{equation*}
 yields that 
 \begin{equation} \label{equ:weighted_L2Linfty_interpolated}
  \biggl\| |x|^{\alpha} \biggl( \sum_{k\in\bZ^4} \bigl| e^{it\Delta} P_k P_N f \bigr|^2 \biggr)^{\frac{1}{2}} \biggr\|_{L^2_t L^r_x} \lesssim N^{\delta} \biggl( \sum_{k\in\bZ^4} \bigl\| e^{it\Delta} P_k P_N f \bigr\|_{L^2_t L^{\frac{3r}{3 + \alpha r}}_x}^2 \biggr)^{\frac{1}{2}}.
 \end{equation}
 Since by assumption $0 \leq \alpha < \frac{3}{4}$, we have for all sufficiently large $r < \infty$ that $\frac{3r}{3+\alpha r} \geq 4$. Hence, we may combine the unit-scale Bernstein estimate~\eqref{equ:unit_scale_bernstein} and the Strichartz estimates~\eqref{equ:strichartz_estimate} to bound the right-hand side of~\eqref{equ:weighted_L2Linfty_interpolated} by
 \begin{align*}
  N^\delta \biggl( \sum_{k\in\bZ^4} \bigl\| e^{it\Delta} P_k P_N f \bigr\|_{L^2_t L^4_x}^2 \biggr)^{\frac{1}{2}} \lesssim N^\delta \biggl( \sum_{k\in\bZ^4} \bigl\| P_k P_N f \bigr\|_{L^2_x}^2 \biggr)^{\frac{1}{2}} \lesssim N^\delta \| P_N f \|_{L^2_x}.
 \end{align*}
 Thus, choosing $1 \ll r < \infty$ sufficiently large and $0 < \delta \ll 1$ sufficiently small so that $\frac{4}{r} + \delta < s$, we conclude that the second line of the right-hand side of~\eqref{equ:weighted_L2Linfty_first_decomposition} is bounded by
 \[
  \sqrt{p} \, \| P_{\leq 1} f \|_{L^2_x} + \sqrt{p} \sum_{N \geq 2} N^{\frac{4}{r}} N^\delta \| P_N f \|_{L^2_x} \lesssim \sqrt{p} \, \|f\|_{H^s_x}.
 \]
 The assertion then follows from Lemma~\ref{lem:probability_estimate}.
\end{proof}

Finally, we will require several almost sure bounds on space-time norms of the free evolution whose proofs are standard.
\begin{lemma}
 Let $s > 0$ and let $f \in H^s_x(\bR^4)$. Denote by $f^\omega$ the randomization of $f$ as defined in~\eqref{equ:randomization}. Then we have for almost every $\omega \in \Omega$ that
 \begin{equation}
  \bigl\| e^{i t \Delta} f^\omega \bigr\|_{L^\infty_t L^4_x(\bR\times\bR^4)} + \bigl\| e^{i t \Delta} f^\omega \bigr\|_{L^\infty_t L^2_x(\bR\times\bR^4)} + \bigl\| e^{i t \Delta} f^\omega \bigr\|_{L^3_t L^6_x(\bR\times\bR^4)}< \infty.
 \end{equation}
\end{lemma}
\begin{proof}
 In order to establish the $L^\infty_t L^4_x(\bR\times\bR^4)$ almost sure bound on the free evolution, we first apply Sobolev embedding in time. Let $0 < \delta \ll 1$ with $0 < \delta < \frac{s}{2}$ and let $2 \leq q < \infty$ sufficiently large such that $\delta > \frac{1}{q}$, then we have 
 \[
  \bigl\| e^{it\Delta} f^\omega \bigr\|_{L^\infty_t L^4_x(\bR\times\bR^4)} \lesssim \bigl\| \langle \partial_t \rangle^\delta e^{it\Delta} f^\omega \bigr\|_{L^q_t L^4_x(\bR\times\bR^4)} \lesssim \bigl\| \langle \nabla \rangle^{2\delta} e^{it\Delta} f^\omega \bigr\|_{L^q_t L^4_x(\bR\times\bR^4)}.
 \]
 It is now standard to use the large deviation estimate from Lemma~\ref{lem:large_deviation_estimate}, the unit-scale Bernstein estimate~\eqref{equ:unit_scale_bernstein} and the Strichartz estimate~\eqref{equ:strichartz_estimate} to infer for all $p \geq q$ that
 \[
  \bigl\| \|e^{it\Delta} f^\omega\|_{L^\infty_t L^4_x(\bR\times\bR^4)} \bigr\|_{L^p_\omega} \lesssim \sqrt{p} \, \|f\|_{H^s_x(\bR^4)},
 \]
 which implies the desired almost sure bound by Lemma~\ref{lem:probability_estimate}. The proofs of the $L^\infty_t L^2_x(\bR\times\bR^4)$ and of the $L^3_t L^6_x(\bR\times\bR^4)$ almost sure bounds are left to the reader.
\end{proof}

\section{Almost sure local well-posedness for the cubic NLS on $\bR^4$} \label{sec:as_lwp}

In this section we establish local well-posedness for the forced cubic NLS
\begin{equation} \label{equ:forced_cubic_nls_lwp}
 \left\{ \begin{aligned}
          (i \partial_t + \Delta) v &= \pm |F+v|^2 (F+v) \\
          v(t_0) &= v_0 \in \dot{H}^1_x(\bR^4)
         \end{aligned} 
 \right.
\end{equation}
for forcing terms $F \colon \bR \times \bR^4 \to \bC$ satisfying $\| F \|_{Y(\bR)} < \infty$. Recall that by Proposition~\ref{prop:as_bound_Y_norm} we have $\| e^{it\Delta} f^\omega \|_{Y(\bR)} < \infty$ almost surely for any $f \in H^s_x(\bR^4)$ with $\frac{1}{3} < s < 1$. The proof of the almost sure local well-posedness result of Theorem~\ref{thm:as_local_wellposedness} for the cubic NLS \eqref{equ:cubic_nls_as_lwp_theorem} is then an immediate consequence of the following local well-posedness result for the forced cubic NLS~\eqref{equ:forced_cubic_nls_lwp}.

\begin{proposition} \label{prop:lwp_forced_cubic_nls}
 Let $t_0 \in \bR$ and let $I$ be an open time interval containing $t_0$. Let $F \in Y(\bR)$ and let $v_0 \in \dot{H}^1_x(\bR^4)$. There exists $0 < \delta \ll 1$ such that if
 \begin{equation} \label{equ:smallness_lwp}
  \| e^{i(t-t_0)\Delta} v_0 \|_{X(I)} + \| F \|_{Y(I)} \leq \delta,
 \end{equation}
 then there exists a unique solution 
 \[
  v \in C \bigl( I ; \dot{H}^1_x(\bR^4) \bigr) \cap X(I)
 \]
 to \eqref{equ:forced_cubic_nls_lwp} on $I \times \bR^4$. Moreover, the solution extends to a unique solution $v \colon I_\ast \times \bR^4 \to \bC$ to the Cauchy problem~\eqref{equ:forced_cubic_nls_lwp} with maximal time interval of existence $I_\ast \ni t_0$, and we have the finite time blowup criterion
 \[
  \sup I_\ast < \infty \quad \Rightarrow \quad \|v\|_{X([t_0, \sup I_\ast))} = +\infty
 \]
 with an analogous statement in the negative time direction. Finally, a global solution $v(t)$ to \eqref{equ:forced_cubic_nls_lwp} satisfying $\|v\|_{X(\bR)} < \infty$ scatters as $t \to \pm \infty$ in the sense that there exist states $v^\pm \in \dot{H}^1_x(\bR^4)$ such that 
 \[
  \lim_{t\to\pm\infty} \, \bigl\| v(t) - e^{it\Delta} v^{\pm} \bigr\|_{\dot{H}^1_x(\bR^4)} = 0.
 \]
\end{proposition}
\begin{proof}
 Without loss of generality we may assume that $t_0 = 0$. Let $I \ni 0$ be an open time interval for which \eqref{equ:smallness_lwp} holds. Note that the existence of such an interval follows from Lemma~\ref{lem:properties_XY}(i) and the assumption that $\|F\|_{Y(\bR)} < \infty$. We construct the desired local solution via a standard contraction mapping argument. Let $\delta > 0$ be an absolute constant whose size will be chosen sufficiently small further below. We define the ball
 \[
  {\mathcal B} := \bigl\{ v \in X(I) : \|v\|_{X(I)} \leq 2 \delta \bigr\}
 \]
 and the map
 \[
  \Phi(v)(t) := e^{it\Delta} v_0 \mp i \int_0^t e^{i(t-s)\Delta} |F+v|^2 (F+v)(s) \, ds. 
 \]
 From our main linear estimate~\eqref{equ:main_linear_estimate} and the nonlinear estimates \eqref{equ:forced_cubic_nonlinearity_in_G}--\eqref{equ:difference_forced_cubic_nonlinearity_in_G}, upon choosing $\delta := (18C)^{-\frac{1}{2}}$, we obtain for any $v, v_1, v_2 \in {\mathcal B}$ that
 \begin{align*}
  \| \Phi(v) \|_{X(I)} &\leq \| e^{it\Delta} v_0 \|_{X(I)} + \biggl\| \int_0^t e^{i(t-s)\Delta} |F+v|^2(F+v)(s) \, ds \biggr\|_{X(I)} \\
  &\leq \| e^{it\Delta} v_0 \|_{X(I)} + C \bigl\| |F+v|^2(F+v) \bigr\|_{G(I)} \\
  &\leq \| e^{it\Delta} v_0 \|_{X(T)} + C \bigl( \|v\|_{X(I)}^3 + \|F\|_{Y(I)}^3 \bigr) \\
  &\leq 2 \delta
 \end{align*}
 and 
 \begin{align*}
  \| \Phi(v_1) - \Phi(v_2) \|_{X(I)} &\leq C \bigl\| |F+v_1|^2 (F+v_1) - |F+v_2|^2 (F+v_2) \bigr\|_{G(I)} \\
  &\leq C \|v_1 - v_2\|_{X(I)} \bigl( \|v_1\|_{X(I)}^2 + \|v_2\|_{X(I)}^2 + \|F\|_{Y(I)}^2 \bigr) \\
  &\leq \frac{1}{2} \|v_1 - v_2\|_{X(I)}.
 \end{align*}
 It follows that the map $\Phi \colon {\mathcal B} \to {\mathcal B}$ is a contraction with respect to the $X(I)$ norm and we infer the existence of a unique solution $v \in C(I; \dot{H}^1_x(\bR^4)) \cap X(I)$ to \eqref{equ:forced_cubic_nls_lwp}.
 
 By iterating this local well-posedness argument we conclude that the solution extends to a unique solution $v \colon I_\ast \times \bR^4 \to \bC$ to \eqref{equ:forced_cubic_nls_lwp} with maximal time interval of existence $I_\ast \ni 0$. We now prove the finite time blowup criterion via a contradiction argument. Let $T_+ := \sup I_\ast < \infty$ and suppose that $\|v\|_{X([0,T_+))} < \infty$. We want to find a time $0 < t_1 < T_+$ such that
 \begin{equation} \label{equ:finite_time_blowup_criterion_proof_smallness}
  \bigl\| e^{i(t-t_1)\Delta} v(t_1) \bigr\|_{X([t_1,T_+))} + \|F\|_{Y([t_1,T_+))} \leq \frac{\delta}{2}.
 \end{equation}
 Since $\|F\|_{Y(\bR)} < \infty$ and $\| e^{i(t-t_1)\Delta} v(t_1) \|_{X([t_1,\infty))} < \infty$, Lemma~\ref{lem:properties_XY}(i) then implies that there exists $\eta > 0$ such that
 \[
  \bigl\| e^{i(t-t_1)\Delta} v(t_1) \bigr\|_{X([t_1,T_+ + \eta))} + \|F\|_{Y([t_1,T_+ + \eta))} \leq \delta.
 \]
 But then the above local well-posedness result implies that the solution $v(t)$ extends beyond time $T_+ = \sup I_\ast$, which is a contradiction. Now to prove~\eqref{equ:finite_time_blowup_criterion_proof_smallness} we use the Duhamel formula for the solution $v(t)$ to write
 \[
  e^{i(t-t_1)\Delta} v(t_1) = v(t) \pm i \int_{t_1}^t e^{i(t-s)\Delta} |F+v|^2(F+v)(s) \, ds.
 \]
 Then our main linear estimate~\eqref{equ:main_linear_estimate} together with \eqref{equ:forced_cubic_nonlinearity_in_G} imply that
 \begin{align*}
  \bigl\| e^{i(t-t_1)\Delta} v(t_1) \bigr\|_{X([t_1,T_+)} &\leq \|v\|_{X([t_1,T_+))} + \biggl\| \int_{t_1}^t e^{i(t-s)\Delta} |F+v|^2(F+v)(s) \, ds \biggr\|_{X([t_1,T_+))} \\
  &\leq \|v\|_{X([t_1,T_+))} + C \bigl( \|F\|_{Y([t_1,T_+))}^3 + \|v\|_{X([t_1,T_+))}^3 \bigr).
 \end{align*}
 Using Lemma~\ref{lem:properties_XY}(i) as well as the assumptions $\|v\|_{X([0,T_+))} < \infty$ and $\|F\|_{Y(\bR)} < \infty$, we may conclude that $\|v\|_{X([t_1,T_+))} \to 0$ and $\|F\|_{Y([t_1, T_+))} \to 0$ as $t_1 \nearrow T_+$, which yields \eqref{equ:finite_time_blowup_criterion_proof_smallness}.
 
 Finally we turn to the proof of the scattering statement for a global solution $v(t)$ to \eqref{equ:forced_cubic_nls_lwp} satisfying $\|v\|_{X(\bR)} < \infty$. By similar arguments as above we infer that the scattering state in the positive time direction
 \[
  v^+ := v_0 \mp i \int_0^\infty e^{-is\Delta} |F+v|^2 (F+v)(s) \, ds
 \]
 belongs to $\dot{H}^1_x(\bR^4)$ and satisfies $\| v(t) - e^{it\Delta} v^+ \|_{\dot{H}^1_x(\bR^4)} \to 0$ as $t \to \infty$. An analogous argument holds for the negative time direction.
\end{proof}
The proof of Theorem~\ref{thm:as_local_wellposedness} is now an immediate consequence of the local well-posedness result from Proposition~\ref{prop:lwp_forced_cubic_nls} for the forced cubic nonlinear Schr\"odinger equation~\eqref{equ:forced_cubic_nls_lwp} and the almost sure bounds on the $Y(\bR)$ norm of the free evolution $e^{it\Delta} f^\omega$ of the random data established in Proposition~\ref{prop:as_bound_Y_norm}.
\begin{proof}[Proof of Theorem~\ref{thm:as_local_wellposedness}]
 We seek a solution to the cubic NLS~\eqref{equ:cubic_nls_as_lwp_theorem} of the form 
 \[
  u(t) = e^{it\Delta} f^\omega + v(t).
 \]
 To this end the nonlinear component $v(t)$ must be a solution to the following forced cubic NLS
 \begin{equation} \label{equ:proof_of_as_lwp_theorem_forced_cubic}
  (i\partial_t + \Delta) v = \pm |e^{it\Delta} f^\omega + v|^2 (e^{it\Delta} f^\omega + v)
 \end{equation}
 with zero initial data $v(0) = 0$. By Proposition~\ref{prop:as_bound_Y_norm} we have $\| e^{it\Delta} f^\omega \|_{Y(\bR)} < \infty$ for almost every $\omega \in \Omega$. Thus, by Lemma~\ref{lem:properties_XY}(i), for almost every $\omega \in \Omega$ there exists an interval $I^\omega$ such that $\|e^{it\Delta} f^\omega\|_{Y(I^\omega)} \leq \delta$, where $0 < \delta \ll 1$ is the small absolute constant from the statement of Proposition~\ref{prop:lwp_forced_cubic_nls}. Consequently, by Proposition~\ref{prop:lwp_forced_cubic_nls} there exists a unique solution $v \in C(I^\omega; \dot{H}^1_x(\bR^4)) \cap X(I^\omega)$ to \eqref{equ:proof_of_as_lwp_theorem_forced_cubic} for almost every $\omega \in \Omega$. This concludes the proof of Theorem~\ref{thm:as_local_wellposedness}.
\end{proof}

\section{Conditional scattering for the forced defocusing cubic NLS on $\bR^4$} \label{sec:conditional_scattering}

In this section we prove the conditional scattering result of Theorem~\ref{thm:scattering_conditional} for the forced defocusing cubic NLS~\eqref{equ:forced_cubic_nls_scattering}. The proof relies on a suitable perturbation theory to compare solutions to the forced defocusing cubic NLS
\begin{equation} \label{equ:forced_cubic_nls_scattering_section}
 \left\{ \begin{aligned}
  (i \partial_t + \Delta) v &= |F+v|^2 (F+v) \\
  v(t_0) &= v_0 \in \dot{H}^1_x(\bR^4).
 \end{aligned} \right.
\end{equation}
with solutions to the ``usual'' defocusing cubic NLS 
\begin{equation} \label{equ:cubic_nls_scattering_section}
 \left\{ \begin{aligned}
  (i \partial_t + \Delta) u &= |u|^2 u \\
  u(t_0) &= u_0 \in \dot{H}^1_x(\bR^4).
 \end{aligned} \right.
\end{equation}
We begin with an a priori estimate on the $X(\bR)$ norm of global solutions to the defocusing cubic NLS~\eqref{equ:cubic_nls_scattering_section}.
\begin{lemma} \label{lem:standard_cubic}
 There exists a non-decreasing function $K \colon [0,\infty) \to [0,\infty)$ with the following property. Let $u_0 \in \dot{H}^1_x(\bR^4)$ and $t_0 \in \bR$. Then there exists a unique global solution $u \in C\bigl(\bR; \dot{H}^1_x(\bR^4)\bigr)$ to the defocusing cubic NLS~\eqref{equ:cubic_nls_scattering_section} satisfying the a priori bound 
 \begin{equation*}
  \|u\|_{X(\bR)} \leq K(E(u_0)),
 \end{equation*}
 where 
 \[
  E(u_0) := \int_{\bR^4} \frac{1}{2} |\nabla u_0|^2 + \frac{1}{4} |u_0|^4 \, dx.
 \]
\end{lemma}
\begin{proof}
 It follows from the work of Ryckman-Visan~\cite{Ryckman_Visan} and Visan~\cite{Visan} that there exists a non-decreasing function $L \colon [0,\infty) \to [0,\infty)$ such that for any $u_0 \in \dot{H}^1_x(\bR^4)$, there exists a unique global solution $u \in C\bigl(\bR; \dot{H}^1_x(\bR^4)\bigr)$ to the defocusing cubic NLS~\eqref{equ:cubic_nls_scattering_section} with initial data $u(t_0) = u_0$ satisfying the a priori bound 
 \begin{equation*}
  \|\nabla u\|_{L^3_t L^3_x(\bR\times\bR^4)} \leq L(E(u_0)).
 \end{equation*}
 Using the linear estimate~\eqref{equ:main_linear_estimate} we then find that 
 \begin{align*}
  \|u\|_{X(\bR)} &\lesssim \|u_0\|_{\dot{H}^1_x(\bR^4)} + \bigl\| |u|^2 u \bigr\|_{G(\bR)} \\
  &\lesssim E(u_0)^{\frac{1}{2}} + \biggl( \sum_{N \in 2^{\bZ}} N^2 \bigl\| P_N \bigl( |u|^2 u \bigr) \bigr\|^2_{L^1_t L^2_x(\bR\times\bR^4)} \biggr)^{\frac{1}{2}} \\
  &\lesssim E(u_0)^{\frac{1}{2}} + \bigl\| \nabla \bigl( |u|^2 u \bigr) \bigr\|_{L^1_t L^2_x(\bR\times\bR^4)} \\
  &\lesssim E(u_0)^{\frac{1}{2}} + \|\nabla u\|_{L^3_t L^3_x(\bR\times\bR^4)} \|u\|_{L^3_t L^{12}_x(\bR\times\bR^4)}^2 \\
  &\lesssim E(u_0)^{\frac{1}{2}} + \|\nabla u\|_{L^3_t L^3_x(\bR\times\bR^4)}^3 \\
  &\lesssim E(u_0)^{\frac{1}{2}} + L(E(u_0))^3,
 \end{align*}
 which implies the assertion.
\end{proof}

Next we develop a suitable perturbation theory to compare solutions to the forced defocusing cubic NLS~\eqref{equ:forced_cubic_nls_scattering_section} with solutions to the defocusing cubic NLS~\eqref{equ:cubic_nls_scattering_section}. The proof proceeds along pre-existing lines using the linear estimate~\eqref{equ:main_linear_estimate}, the nonlinear estimate~\eqref{equ:difference_equation_nonlinearity_in_G} and the time divisibility properties~\eqref{equ:time_divisibility} of the $X(I)$ and $Y(I)$ spaces as key ingredients.
\begin{lemma}(Short-time perturbations) \label{lem:short_time_perturbations}
 Let $I \subset \bR$ be a time interval with $t_0 \in I$ and let $v_0, u_0 \in \dot{H}^1_x(\bR^4)$. There exist small absolute constants $0 < \delta \ll 1$ and $0 < \eta_0 \ll 1$ with the following properties. Let $u \colon I \times \bR^4 \to \bC$ be the solution to~\eqref{equ:cubic_nls_scattering_section} with initial data $u(t_0) = u_0$ satisfying 
 \begin{equation}
  \|u\|_{X(I)} \leq \delta
 \end{equation}
 and let $F \colon I \times \bR^4 \to \bC$ be a forcing term such that 
 \begin{equation}
  \|F\|_{Y(I)} \leq \eta 
 \end{equation}
 for some $0 < \eta \leq \eta_0$. Suppose also that 
 \begin{equation}
  \|v_0 - u_0\|_{\dot{H}^1_x(\bR^4)} \leq \eta_0.
 \end{equation}
 Then there exists a unique solution $v \colon I \times \bR^4 \to \bC$ to~\eqref{equ:forced_cubic_nls_scattering_section} with initial data $v(t_0) = v_0$ and we have 
 \begin{equation} \label{equ:short_time_perturbation_difference_bound} 
  \|v-u\|_{L^\infty_t \dot{H}^1_x(I\times\bR^4)} + \|v-u\|_{X(I)} \leq C_0 \bigl( \|v_0 - u_0\|_{\dot{H}^1_x(\bR^4)} + \eta \bigr)
 \end{equation}
 for some absolute constant $C_0 \geq 1$.
\end{lemma}
\begin{proof}
 In view of the local existence theory from Proposition~\ref{prop:lwp_forced_cubic_nls} it suffices to establish~\eqref{equ:short_time_perturbation_difference_bound} as an a priori estimate. We define $w := v-u$ and observe that $w$ is a solution to the difference equation 
 \begin{equation*}
  \left\{ \begin{aligned}
   (i \partial_t + \Delta) w &= |F+u+w|^2(F+u+w) - |u|^2 u \text{ on } I \times \bR^4, \\
   w(t_0) &= v_0 - u_0. 
  \end{aligned} \right.
 \end{equation*}
 By the linear estimate~\eqref{equ:main_linear_estimate} and the nonlinear estimate~\eqref{equ:difference_equation_nonlinearity_in_G}, we find that
 \begin{align*}
  &\|w\|_{L^\infty_t \dot{H}^1_x(I\times\bR^4)} + \|w\|_{X(I)} \\
  &\quad \lesssim \|v_0-u_0\|_{\dot{H}^1_x(\bR^4)} + \|F\|_{Y(I)}^3 + \|w\|_{X(I)}^3 + \|u\|_{X(I)}^2 \|F\|_{Y(I)} + \|u\|_{X(I)}^2 \|w\|_{X(I)} \\
  &\quad \lesssim \|v_0-u_0\|_{\dot{H}^1_x(\bR^4)} + \eta^3 + \|w\|_{X(I)}^3 + \delta^2 \eta + \delta^2 \|w\|_{X(I)}.
 \end{align*}
 The assertion now follows from a standard continuity argument.
\end{proof}

\begin{lemma}(Long-time perturbations) \label{lem:long_time_perturbations}
 Let $I \subset \bR$ be a time interval with $t_0 \in I$ and let $v_0 \in \dot{H}^1_x(\bR^4)$. Let $u \colon I \times \bR^4 \to \bC$ be the solution to~\eqref{equ:cubic_nls_scattering_section} with initial data $u(t_0) = v_0$ satisfying 
 \begin{equation}
  \|u\|_{X(I)} \leq K.
 \end{equation}
 Then there exists $0 < \eta_1(K) \ll 1$ such that for any forcing term $F \colon I \times \bR^4 \to \bC$ satisfying 
 \begin{equation}
  \|F\|_{Y(I)} \leq \eta 
 \end{equation}
 for some $0 < \eta \leq \eta_1(K)$, there exists a unique solution $v \colon I \times \bR^4 \to \bC$ to~\eqref{equ:forced_cubic_nls_scattering_section} with initial data $v(t_0) = v_0$ and it holds that 
 \begin{equation} \label{equ:long_time_perturbation_difference_bound}
  \|v-u\|_{L^\infty_t \dot{H}^1_x(I\times\bR^4)} + \|v-u\|_{X(I)} \lesssim \exp \bigl( C_1 K^{\frac{4}{\ve}} \bigr) \eta 
 \end{equation}
 for some absolute constant $C_1 \gg 1$. In particular, it holds that 
 \begin{equation} \label{equ:long_time_perturbation_difference_bound_nonsharp}
  \|v-u\|_{L^\infty_t \dot{H}^1_x(I\times\bR^4)} + \|v-u\|_{X(I)} \lesssim 1.
 \end{equation}
\end{lemma}
\begin{proof}
 We may assume without loss of generality that $t_0 = \inf I$. Moreover, it again suffices to establish~\eqref{equ:long_time_perturbation_difference_bound} as an a priori estimate in view of the local existence theory from Proposition~\ref{prop:lwp_forced_cubic_nls}. We first use the time divisibility property of the $X(I)$ norm, see Lemma~\ref{lem:properties_XY}, to partition the interval~$I$ into $J \equiv J(K)$ consecutive intervals $I_j$, $j = 1, \ldots, J$, with disjoint interiors such that
 \begin{equation*}
  \|u\|_{X(I_j)} \leq \delta  
 \end{equation*}
 for $j = 1, \ldots, J$, where $0 < \delta \ll 1$ is the absolute constant from the statement of Lemma~\ref{lem:short_time_perturbations}. Note that by~\eqref{equ:time_divisibility} we have that 
 \begin{equation} \label{equ:estimate_on_J_longtime}
  J \sim \frac{\|u\|_{X(I)}^{\frac{4}{\ve}}}{\delta^{\frac{4}{\ve}}} \lesssim K^{\frac{4}{\ve}}.
 \end{equation}
 In the following we denote $t_{j-1} := \inf I_j$ for $j = 1, \ldots, J$. We would like to apply Lemma~\ref{lem:short_time_perturbations} on each interval $I_j$ to infer bounds on the $X(I_j)$ norm of $v-u$. To this end we have to make sure that for $j = 1, \ldots, J$ it holds that
 \begin{equation} \label{equ:smallness_requirement_F_longtime}
  \|F\|_{Y(I_j)} \leq \eta_0
 \end{equation}
 and 
 \begin{equation} \label{equ:smallness_requirement_H1_difference_longtime}
  \| v(t_{j-1}) - u(t_{j-1}) \|_{\dot{H}^1_x(\bR^4)} \leq \eta_0,
 \end{equation} 
 where $0 < \eta_0 \ll 1$ is the absolute constant from the statement of Lemma~\ref{lem:short_time_perturbations}. Below we will in particular choose $0 < \eta_1(K) \leq \eta_0$ which takes care of \eqref{equ:smallness_requirement_F_longtime}. To ensure~\eqref{equ:smallness_requirement_H1_difference_longtime} we now prove by induction that we have 
 \begin{equation} \label{equ:inductive_difference_bounds_longtime}
  \|v-u\|_{L^\infty_t \dot{H}^1_x(I_j\times\bR^4)} + \|v-u\|_{X(I_j)} \leq (2 C_0)^j \eta
 \end{equation}
 for $j = 1, \ldots, J$, if we choose $0 < \eta_1(K) \ll 1$ sufficiently small depending on the size of $K$. Note that since~\eqref{equ:smallness_requirement_H1_difference_longtime} trivially holds for $j = 1$, we obtain~\eqref{equ:inductive_difference_bounds_longtime} for the case $j=1$ from an application of~\eqref{equ:short_time_perturbation_difference_bound}. Now suppose that~\eqref{equ:inductive_difference_bounds_longtime} holds for all $1 \leq i \leq j-1$ and suppose that
 \begin{equation} \label{equ:inductive_smallness_longtime}
  (2 C_0)^{j-1} \eta \leq \eta_0,
 \end{equation}
 then we can prove that \eqref{equ:inductive_difference_bounds_longtime} also holds for $j$. By the inductive hypothesis we can apply~\eqref{equ:short_time_perturbation_difference_bound} on the interval $I_j$ and obtain that
 \begin{align*}
  \|v-u\|_{L^\infty_t \dot{H}^1_x(I_j\times\bR^4)} + \|v-u\|_{X(I_j)} &\leq C_0 \bigl( \|v(t_{j-1}) - u(t_{j-1})\|_{\dot{H}^1_x(\bR^4)} + \eta \bigr) \\
  &\leq C_0 \bigl( (2 C_0)^{j-1} \eta + \eta \bigr) \\
  &\leq (2 C_0)^j \eta,
 \end{align*}
 which yields~\eqref{equ:inductive_difference_bounds_longtime} for $j$. In order to complete the induction, we observe that in view of~\eqref{equ:estimate_on_J_longtime} it suffices to fix 
 \begin{equation*}
  \eta_1(K) := \exp \bigl( - C_1 K^{\frac{4}{\ve}} \bigr) \eta_0
 \end{equation*}
 for some large absolute constant $C_1 \gg 1$ to guarantee that \eqref{equ:inductive_smallness_longtime} holds for $j = 1, \ldots, J$. 
 
 Finally, we sum up the bounds~\eqref{equ:inductive_difference_bounds_longtime} to infer that 
 \begin{align*}
  \|v-u\|_{L^\infty_t \dot{H}^1_x(I\times\bR^4)} + \|v-u\|_{X(I)} &\leq \sum_{j=1}^J \|v-u\|_{L^\infty_t \dot{H}^1_x(I_j\times\bR^4)} + \|v-u\|_{X(I_j)} \\
  &\leq \sum_{j=1}^J (2 C_0)^j \eta \\
  &\lesssim (2 C_0)^J \eta \\
  &\lesssim \exp \bigl( C_1 K^{\frac{4}{\ve}} \bigr) \eta.
 \end{align*}
 This establishes~\eqref{equ:long_time_perturbation_difference_bound} and then \eqref{equ:long_time_perturbation_difference_bound_nonsharp} follows from the choice of $\eta_1(K)$. 
\end{proof}

We are now in a position to give the proof of Theorem~\ref{thm:scattering_conditional}.
\begin{proof}[Proof of Theorem~\ref{thm:scattering_conditional}]
 Let $v(t)$ be the unique solution to the forced cubic NLS~\eqref{equ:forced_cubic_nls_scattering} defined on its maximal time interval of existence $I_\ast = (T_-, T_+)$ satisfying the a priori energy bound~\eqref{equ:energy_hypothesis}, namely 
 \[
  M := \sup_{t \in I_\ast} \, E(v(t)) < \infty.
 \]
 By Proposition~\ref{prop:lwp_forced_cubic_nls} it suffices to prove that $\|v\|_{X(I_\ast)} < \infty$ in order to infer that $v(t)$ exists globally in time and scatters as $t \to \pm \infty$. Moreover, by time reversal symmetry it is enough to argue forward in time. 
 
 To this end we partition the maximal forward interval of existence $[0, T_+)$ into $J \equiv J(M, \|F\|_{Y(\bR)})$ consecutive intervals $I_j$ with disjoint interiors such that 
 \begin{equation*}
  \|F\|_{Y(I_j)} = \eta_1(K(M)),
 \end{equation*}
 where $\eta_1(K(M))$ is the small constant from the statement of Lemma~\ref{lem:long_time_perturbations} and $K(\cdot)$ is the non-decreasing function from the statement of Lemma~\ref{lem:standard_cubic}. Note that by Lemma~\ref{lem:properties_XY}(ii), the necessary number $J$ of such intervals $I_j$ is bounded from above by
 \[
  J \lesssim \frac{\|F\|_{Y(\bR)}^{\frac{4}{\ve}}}{\eta_1(K(M))^{\frac{4}{\ve}}}. 
 \]
 In the following we use the notation $t_{j-1} := \inf I_j$ for $j = 1, \ldots, J$. On each interval $I_j$ we compare the solution $v(t)$ to the forced cubic NLS~\eqref{equ:forced_cubic_nls_scattering} with the solution $u(t)$ to the usual defocusing cubic NLS~\eqref{equ:cubic_nls_scattering_section} with initial data $u(t_{j-1}) = v(t_{j-1})$. By Lemma~\ref{lem:standard_cubic} and the a priori energy hypothesis~\eqref{equ:energy_hypothesis}, $u(t)$ in fact exists globally in time and satisfies
 \[
  \|u\|_{X(I_j)} \leq \|u\|_{X(\bR)} \leq K(M).
 \]
 By the above choice of $\eta_1(K(M))$ we are thus in the position to apply the long-time perturbation estimate~\eqref{equ:long_time_perturbation_difference_bound_nonsharp} from Lemma~\ref{lem:long_time_perturbations} to infer that 
 \begin{equation*}
  \|v\|_{X(I_j)} \leq \|u\|_{X(I_j)} + \|v-u\|_{X(I_j)} \lesssim K(M) + 1.
 \end{equation*}
 Summing up these estimates we obtain the desired bound 
 \begin{align*}
  \|v\|_{X([0,T_+))} \leq \sum_{j=1}^J \|v\|_{X(I_j)} \lesssim J \cdot (K(M) + 1) \leq C\bigl(M, \|F\|_{Y(\bR)}\bigr).
 \end{align*}
 This finishes the proof of Theorem~\ref{thm:scattering_conditional}.
\end{proof}

\section{Almost sure scattering for the defocusing cubic NLS on $\bR^4$ for radial data} \label{sec:radial_scattering}

The main result of this section is the following uniform-in-time energy bound for solutions to the forced defocusing cubic NLS
\begin{equation} \label{equ:forced_cubic_nls}
 \left\{ \begin{aligned} 
  (i \partial_t + \Delta) v &= |F+v|^2 (F+v), \\
  v(0) &= v_0 \in H^1_x(\bR^4)
 \end{aligned} \right.
\end{equation}
with a forcing term $F$ that is a solution to the linear Schr\"odinger equation $(i \partial_t + \Delta) F = 0$ and satisfies a collection of suitable space-time estimates.
\begin{proposition} \label{prop:energy_growth_bound}
 Let $v_0 \in H^1_x(\bR^4)$. Assume that $F$ is a solution to the linear Schr\"odinger equation $(i \partial_t + \Delta) F = 0$ and satisfies
 \begin{equation} \label{equ:assumptions_uniform_energy_bound}
  \begin{aligned}
   &F \in Y(\bR),  &&F \in \bigl( L^\infty_t L^2_x \cap L^\infty_t L^4_x \cap L^3_t L^6_x \bigr)(\bR \times \bR^4), \\
   &\langle x \rangle^{\frac{1}{2}} \langle \nabla \rangle F \in L^2_t L^\infty_x(\bR \times \bR^4),  &&\nabla F \in L^2_t L^4_x(\bR\times\bR^4). 
  \end{aligned}
 \end{equation}
 Let $v(t)$ be the solution to the forced defocusing cubic NLS~\eqref{equ:forced_cubic_nls} with maximal time interval of existence $I_\ast$. Then we have 
 \begin{align*}
  \sup_{t \in I_\ast} \, E(v(t)) &\leq C \exp \Bigl( C \bigl( \| F \|_{L^3_t L^6_x(\bR\times\bR^4)}^3 + \bigl\| \langle x \rangle^{\frac{1}{2}} \langle \nabla \rangle F \bigr\|_{L^2_t L^\infty_x(\bR\times\bR^4)}^2 + \| \nabla F \|_{L^2_t L^4_x(\bR\times\bR^4)}^2 \bigr) \Bigr) \times\\
  &\qquad \qquad \qquad \quad \times \bigl( E(v_0) + 1 + \|v_0\|_{L^2_x(\bR^4)}^2 + \| F \|_{L^\infty_t L^2_x(\bR\times\bR^4)}^2 + \| F \|_{L^\infty_t L^4_x(\bR\times\bR^4)}^4 \bigr)
 \end{align*}
 for some absolute constant $C \geq 1$, where 
 \begin{equation*}
  E(v(t)) := \int_{\bR^4} \frac{1}{2} |\nabla v(t)|^2 + \frac{1}{4} |v(t)|^4 \, dx.
 \end{equation*}
\end{proposition}

Before we turn to the proof of Proposition~\ref{prop:energy_growth_bound}, we first note that by combining this uniform-in-time energy bound with the conditional scattering result from Theorem~\ref{thm:scattering_conditional} and with the almost sure bounds for the free Schr\"odinger evolution from Section~\ref{sec:as_bounds_free_evolution}, we immediately obtain a proof of Theorem~\ref{thm:scattering_radial}.
\begin{proof}[Proof of Theorem~\ref{thm:scattering_radial}]
 We seek a global scattering solution $u(t)$ to \eqref{equ:cubic_nls_radial_scattering_thm} of the form 
 \[ 
  u(t) = e^{it\Delta} f^\omega + v(t).
 \]
 To this end the nonlinear component $v(t)$ must be a solution to the following forced cubic NLS 
 \begin{equation} \label{equ:scattering_proof_forced_cubic_nls}
  (i\partial_t + \Delta) v = |e^{it\Delta} f^\omega + v|^2 (e^{it\Delta} f^\omega + v)
 \end{equation}
 with zero initial data $v(0) = 0$. By the assumptions on the function $f \in H^s_x(\bR^4)$ and the almost sure bounds from Section~\ref{sec:as_bounds_free_evolution}, the forcing term $e^{it\Delta} f^\omega$ in~\eqref{equ:scattering_proof_forced_cubic_nls} satisfies the space-time bounds~\eqref{equ:assumptions_uniform_energy_bound} for almost every $\omega \in \Omega$. Thus, by the local existence result from Proposition~\ref{prop:lwp_forced_cubic_nls} and by Proposition~\ref{prop:energy_growth_bound}, for almost every $\omega \in \Omega$ the forced cubic NLS~\eqref{equ:scattering_proof_forced_cubic_nls} has a solution whose energy is uniformly bounded on its maximal interval of existence and hence, by Theorem~\ref{thm:scattering_conditional}, exists globally in time and scatters. This finishes the proof of Theorem~\ref{thm:scattering_radial}.
\end{proof}

An important ingredient for the proof of the uniform-in-time energy estimate from Proposition~\ref{prop:energy_growth_bound} is the following approximate Morawetz estimate for the forced defocusing cubic NLS~\eqref{equ:forced_cubic_nls}.
\begin{proposition}
 Let $v \colon I \times \bR^4 \to \bC$ be a solution to the forced defocusing cubic NLS~\eqref{equ:forced_cubic_nls} on a time interval~$I$. Then we have 
 \begin{equation} \label{equ:nls_morawetz}
  \begin{aligned}
   \int_I \int_{\bR^4} \frac{|v|^4}{|x|} \, dx \, dt &\lesssim \| v \|_{L^\infty_t \dot{H}^1_x(I\times\bR^4)} \| v \|_{L^\infty_t L^2_x(I\times\bR^4)} \\
   &\quad \quad + \| v \|_{L^\infty_t \dot{H}^1_x(I\times\bR^4)} \bigl\| |F+v|^2 (F+v) - |v|^2 v \bigr\|_{L^1_t L^2_x(I\times\bR^4)}.
  \end{aligned}
 \end{equation}
\end{proposition}
\begin{proof}
 We write the forced defocusing cubic NLS~\eqref{equ:forced_cubic_nls} as 
 \[
  (i \partial_t + \Delta) v = |v|^2 v + H, \quad \text{where } H := |F+v|^2 (F+v) - |v|^2 v.
 \]
 Given a weight $a = a(x)$, the Morawetz action 
 \[
  m(t) := 2 \, \im \int_{\bR^4} \partial_k a(x) \partial^k v(t,x) \bar{v}(t,x) \, dx
 \]
 satisfies
 \begin{align*}
  \partial_t m(t) = \int_{\bR^4} \bigl\{ - \Delta \Delta a |v|^2 + 4 \re \, \partial_j \partial_k a \partial^j \bar{v} \partial^k v + \Delta a |v|^4 + 4 \partial_k a \, \re \, \bigl( \overline{H} \partial^k v \bigr) + 2 \Delta a \, \re \, \bigl( \bar{v} H \bigr) \bigr\} \, dx,
 \end{align*}
 where we are tacitly summing over repeated indices. Using the standard Lin-Strauss Morawetz weight $a(x) := |x|$ with
 \[
  \partial_k a = \frac{x_k}{|x|}, \quad \partial_j \partial_k a = \frac{\delta_{jk}}{|x|} - \frac{x_j x_k}{|x|^3}, \quad \Delta a = \frac{3}{|x|}, \quad \Delta \Delta a = -\frac{1}{|x|^3},
 \]
 the Morawetz estimate~\eqref{equ:nls_morawetz} follows from applying the fundamental theorem of calculus, the Cauchy-Schwarz inequality and Hardy's inequality.
\end{proof}

We are now in a position to establish the proof of Proposition~\ref{prop:energy_growth_bound}.

\begin{proof}[Proof of Proposition~\ref{prop:energy_growth_bound}]
 Let $v(t)$ be the solution to the forced defocusing cubic NLS~\eqref{equ:forced_cubic_nls} with maximal time interval of existence $I_\ast$ provided by Proposition~\ref{prop:lwp_forced_cubic_nls}. Noting that the initial data $v_0 \in H^1_x(\bR^4)$ is also assumed to have finite mass, we first infer a uniform-in-time bound on the mass of $v(t)$. Since the forcing term $F$ is a solution to the linear Schr\"odinger equation, $(F+v)$ is a solution to the standard cubic NLS. We therefore have conservation of mass for $(F+v)$ and thus find for any time interval~$I$ that
 \begin{equation} \label{equ:mass_bound}
  \begin{aligned}
   \|v\|_{L^\infty_t L^2_x(I\times\bR^4)} &\leq \| F + v \|_{L^\infty_t L^2_x(I\times\bR^4)} + \|F\|_{L^\infty_t L^2_x(I\times\bR^4)} \\
   &= \|F(0) + v(0)\|_{L^2_x(\bR^4)} + \|F\|_{L^\infty_t L^2_x(I\times\bR^4)} \\
   &\lesssim \|F\|_{L^\infty_t L^2_x(I\times\bR^4)} + \|v_0\|_{L^2_x(\bR^4)}.
  \end{aligned}
 \end{equation}
  Now we define for $T > 0$ with $T \in I_\ast$ the quantities
 \begin{align*}
  A(T) &:= \sup_{t \in [0,T]} E(v(t)), \\
  B(T) &:= \int_0^T \int_{\bR^4} \frac{|v(t,x)|^4}{|x|} \, dx \, dt,
 \end{align*}
 and compute that
 \begin{align*}
  \partial_t E(v(t)) = \re \int_{\bR^4} \partial_t \overline{v} \bigl( - \Delta v + |v|^2 v \bigr) \, dx = - \re \int_{\bR^4} \partial_t \overline{v} \bigl( |F+v|^2 (F+v) - |v|^2 v \bigr) \, dx.
 \end{align*}
 Reallocating the time derivative we find
 \begin{align*}
  \partial_t E(v(t)) &= - \frac{1}{4} \partial_t \int_{\bR^4} \bigl( |F+v|^4 - |F|^4 - |v|^4 \bigr) \, dx + \re \int_{\bR^4} \bigl( |F+v|^2 (F+v) - |F|^2 F \bigr) \partial_t \overline{F} \, dx.
 \end{align*}
 Since $F$ is a solution to the linear Schr\"odinger equation, we may insert $\partial_t F = i \Delta F$ in the last term on the right-hand side of the previous line and obtain upon integrating by parts that
 \begin{align*}
  \partial_t E(v(t)) &= - \frac{1}{4} \partial_t \int_{\bR^4} \bigl( |F+v|^4 - |F|^4 - |v|^4 \bigr) \, dx - \im \int_{\bR^4} \nabla \bigl( |F+v|^2 (F+v) - |F|^2 F \bigr) \cdot \nabla \overline{F} \, dx.
 \end{align*}
 By the fundamental theorem of calculus we then conclude that
 \begin{equation} \label{equ:bound_on_integrated_energy_derivative}
  \begin{aligned}
   \int_0^T | \partial_t E(v(t)) | \, dt &\lesssim \bigl\| |F+v|^4 - |F|^4 - |v|^4 \bigr\|_{L^\infty_t L^1_x([0,T]\times\bR^4)} \\
   &\quad \quad + \bigl\| \nabla \bigl( |F+v|^2 (F+v) - |F|^2 F \bigr) \cdot \nabla \overline{F} \bigr\|_{L^1_t L^1_x([0,T]\times\bR^4)}.
  \end{aligned}
 \end{equation}
 By Young's inequality, for any $\delta > 0$ the first term on the right-hand side of~\eqref{equ:bound_on_integrated_energy_derivative} can be estimated by
 \begin{equation} \label{equ:bound_energy_term1}
  \begin{aligned}
   \bigl\| |F+v|^4 - |F|^4 - |v|^4 \bigr\|_{L^\infty_t L^1_x([0,T]\times\bR^4)} &\leq \delta \|v\|_{L^\infty_t L^4_x([0,T]\times\bR^4)}^4 + C_\delta \|F\|_{L^\infty_t L^4_x([0,T]\times\bR^4)}^4 \\
   &\leq \delta A(T) + C_\delta \|F\|_{L^\infty_t L^4_x([0,T]\times\bR^4)}^4,
  \end{aligned}
 \end{equation}
 while by H\"older's inequality the second term on the right-hand side of~\eqref{equ:bound_on_integrated_energy_derivative} is bounded by the sum of the following five schematic terms (where the space-time norms are taken over $[0,T] \times \bR^4$)
 \begin{equation} \label{equ:bound_energy_term2}
  \begin{aligned}
   \bigl\| (\nabla v) F^2 (\nabla F) \bigr\|_{L^1_t L^1_x} &\leq \| \nabla v \|_{L^\infty_t L^2_x} \| F \|_{L^\infty_t L^4_x} \| F \|_{L^2_t L^\infty_x} \| \nabla F \|_{L^2_t L^4_x} \\
   &\lesssim A(T)^{\frac{1}{2}} \| F \|_{L^\infty_t L^4_x} \| F \|_{L^2_t L^\infty_x} \| \nabla F \|_{L^2_t L^4_x}, \\
   \bigl\| v F (\nabla F)^2 \bigr\|_{L^1_t L^1_x} &\leq \| v \|_{L^\infty_t L^4_x} \| F \|_{L^\infty_t L^4_x} \| \nabla F \|_{L^2_t L^4_x}^2 \\
   &\lesssim A(T)^{\frac{1}{4}} \| F \|_{L^\infty_t L^4_x} \| \nabla F \|_{L^2_t L^4_x}^2, \\
   \bigl\| (\nabla v) v F (\nabla F) \bigr\|_{L^1_t L^1_x} &\leq \| \nabla v \|_{L^\infty_t L^2_x} \| v \|_{L^\infty_t L^4_x} \| F \|_{L^2_t L^\infty_x} \| \nabla F \|_{L^2_t L^4_x} \\
   &\lesssim A(T)^{\frac{3}{4}} \| F \|_{L^2_t L^\infty_x} \| \nabla F \|_{L^2_t L^4_x}, \\
   \bigl\| v^2 (\nabla F)^2 \bigr\|_{L^1_t L^1_x} &\leq \| v \|_{L^\infty_t L^4_x}^2 \| \nabla F \|_{L^2_t L^4_x}^2 \\
   &\lesssim A(T)^{\frac{1}{2}} \| \nabla F \|_{L^2_t L^4_x}, \\
   \bigl\| v^2 \nabla v \nabla F \bigr\|_{L^1_t L^1_x} &\leq \bigl\| |x|^{-\frac{1}{2}} v^2 \bigr\|_{L^2_t L^2_x} \| \nabla v \|_{L^\infty_t L^2_x} \bigl\| |x|^{+\frac{1}{2}} \nabla F \bigr\|_{L^2_t L^\infty_x} \\
   &\lesssim B(T)^{\frac{1}{2}} A(T)^{\frac{1}{2}} \bigl\| |x|^{+\frac{1}{2}} \nabla F \bigr\|_{L^2_t L^\infty_x}.
  \end{aligned}
 \end{equation}
 Moreover, from the Morawetz estimate~\eqref{equ:nls_morawetz} and the mass bound~\eqref{equ:mass_bound} we obtain that
 \begin{equation} \label{equ:bound_morawetz_quantity}
  \begin{aligned}
   B(T) &\lesssim \| v \|_{L^\infty_t \dot{H}^1_x([0,T]\times\bR^4)} \| v \|_{L^\infty_t L^2_x([0,T]\times\bR^4)} \\
   &\quad + \| v \|_{L^\infty_t \dot{H}^1_x([0,T]\times\bR^4)} \bigl\| |F|^3 + |v|^2 |F| \bigr\|_{L^1_t L^2_x([0,T]\times\bR^4)}  \\
   &\lesssim A(T) + \|v\|_{L^\infty_t L^2_x([0,T]\times\bR^4)}^2 + A(T)^{\frac{1}{2}} \|F\|_{L^3_t L^6_x([0,T]\times\bR^4)}^3 \\ 
   &\quad + A(T)^{\frac{1}{2}} \bigl\| |x|^{-\frac{1}{2}} |v|^2 \bigr\|_{L^2_t L^2_x([0,T]\times\bR^4)} \bigl\| |x|^{\frac{1}{2}} F \bigr\|_{L^2_t L^\infty_x([0,T]\times\bR^4)} \\
   &\lesssim A(T) + \|v_0\|_{L^2_x(\bR^4)}^2 + \|F\|_{L^\infty_t L^2_x(\bR\times\bR^4)}^2 + A(T)^{\frac{1}{2}} \|F\|_{L^3_t L^6_x([0,T]\times\bR^4)}^3 \\
   &\quad + A(T)^{\frac{1}{2}} B(T)^{\frac{1}{2}} \bigl\| |x|^{\frac{1}{2}} F \bigr\|_{L^2_t L^\infty_x([0,T]\times\bR^4)}.
  \end{aligned}
 \end{equation}
 We collect all divisible space-time norms of the forcing term $F$ that have appeared in the previous estimates in the following norm 
 \[
  \|F\|_{Z([0,T])} := \|F\|_{L^3_t L^6_x([0,T]\times\bR^4)} + \bigl\| \langle x \rangle^{\frac{1}{2}} \langle \nabla \rangle F \bigr\|_{L^2_t L^\infty_x([0,T]\times\bR^4)} + \| \nabla F \|_{L^2_t L^4_x([0,T]\times\bR^4)}.
 \]
 By the fundamental theorem of calculus we have that
 \[
  A(T) \leq E(v(0)) + \int_0^T |\partial_t E(v(t))| \, dt, 
 \]
 hence we may infer from \eqref{equ:bound_on_integrated_energy_derivative}--\eqref{equ:bound_energy_term2} that for sufficiently small $\delta > 0$,
 \begin{equation*}
  \begin{aligned}
   A(T) &\lesssim E(v(0)) + \|F\|_{L^\infty_t L^4_x(\bR\times\bR^4)}^4 + A(T)^{\frac{1}{2}} \| F \|_{L^\infty_t L^4_x(\bR\times\bR^4)} \| F \|_{Z([0,T])}^2 \\
   &\quad + A(T)^{\frac{1}{4}} \| F \|_{L^\infty_t L^4_x(\bR\times\bR^4)} \|F\|_{Z([0,T])}^2 + A(T)^{\frac{3}{4}} \| F \|_{Z([0,T])}^2 \\
   &\quad + A(T)^{\frac{1}{2}} \|F\|_{Z([0,T])} + A(T)^{\frac{1}{2}} B(T)^{\frac{1}{2}} \|F\|_{Z([0,T])}.
  \end{aligned}
 \end{equation*}
 Moreover, from the estimate~\eqref{equ:bound_morawetz_quantity} we have 
 \begin{equation*}
  \begin{aligned}
   B(T) &\lesssim A(T) + \|v_0\|_{L^2_x(\bR^4)}^2 + \|F\|_{L^\infty_t L^2_x(\bR\times\bR^4)}^2 + A(T)^{\frac{1}{2}} \|F\|_{Z([0,T])}^3 + A(T)^{\frac{1}{2}} B(T)^{\frac{1}{2}} \|F\|_{Z([0,T])}.
  \end{aligned}
 \end{equation*}
 Thus, by a continuity argument we may now conclude that there exists a sufficiently small absolute constant $0 < \eta \ll 1$ such that if
 \[
  \|F\|_{Z([0,T])} \leq \eta,
 \]
 then it holds that 
 \[
  A(T) \lesssim E(v(0)) + 1 + \|v_0\|_{L^2_x(\bR^4)}^2 + \|F\|_{L^\infty_t L^2_x(\bR\times\bR^4)}^2 + \|F\|_{L^\infty_t L^4_x(\bR\times\bR^4)}^4.
 \]
 By divisibility of the $Z(\bR)$ norm and by time-reversibility, we iterate this argument finitely many times to conclude the desired uniform-in-time energy bound
 \begin{align*}
  \sup_{t \in I_\ast} \, E(v(t)) &\leq C \exp \Bigl( C \bigl( \| F \|_{L^3_t L^6_x(\bR\times\bR^4)}^3 + \bigl\| \langle x \rangle^{\frac{1}{2}} \langle \nabla \rangle F \bigr\|_{L^2_t L^\infty_x(\bR\times\bR^4)}^2 + \| \nabla F \|_{L^2_t L^4_x(\bR\times\bR^4)}^2 \bigr) \Bigr) \times\\
  &\qquad \qquad \qquad \quad \times \bigl( E(v_0) + 1 + \|v_0\|_{L^2_x(\bR^4)}^2 + \| F \|_{L^\infty_t L^2_x(\bR\times\bR^4)}^2 + \| F \|_{L^\infty_t L^4_x(\bR\times\bR^4)}^4 \bigr)
 \end{align*}
 for some absolute constant $C \geq 1$. 
\end{proof}

\appendix 

\section{Proof of Theorem~\protect{\ref{thm:scattering_nlw_radial}}}

 Here we sketch the proof of Theorem~\ref{thm:scattering_nlw_radial} for the defocusing energy-critical nonlinear wave equation in four space dimensions
 \begin{equation}  \label{equ:cubic_nlw}
  \left\{ \begin{aligned}
   -\partial_t^2 u + \Delta u &= u^3 \text{ on } \bR \times \bR^4, \\
   (u, \partial_t u)|_{t=0} &= (f_0^\omega, f_1^\omega) \in H^s_x(\bR^4) \times H^{s-1}_x(\bR^4).
  \end{aligned} \right.
 \end{equation}
 We seek a global, scattering solution to \eqref{equ:cubic_nlw} of the form 
 \[
  u(t) = S(t)(f_0^\omega, f_1^\omega) + v(t).
 \]
 To this end we pass to the study of the more general forced cubic wave equation for the nonlinear component~$v(t)$, namely
 \begin{equation}  \label{equ:forced_cubic_nlw}
  \left\{ \begin{aligned}
   -\partial_t^2 v + \Delta v &= (F+v)^3 \text{ on } \bR \times \bR^4, \\
   (v, \partial_t v)|_{t=0} &= (v_0, v_1) \in \dot{H}^1_x(\bR^4) \times L^2_x(\bR^4)
  \end{aligned} \right.
 \end{equation}
 for forcing terms $F \colon \bR \times \bR^4 \to \bR$ satisfying suitable space-time integrability properties. The proof of Theorem \ref{thm:scattering_nlw_radial} will follow from new, improved almost sure bounds for the free wave evolution of randomized radially symmetric data. Indeed, we recall statements of two of the main theorems from the authors' previous work~\cite{DLuM}.

 \begin{theorem}[Theorem 1.1, \cite{DLuM}]  \label{thm:conditional_scattering_nlw}
  There exists a non-decreasing function $K \colon [0,\infty) \to [0, \infty)$ with the following property. Let $(v_0, v_1) \in \dot{H}^1_x(\bR^4) \times L^2_x(\bR^4)$ and $F \in L^3_t L^{6}_x(\bR \times \bR^4)$. Let $v(t)$ be a solution to \eqref{equ:forced_cubic_nlw} defined on its maximal time interval of existence $I_\ast$. Suppose in addition that
  \begin{equation} 
   M := \sup_{t \in I_\ast} \, E(v(t)) < \infty,
  \end{equation}
  where
  \[
   E(v(t)) = \int_{\bR^4} \frac{1}{2} |\nabla_x v(t)|^2 + \frac{1}{2} |\partial_t v(t)|^2 + \frac{1}{4} |v(t)|^4 \, dx.
  \]
  Then $I_\ast = \bR$, that is $v(t)$ is globally defined, and it holds that
  \begin{equation} 
   \|v\|_{L^3_t L^{6}_x(\bR\times\bR^4)} \leq C \|F\|_{L^3_t L^{6}_x(\bR\times\bR^4)} \bigl( K(M) + 1 \bigr) \exp \bigl( C \,K(M)^3 \bigr)
  \end{equation}
  for some absolute constant $C > 0$. In particular, the solution $v(t)$ scatters to free waves as $t \to \pm \infty$ in the sense that there exist states $(v_0^{\pm}, v_1^{\pm}) \in \dot{H}^1_x(\bR^4) \times L^2_x(\bR^4)$ such that
  \[
   \lim_{t \to \pm \infty} \, \bigl\| \nabla_{t,x} \bigl( v(t) - S(t)(v_0^\pm, v_1^\pm) \bigr) \bigr\|_{L^2_x(\bR^4)} = 0.
  \]
  \end{theorem}

 \begin{theorem}[Theorem 1.2, \cite{DLuM}] \label{thm:energy_bound_nlw}
  Let $(v_0, v_1) \in \dot{H}^1_x(\bR^4) \times L^2_x(\bR^4)$. Assume that
  \begin{align} 
   F \in L^3_t L^{6}_x(\bR \times \bR^4) \qquad \textup{and}  \qquad |x|^{\frac{1}{2}} F \in L^2_t L^{\infty}_x(\bR \times \bR^4).
  \end{align}
  Let $v(t)$ be a solution to \eqref{equ:forced_cubic_nlw} defined on its maximal time interval of existence $I_\ast$. Then we have 
  \[
   \sup_{t \in I_\ast} \, E(v(t)) \leq C \exp \Bigl( C \bigl( \|F\|_{L^3_t L^6_x(\bR\times\bR^4)}^3 + \bigl\| |x|^{\frac{1}{2}} F \bigr\|_{L^2_t L^\infty_x(\bR\times\bR^4)}^2 \bigr) \Bigr) ( E(v(0)) + 1 )
  \]
  for some absolute constant $C > 0$. It therefore holds that $I_\ast = \bR$, that is $v(t)$ exists globally in time, and the solution $v(t)$ scatters to free waves as $t \to \pm \infty$.
 \end{theorem}

 Thus, in light of Theorem~\ref{thm:conditional_scattering_nlw} and Theorem~\ref{thm:energy_bound_nlw}, the proof of Theorem~\ref{thm:scattering_nlw_radial} reduces to proving that for any $0 < s < 1$ and any radially symmetric $f \in H^s_x(\bR^4)$ we have that almost surely,
 \begin{align} \label{equ:prob_apriori}
  \bigl\| e^{ \pm i t|\nabla|} f^\omega \bigr\|_{L^3_t L^6_x(\bR\times\bR^4)}^3 + \bigl\| |x|^{\frac{1}{2}} e^{ \pm i t|\nabla|} f^\omega \bigr\|_{L^2_t L^\infty_x(\bR\times\bR^4)} < \infty.
 \end{align}
 These two almost sure bounds are established in the next two propositions.

\begin{proposition} \label{prop:weighted_L2Linfty_wave}
 Let $0 < s < 1$ and $0 \leq \alpha < 1$. Let $f \in H^s_x(\bR^4)$ be radially symmetric and denote by $f^\omega$ the randomization of $f$ as defined in~\eqref{equ:wave_randomization}. 
 Then there exist absolute constants $C > 0$ and $c > 0$ such that for any $\lambda > 0$ it holds that
 \begin{equation}
  \bP \Bigl( \Bigl\{ \omega \in \Omega : \bigl\| \langle x \rangle^\alpha e^{\pm i t |\nabla|} f^\omega \bigr\|_{L^2_t L^\infty_x(\bR \times \bR^4)} > \lambda  \Bigr\} \Bigr) \leq C \exp \Bigl( - c \lambda^2 \|f\|_{H^s_x(\bR^4)}^{-2} \Bigr).
 \end{equation}
 In particular, we have for almost every $\omega \in \Omega$ that
 \begin{equation}
  \bigl\| \langle x \rangle^\alpha e^{\pm i t |\nabla|} f^\omega \bigr\|_{L^2_t L^\infty_x(\bR \times \bR^4)} < \infty.
 \end{equation}
\end{proposition}

The proof of Proposition~\ref{prop:weighted_L2Linfty_wave} is essentially a verbatim copy of the proof of the weighted global-in-time $L^2_t L^\infty_x$ almost sure bound for the derivative of the Schr\"odinger evolution of randomized radial data from Proposition~\ref{prop:weighted_nabla_L2Linfty_schroedinger}, the only difference being that we employ local energy decay estimates for the free wave evolution instead of the related local smoothing estimates for the Schr\"odinger evolution. We recall the precise local energy decay estimates that we use in the next lemma and correspondingly leave the details of the proof of Proposition~\ref{prop:weighted_L2Linfty_wave} to the reader.

\begin{lemma}[\protect{Local energy decay; \cite[(2.16)]{KPV_local_energy}, \cite{Smith_Sogge}, \cite{Keel_Smith_Sogge}, \cite[Appendix]{Sterbenz}}] \label{prop:local_energy_decay}
 Let $f \in L^2_x(\bR^4)$. Then it holds that
 \begin{equation} 
  \sup_{R > 0} \, R^{-\frac{1}{2}} \bigl\| e^{\pm i t |\nabla|} f \bigr\|_{L^2_t L^2_x(\bR \times \{ |x| \leq R \})} \lesssim \| f \|_{L^2_x(\bR^4)}.
 \end{equation}
\end{lemma}

\begin{proposition} 
 Let $0 < s < 1$. Let $f \in H^s_x(\bR^4)$ be radially symmetric and denote by $f^\omega$ the randomization of $f$ as defined in~\eqref{equ:wave_randomization}. 
 Then there exist absolute constants $C > 0$ and $c > 0$ such that for any $\lambda > 0$ it holds that
 \begin{equation}
  \bP \Bigl( \Bigl\{ \omega \in \Omega : \bigl\| e^{\pm i t |\nabla|} f^\omega \bigr\|_{L^3_t L^6_x(\bR \times \bR^4)} > \lambda  \Bigr\} \Bigr) \leq C \exp \Bigl( - c \lambda^2 \|f\|_{H^s_x(\bR^4)}^{-2} \Bigr).
 \end{equation}
 In particular, we have for almost every $\omega \in \Omega$ that
 \begin{equation}
  \bigl\| e^{\pm i t |\nabla|} f^\omega \bigr\|_{L^3_t L^6_x(\bR \times \bR^4)} < \infty.
 \end{equation}
\end{proposition}
\begin{proof}
 In what follows, all space-time norms are taken over $\bR \times \bR^4$. By H\"older's inequality we have
 \begin{align*}
  \bigl\| e^{\pm i t |\nabla|} f^\omega \bigr\|_{L^3_t L^6_x} \leq \bigl\| e^{\pm i t |\nabla|} f^\omega \bigr\|_{L^\infty_t L^2_x}^{\frac{1}{3}} \bigl\| e^{\pm i t |\nabla|} f^\omega \bigr\|_{L^2_t L^\infty_x}^{\frac{2}{3}} = \bigl\| f^\omega \bigr\|_{L^2_x}^{\frac{1}{3}} \bigl\| e^{\pm i t |\nabla|} f^\omega \bigr\|_{L^2_t L^\infty_x}^{\frac{2}{3}}.
 \end{align*}
 A simple application of Minkowski's inequality and of the large deviation estimate from Lemma~\ref{lem:large_deviation_estimate} yields for any $p \geq 2$ that
 \[
  \bigl\| e^{\pm i t |\nabla|} f^\omega \bigr\|_{L^p_\omega L^2_x} \lesssim \sqrt{p} \, \| f \|_{L^2_x},
 \]
 while in the proof of Proposition~\ref{prop:weighted_L2Linfty_wave} we establish for all sufficiently large $p < \infty$ that
 \[
  \bigl\| e^{\pm i t |\nabla|} f^\omega \bigr\|_{L^p_\omega L^2_t L^\infty_x} \lesssim \sqrt{p} \, \| f \|_{H^s_x}.
 \]
 Thus, we conclude for all sufficiently large $p < \infty$ that
 \[
  \bigl\| e^{\pm i t |\nabla|} f^\omega \bigr\|_{L^p_\omega L^3_t L^6_x} \lesssim \sqrt{p} \, \| f \|_{H^s_x}
 \]
 and the claim follows from Lemma~\ref{lem:probability_estimate}.
\end{proof}

\bibliographystyle{myamsplain}
\bibliography{references}

\end{document}